\documentclass[12pt,reqno]{amsart}
\usepackage{mathrsfs}
\usepackage{geometry}
\usepackage{titletoc}
\usepackage{mathpazo}
\usepackage{amsmath}
\usepackage{amssymb} 
\usepackage{enumerate} 
\usepackage{mathtools} 
\usepackage[table]{xcolor} 
\usepackage[all]{xy} 
\usepackage{tikz} 
\usepackage{tikz-cd}
\usepackage{indentfirst} 
\usepackage{setspace} 

\usepackage[colorlinks,linkcolor=red,anchorcolor=green,citecolor=blue]{hyperref} 
\hypersetup{linktocpage = true} 

\usepackage{rotating} 
\allowdisplaybreaks[4]

\usepackage{indentfirst}
\usepackage{amsmath}
\usepackage{cases}

\usepackage{ytableau} 
\usepackage{longtable} 
\newcolumntype{M}[1]{>{\centering\arraybackslash}m{#1}} 

\usepackage{mathrsfs}
\DeclareFontFamily{OMS}{rsfs}{\skewchar\font'60}
\DeclareFontShape{OMS}{rsfs}{m}{n}{<-5>rsfs5 <5-7>rsfs7 <7->rsfs10 }{}
\DeclareSymbolFont{rsfs}{OMS}{rsfs}{m}{n}
\DeclareSymbolFontAlphabet{\scr}{rsfs}
\DeclareSymbolFontAlphabet{\scr}{rsfs}

\usepackage[T1]{fontenc}

\newcommand\cD{{\mathcal D}}
\newcommand\cE{{\mathcal E}}
\newcommand\cF{{\mathcal F}}

\newcommand\cH{{\mathcal H}}

\newcommand\cL{{\mathcal L}}

\newcommand\cR{{\mathcal R}}

\newcommand\cT{{\mathcal T}}

\newcommand{\tcF}{\tilde{\cF}}

\newcommand{\id}{{\rm id}}

\newcommand{\Aut}{\rm Aut}

\theoremstyle{plain}
\newtheorem{thm}{Theorem}[section]
\newtheorem{lemma}[thm]{Lemma}
\newtheorem{prop}[thm]{Proposition}
\newtheorem{cor}[thm]{Corollary}
\newtheorem{defn}[thm]{Definition}

\theoremstyle{definition}
\newtheorem{example}[thm]{Example}
\newtheorem{remark}[thm]{Remark}

\newcommand{\btheorem}{\begin{thm}}
	\newcommand{\etheorem}{\end{thm}}
\newcommand{\bproposition}{\begin{prop}}
	\newcommand{\eproposition}{\end{prop}}
\newcommand{\bdefinition}{\begin{defn}}
	\newcommand{\edefinition}{\end{defn}}
\newcommand{\bcorollary}{\begin{cor}}
	\newcommand{\ecorollary}{\end{cor}}
\newcommand{\bproof}{\begin{proof}}
	\newcommand{\eproof}{\end{proof}}
\newcommand{\bremark}{\begin{remark}}
	\newcommand{\eremark}{\end{remark}}
\newcommand{\eexample}{\end{example}}
\newcommand{\bexample}{\begin{example}}

\newcommand{\elemma}{\end{lemma}}
\newcommand{\blemma}{\begin{lemma}}

\newcommand{\sq}{\sqrt{-1}}

\newcommand{\p}{\partial}

\renewcommand{\bar}{\overline}
\newcommand{\eps}{\varepsilon}

\newcommand{\beq}{\begin{equation}}
\newcommand{\eeq}{\end{equation}}
\newcommand{\ee}{\end{eqnarray*}}
\newcommand{\be}{\begin{eqnarray*}}

\renewcommand{\tilde}{\widetilde}

\renewcommand{\>}{\rightarrow}
\newcommand{\ddc}{{\rm dd^c}}

\newcommand{\CC}{\mathbb {C}}

\newcommand{\EE}{{\mathbb E}}
\newcommand{\FF}{{\mathbb F}}
\newcommand{\GG}{{\mathbb G}}
\newcommand{\HH}{{\mathbb H}}
\newcommand{\II}{{\mathbb I}}
\newcommand{\JJ}{{\mathbb J}}

\newcommand{\MM}{{\mathbb M}}
\newcommand{\NN}{{\mathbb N}}

\newcommand{\RR}{{\mathbb R}}
\newcommand{\TT}{{\mathbb T}}

\newcommand{\db}{\overline{\partial}}

\newcommand{\ddbar}{\sqrt{-1} \partial \overline{\partial}}

\newcommand{\tr}{\mathrm{tr}}

\newcommand{\Ric}{\mathrm{Ric}}
\renewcommand{\d}{\mathrm{d}}

\newcommand{\vv}{\mathrm v}

\newcommand{\ww}{\mathrm w}
\newcommand{\ext}{\mathrm{ext}}
\newcommand{\Ent}{\mathrm{Ent}}
\newcommand{\Psh}{\mathrm{Psh}}

\newcommand{\Vol}{\mathrm{Vol}}
\newcommand{\aut}{\mathrm{aut}}

\numberwithin{equation}{section} 

\makeatletter

\ifnum\@ptsize=0  \fi \ifnum\@ptsize=2  \fi

\usepackage{fancyhdr}
\title{On the existence of weighted-cscK metrics}

\author{Jiyuan Han and Yaxiong Liu}

\address{Jiyuan Han, Institute for Theoretical Sciences,
Westlake University,
No.600 Dunyu Road, 
Hangzhou, 310030, China}
\email{hanjiyuan@westlake.edu.cn}

\address{Yaxiong Liu, Department of Mathematics,
University of Maryland,
4176 Campus Dr,
College Park, MD 20742, USA}
\email{yxliu238@umd.edu}

\begin{document}

\begin{abstract}
    In this paper, we prove that on a smooth K\"ahler manifold, the $\GG$-coercivity of the weighted Mabuchi functional implies the existence of the $(\vv,\ww)$-weighted-cscK (extremal) metric with $\vv$ log-concave (firstly studied in \cite{Lah19}), e.g, extremal metrics, K\"ahler--Ricci solitons, $\mu$-cscK metrics. 
\end{abstract}

\dedicatory{Dedicated to Professor Akito Futaki for his 70th Birthday}

	\maketitle
	
\small{\setcounter{tocdepth}{2} \tableofcontents
		\dottedcontents{section}[0.8cm]{}{1.8em}{1555pt}
		\dottedcontents{subsection}[2.0cm]{}{3em}{1555pt}
}

\section{Introduction}
In this paper, we apply the variational approach and continuity method to study the weighted-cscK (extremal) metrics, the study of which is initiated in papers \cite{Lah19, Lah23, AJL23}.
The weighted-cscK (extremal) metric is a natural generalization of the cscK (extremal) metric. The (generalized) Yau--Tian--Donaldson conjecture predicts an equivalence between the existence of cscK metric(s) and the $\GG$-uniform K-stability condition. By the work of many authors (see an incomplete list: \cite{Tian97, DR17, BDL20, CC21a, CC21b, Li22a, BJ23}), much progress has been made, and it is known that the existence of cscK metric(s) is equivalent to the $\GG$-coercivity of the Mabuchi functional, and can be implied by the $\GG$-uniformly K-stability for models.

The category of weighted-cscK metrics includes a large class of important metrics, such as K\"ahler--Einstein metrics, cscK metrics, K\"ahler--Ricci solitons, Mabuchi solitons, $\mu$-cscK metrics, etc. Indeed, the log-concavity assumption of our main theorem \ref{main-thm} is satisfied in all these examples. It provides a universal framework to treat these different canonical metrics. This viewpoint has some interesting applications, e.g, in \cite{Li21}, the author finds the Ricci-flat K\"ahler cone metric on an irregular cone by transforming it (within the weighted-cscK category) to a $g$-soliton on a quasi-regular cone.
Because of its importance, it is attempting to generalize the Yau--Tian--Donaldson type theory to the weighted-cscK case which would give an algebraic criterion for the resolution of the equation. 

Let $(X, \omega)$ be a compact K\"ahler manifold.
Let $T$ be an $\ell$-dimensional real torus generated by Hamiltonian vector fields of $(X, \omega)$. 
The complexified torus $T_\CC$ acts holomorphically on $(X,\omega)$. 
The Lie algebra $\mathfrak{t} := \mathrm{Lie} (T)$ is identified with $N_\RR \simeq \mathbb{R}^\ell$.
Similarly, we identify the dual $\mathfrak{t}^*$ with $M_\RR \simeq \mathbb{R}^\ell$ with the corresponding dual real coordinates, which will be denote by 
$y:=(y_1,\cdots,y_{\ell})$. 

Let $\{\xi^{\alpha}\}$ be a standard basis of $\mathfrak{t}$, and let 
$m^{\xi^\alpha}_\phi\in C^{\infty}(X)$ be the associated Hamiltonian function for a K\"ahler form $\omega_{\phi}\in [\omega]$. 
We have
\begin{equation*}
	\iota_{\xi^{\alpha}}\omega_\phi
	=\dfrac{\sqrt{-1}}{2\pi}\bar{\partial} m^{\xi^\alpha}_\phi.
\end{equation*}
By \cite{Ati82}, the image of 
 $m_{\phi}:X\rightarrow \mathfrak{t}^*=\mathbb{R}^\ell$
is a convex polytope $P$, which is called the \textit{associated moment polytope}. It is explicitly given by
\begin{equation*}
	m_{\phi}(x)
	:=(m_\phi^{\xi^1}(x),\cdots,m_\phi^{\xi^\ell}(x)), \quad
	P:=m_{\phi}(X).
\end{equation*}

We use $\alpha$, $\beta$ to represent the index of the Lie algebra $\mathfrak{t}$, and $i,j$ to represent the index of base manifold $X$. 
Let $y=(y_\alpha)\in P$ denote for the coordinate. For any smooth function $\vv$ on $P$, we will always use $\vv_{,\alpha}$ to denote the derivative of $\vv$ over $y_\alpha$.

\begin{defn} 
    We define the Lie algebra
    \begin{align*}
        {{\aut}}_T(X) = \{\xi\in {{\aut}(X)}: [\xi, c]=0, \forall c\in \mathfrak{t}_\CC\}.
    \end{align*}
    Let ${\Aut}_T(X)$ be the connected subgroup of $\Aut(X)$ that is generated by ${\aut}_T(X)$. 
    Let $\GG = K_\CC \subset {\Aut}_T(X)$ be a reductive Lie group, and $K$ contains a maximal torus of ${\Aut}_T(X)$.
    Denote $\TT$ as the center of $\GG$. Note that, since $T$ is a subgroup of the center of ${\Aut}_T(X)$ and $K$ contains a maximal torus, we have $T\subset K$.
\end{defn}

In \cite{Lah19}, the author considers the following weighted-cscK equation
\begin{align}
    \label{weighted-csck}
    {S_{\vv}(\phi)} = {\ww(m_\phi)}\cdot \ell_{\rm ext}(m_\phi),
\end{align}
where  $\vv,\ww\in C^\infty(P,\RR_{>0})$, $\ell_{\ext}$ is an affine function on $P$, and the weighted-scalar curvature is given by
\begin{align*}
    S_\vv(\phi)
:=\vv(m_\phi)S(\phi)-2\vv_{,\alpha}(m_\phi)\Delta_\phi(m_\phi^{\xi^\alpha})-\vv_{,\alpha\beta}(m_\phi)\langle\xi^\alpha,\xi^\beta\rangle_\phi.
\end{align*}
Note that when $\vv=\ww=const$, the solution of \eqref{weighted-csck} is the extremal metric.

Equation \eqref{weighted-csck} can be equivalently written as a system of second-order PDEs:
 \begin{subnumcases}{}
		\vv(m_{\phi})(\omega+\ddc \phi)^n
		=e^{F}\omega^n, 
		\label{couple system1 weighted cscK intro} \\
		\left(\Delta_{\phi}-\frac{\vv_{,\alpha}(m_{\phi})}{\vv(m_{\phi})} J\xi^\alpha\right)F
		=-\ell_{\ext}(m_{\phi})\frac{\ww(m_{\phi})}{\vv(m_{\phi})} 
		+\frac{\tr_{\vv,\phi}\Ric(\omega)}{\vv(m_\phi)}, 
		\label{couple system2 weighted cscK intro}
	\end{subnumcases}
 where for a $(1,1)$-form $\eta$, 
$\tr_{\vv,\phi}(\eta)$ is defined by
\begin{equation*}
\tr_{\vv,\phi}(\eta)
:=\vv(m_\phi)\tr_{\phi}\eta+\langle\d\vv(m_\phi),m_\eta\rangle,
\end{equation*}
where the definition of $\langle\d\vv(m_\phi),m_\eta\rangle$ can be found in Definition \ref{weighted_fn_def}.

From the variational point of view, the equation \eqref{weighted-csck} can also be considered as the Euler-Lagrangian of the weighted Mabuchi functional $\MM_{\vv,\ww\cdot\ell_{\ext}}$ (see Definition \ref{def of weighted Mabuchi}). 
In \cite{Lah23}, the author shows that the weighted-cscK metric is a minimizer of $\MM_{\vv,\ww\cdot\ell_{\ext}}$ and such metric is unique up to automorphism. 
It is proved in \cite{AJL23} that the existence of the solution of \eqref{weighted-csck} implies the $\GG$-coercivity of $\MM_{\vv,\ww\cdot\ell_{\ext}}$. 
In this article, we aim to prove the inverse direction and conclude the following theorem:
\begin{thm}
    \label{main-thm}
    Assume $\vv$ is log-concave.
    The following statements are equivalent: \\
    (1). Up to the automorphism group ${\Aut}_T(X)$, there exists a unique solution of the weighted-cscK (extremal) equation \eqref{weighted-csck}.\\
    (2). The weighted Mabuchi functional for weighted-cscK (extremal) equation is $\mathbb{G}$-coercive, i.e, 
    \begin{align*}
        \MM_{\vv,\ww\cdot\ell_{\ext}}(\phi) \geq \delta \cdot \JJ_{\vv, \TT}(\phi) - C,
    \end{align*}
    for some constants $\delta>0$, $C>0$, and any $\phi\in\cE^1_{K}(X, \omega)$.
\end{thm}

To show (2) implies (1), we consider the following weighted version of the continuity path of Chen \cite{Ch18} for weighted extremal metrics,
\begin{align}
\label{Chen-continuous path}
t\left(\frac{S_{\vv}(\phi)}{\ww(m_\phi)}-\ell_{\rm ext}(m_\phi)\right)-(1-t)\left(\frac{\tr_{\vv,\phi}(\theta)}{\ww(m_\phi)}-\underline{\theta}\right)=0,
\end{align}
where $\theta$ is a $K$-invariant smooth positive $(1,1)$-form in the same class of $\omega$ and $\underline{\theta}$ is defined by \eqref{const theta_lower_bar}, which is a constant only depending on $[\omega]$ by \cite[Lemma 2]{Lah19}.
\begin{remark}
    The idea of the continuity path \eqref{Chen-continuous path} can be seen from the variational picture. Since $-\Ric(\omega)$ may not be positive, replacing a small part of the term $-\EE_\vv^\Ric$ in the weighted Mabuchi functional by $\EE_\vv^\theta$ for some positive (1,1)-form $\theta$ will make the weighted Mabuchi functional more coercive. Intuitively, this makes \eqref{Chen-continuous path} always solvable for $t\in [0,1)$. 
    
    Such continuity method is used in \cite{CC21a,CC21b} for the resolution of the cscK problem. The idea of using continuity path to solve Monge-Amp\`ere type equations dates back to the fundamental work of \cite{Yau78}.
\end{remark}

Similar with \cite{CC21a}, the equation \eqref{Chen-continuous path} can be re-written as a coupled PDE system \begin{subnumcases}{}
		\vv(m_\phi)\omega_\phi^n
		=e^F\omega^n, 
		\label{couple system1 intro} \\
		\left(\Delta_\phi-\frac{\vv_{,\alpha}(m_\phi)}{\vv(m_\phi)} J\xi^\alpha\right)F
		=-\ell_{\ext}(m_\phi)\frac{\ww(m_\phi)}{\vv(m_\phi)}
		+(1-\frac{1}{t})\frac{\tr_{\vv,\phi}(\theta)}{\vv(m_\phi)}\nonumber \\
		\qquad\qquad\qquad\qquad\qquad\qquad 
		-(1-\frac{1}{t})\underline{\theta}\frac{\ww(m_\phi)}{\vv(m_\phi)} +\frac{\tr_{\vv,\phi}\Ric(\omega)}{\vv(m_\phi)}. 
		\label{couple system2 intro}
\end{subnumcases}
Compared with the equations in \cite{CC21a}, the first equation \eqref{couple system1 intro} is a weighted Monge-Amp\`ere equation, and the second linear equation \eqref{couple system2 intro} also contains first-order derivative term of $F$ (which is a third-order derivative of $\phi$). 
This makes the $C^0$-estimate and $C^2$-estimate of $\phi$, the gradient estimate of $F$ more challenging.

We prove a uniform $C^0$-estimate by modifying Guo-Phong's argument in \cite{guo2022}.
In addition, in the calculation of the $C^2$-estimate of $\phi$ and the gradient estimate of $F$, we need to modify Chen-Cheng's method in \cite{CC21b} (in particular, we need to choose new test functions and integral weights in order to deal with the third-order derivative terms (of $\phi$)).

\subsection*{Organization}
In Section \ref{Preliminaries}, we recall some backgrounds on weighted-cscK metrics and re-write the weighted-cscK equation to a coupled PDE system.

In Section \ref{Apriori estimates}, we consider a modified coupled PDE system and prove the uniform $C^0$ estimate  and $W^{2,p}$-estimate  for $\phi$, and the gradient estimate for $F+f_*-\log(\vv(m_\phi))$. Essentially, we only need the log-concavity assumption in the proof of $W^{2,p}$-estimate.

In Section \ref{Proof of Continuity Method}, we prove Theorem \ref{main-thm} by using the continuity method. We give a sketched proof of the openness part in Subsection \ref{Openness}, where we follow the argument in \cite{Ch18}, \cite{Has19}. 
We prove the closedness part in Subsection \ref{Closedness}, 
where we follow the strategy of \cite{CC21b}.

In Section \ref{Appendix} (Appendix), we provide $C^2$-estimate and higher-order estimates for the weighted-cscK equation.
In addition, we give a detailed computation of the linearized operators used in the openness part.

~\\
{\bf Relation with Di-Nezza--Jubert--Lahdili's work}

As our project was nearing completion, we were informed that Eleonora Di-Nezza, Simon Jubert and Abdellah Lahdili also proved the a priori estimates in \cite{DNJL24} by using a method quite different from ours. Our two groups finished our projects independently.

\subsection*{Acknowledgement} 
The authors would like to thank Abdellah Lahdili for sharing his idea on showing the openness near the starting point of the continuity path around 2020. 
The authors would like to thank Eleonora Di Nezza, Simon Jubert, Abdellah Lahdili for pointing out an error (which causes a problem in $W^{2,p}$ estimate if we don't assume the log-concavity of $\vv$) in the first version of our draft.
The authors would like to thank Jianchun Chu, Lifan Guan, Chang Li for their helpful discussions.
The authors would like to thank Prof. Gang Tian, Akito Futaki, Chi Li, Eiji Inoue, Kewei Zhang for their interest and helpful comments. 
The authors would like to thank the anonymous referees for many helpful comments.
J. Han is supported by National Key R\&D Program of China 2023YFA1009900, NSFC No.12301059, XHD23A0101.

\section{Preliminaries}
\label{Preliminaries}

In this section, we will collect the foundations of the variational approach and continuity method that are needed in further analysis. In particular, we will review the weighted Monge-Amp\`ere measure $\vv(m_\phi)\omega_\phi^n$, $\cE^1$-space and $\d_1$-distance, the functionals $\MM_{\vv,\ww}, \JJ_\vv, \EE_\vv$,  coercivity, weighted-cscK (extremal) equations and continuity path.

\subsection{Weighted Monge-Amp\`ere measures, space of potentials, and functionals}
\label{Weighted Monge-Ampere measures, space of potentials, and functionals}
The variational approach has successful applications in the resolution of K\"ahler-Einstein problem, see \cite{Tian97, BBGZ13, BBEGZ19, BBJ21}.
In this subsection, we will define the notions needed in the variational approach. 
We will follow the lines in \cite{BWN14, HL20,AJL23}.

We will consider $K$-invariant $\omega$-psh functions:
\begin{align}
     \Psh_{K}(X,\omega) = \{\phi\in \Psh(X,\omega): \phi \text{ is invariant under the } K\text{-action}\}.
\end{align}

In \cite{BWN14}, the authors define the weighted Monge-Amp\`ere measure by using Kiselman's partial Legendre transform and approximation by series of step functions. Then in \cite{HL20}, the authors present an alternative definition by using a more direct fibration construction, which is more adapted in defining the non-Archimedean functionals.

\begin{prop}[\cite{BWN14, HL20}]
    Let $\vv$ be a smooth positive function on the moment polytope $P$. For any $\omega$-psh function $\phi\in \Psh_K(X,\omega)$, the weighted Monge-Amp\`ere measure ${\rm MA}_\vv(\phi) := \vv(m_\phi)\omega_\phi^n$ is well-defined as a Radon measure.
\end{prop}

It should be noted that such weighted Monge--Amp\`ere measure is understood in the sense of non-pluripolar product.
The following weighted functionals and potential spaces are defined in \cite{BWN14, HL20, Lah19, AJL23}.
\begin{defn}
\label{weighted_fn_def}
    For $\phi\in \Psh_K(X,\omega)$, $\phi = \phi_0 + u$, $\phi_t = \phi_0 + t u$ ($0\leq t\leq 1$), we define the weighted energy functionals:
    \begin{align}
        \EE_\vv(\phi,\phi_0) = \frac{1}{\mathbb{V}_\vv(X)} \int_0^1 \int_X (\phi-\phi_0) \vv(m_{\phi_t}) \frac{\omega_{\phi_t}^n}{n!} \d t,
    \end{align}
    \begin{align}
        \EE_\vv^\eta(\phi,\phi_0) 
        = \frac{1}{\mathbb{V}_\vv(X)} \int_0^1 \int_X (\phi-\phi_0) \Big( \vv(m_{\phi_t}) \eta\wedge \frac{\omega_{\phi_t}^{n-1}}{(n-1)!} + \langle \d\vv(m_{\phi_t}),m_\eta \rangle \frac{\omega_{\phi_t}^n}{n!} \Big)\d t,
    \end{align}
    where $\eta$ is a smooth $(1,1)$-form, which can be locally expressed as $\eta=\ddc \psi$. And $m_\eta(\xi) := -J\xi(\psi)$, which is globally well-defined. And $\langle\d\vv(m_\phi),\cdot\rangle=\d\vv(m_\phi)^{\sharp} = \sum_\alpha \vv_{,\alpha} \xi^\alpha$ as a $\mathfrak{t}$-valued function on $X$, where $(\cdot)^{\sharp}$ is with respect to the K\"ahler form. Therefore, $\langle\d\vv(m_\phi),m_\eta\rangle=\sum_\alpha\vv_{,\alpha}\cdot \xi^\alpha(m_\eta)=-\sum_\alpha \vv_{,\alpha} J\xi^\alpha(\psi)$.
    
    We denote $\EE_\vv(\phi) = \EE_\vv(\phi,0)$. (The same convention applies for other energy functionals.)
    We define the weighted I-functional:
    \begin{align}
        \II_\vv(\phi,\phi_0) = \frac{1}{\mathbb{V}_\vv(X)} \int_X (\phi-\phi_0) \left( \vv(m_{\phi_0})\frac{\omega_{\phi_0}^n}{n!} - \vv(m_{\phi})\frac{\omega_{\phi}^n}{n!} \right),
    \end{align}
    and the weighted J-functional:
    \begin{align}
        &\JJ_\vv(\phi,\phi_0) \\
        =& \frac{1}{\mathbb{V}_\vv(X)} \int_X (\phi-\phi_0) \vv(m_{\phi_0})\frac{\omega_{\phi_0}^n}{n!} - \EE_\vv(\phi,\phi_0) 
        =: \Lambda_\vv(\phi,\phi_0) - \EE_\vv(\phi,\phi_0) \nonumber \\
        =& \int_X \sum_{k=0}^{n-1} \int_0^1 \int_0^t \vv(m_{\phi_s}) \frac{1}{k!(n-k-1)!}(1-s)^k s^{n-k-1} 
        \d s\d t 
         \cdot \d (\phi-\phi_0)\wedge \d^{\rm c} (\phi-\phi_0) \wedge \omega^k \wedge \omega_\phi^{n-k-1} \nonumber.
    \end{align}
    We also define the reduced weighted J-functional:
    \begin{align}
        \JJ_{\vv,\TT}(\phi) = \inf_{\sigma\in \TT} \JJ_{\vv}(\sigma^* \phi).
    \end{align}
\end{defn}
\begin{remark} There exists a constant $C$ which depends on $\vv$, such that
    \begin{align*}
        \frac{1}{C} \JJ(\phi) \leq \JJ_\vv(\phi) \leq C \JJ(\phi),
    \end{align*}
    and
    \begin{align*}
        \frac{1}{C} \big(\II(\phi)-\JJ(\phi)\big) \leq \II_\vv(\phi)-\JJ_\vv(\phi) \leq C \big( \II(\phi)-\JJ(\phi)\big),
    \end{align*}
    where $\II$, $\JJ$ are the classical $I$-functional and $J$-functional:
    \begin{align*}
    	\II(\phi):=\frac{1}{\mathbb{V}(X)}\int_X\phi\omega^n-\frac{1}{\mathbb{V}(X)}\int_X\phi\omega_\phi^n
    \end{align*}
    and
    \begin{align*}
    	\JJ(\phi):=\frac{1}{\mathbb{V}(X)}\int_X\phi\omega^n-E(\phi)
    	=\frac{1}{\mathbb{V}(X)}\int_X\phi\omega^n -\frac{1}{(n+1)\mathbb{V}(X)}\sum_{j=0}^n\int_X\phi\omega_\phi^j\wedge\omega^{n-j}.
    \end{align*}
\end{remark}

\begin{defn} We define the following spaces of potentials:
    \begin{align}
        \cE_K(X,\omega) = \{\phi\in \Psh_K(X,\omega): \int_X \vv(m_\phi) \frac{\omega_\phi^n}{n!} = \int_X \vv(m_\omega) \frac{\omega^n}{n!} =: \mathbb{V}_\vv(X)\},
    \end{align}
    \begin{align}
        \cE^1_K(X,\omega) = \{\phi\in \cE_K(X,\omega):  \EE_\vv(\phi)>-\infty\}.
    \end{align}
\end{defn}

\begin{defn}
\label{def of weighted Mabuchi}
    For any $\phi\in\cE^1_K(X,\omega)$,
    we define the (weighted) Mabuchi functional (for weighted-cscK problem) as:
    \begin{align*}
        \MM_{\vv,\ww}(\phi) 
        =& \int_X \log(\frac{\vv(m_\phi)\omega_\phi^n}{\omega^n}) \vv(m_\phi) \frac{\omega_\phi^n}{n!} - \int_X \log(\vv(m_{\omega}))\vv(m_{\omega}) \frac{\omega^n}{n!}\\
        & -\EE_\vv^{\Ric(\omega)}(\phi) + \EE_\ww(\phi).
    \end{align*}
    We also define the (weighted) Mabuchi for (weighted-extremal problem) as:
    \begin{align*}
        \MM_{\vv,\ww\cdot \ell_{\ext}}(\phi) 
        =& \int_X \log(\frac{\vv(m_\phi)\omega_\phi^n}{\omega^n}) \vv(m_\phi) \frac{\omega_\phi^n}{n!} - \int_X \log(\vv(m_{\omega}))\vv(m_{\omega}) \frac{\omega^n}{n!}\\
        & -\EE_\vv^{\Ric(\omega)}(\phi) + \EE_{\ww\cdot \ell_{\ext}}(\phi).
    \end{align*}
    We will denote
    \begin{align*}
        \MM_{\vv,\ww\cdot \ell_{\ext}}(\phi) 
        = \HH_\vv(\phi) -\EE_\vv^{\Ric(\omega)}(\phi) + \EE_{\ww\cdot \ell_{\ext}}(\phi) - C_0,
    \end{align*}
    where
    \begin{align*}
        \HH_\vv(\phi) = \int_X \log(\frac{\vv(m_\phi)\omega_\phi^n}{\omega^n}) \vv(m_\phi) \frac{\omega_\phi^n}{n!}, \;  C_0 = \int_X \log(\vv(m_{\omega}))\vv(m_{\omega}) \frac{\omega^n}{n!}.
    \end{align*}
\end{defn}

In \cite{Chen00, Dar15, BWN14, BB17, BDL17, Lah19, HL20,  Lah23}, the authors show that the space $\cE^1_{K}(X,\omega)$ is granted with a Finsler structure. The functionals have controlled behaviours with respect to such Finsler structure.
\begin{prop}
    Let $\phi_0,\phi_1$ be two $\omega$-psh functions in $\cE^1_{K}(X,\omega)$. Then there exists a unique geodesic segment $\phi_t$ $(0\leq t \leq 1)$ in $\cE^1_{K}(X,\omega)$ that connects $\phi_0$ and $\phi_1$. The space $\cE^1_{K}(X,\omega)$ is geodesically complete and associated with a distance $\d_1$. In addition, along any geodesic, $\EE_\vv$ is affine, $\JJ_\vv$, $\MM_{\vv,\ww}$ are convex.
\end{prop}

\begin{remark}
\label{v-top}
    Since $\vv$ has positive lower and upper bounds, the strong topology defined by $\EE_\vv$ is equivalent to the strong topology defined by $\EE$. And the entropies $\HH$, $\HH_\vv$ are comparable, i.e.,  there exists $C>0$, such that $\frac{1}{C} (\HH+1) \leq \HH_\vv \leq C (\HH+1)$.
\end{remark}

\begin{defn}[Coercivity]
    Let $\FF$ be a functional defined over the potential space $\cE^1_{K}(X,\omega)$.
    We say that $\FF$ is coercive if there exist constants $\delta>0, C$, such that for any $\phi \in \cE^1_{K}(X,\omega)$,
    \begin{align*}
        \FF(\phi) \geq \delta \cdot \JJ_{\vv}(\phi) - C.
    \end{align*}
    We say that $\FF$ is $\mathbb{G}$-coercive if there exist constants $\delta>0, C$, such that for any $\phi \in \cE^1_{K}(X,\omega)$,
    \begin{align*}
        \FF(\phi) \geq \delta \cdot \JJ_{\vv,\TT}(\phi) - C.
    \end{align*}
\end{defn}

\begin{remark}
    All the definitions and propositions of this subsection can be generalized to the setting of projective klt varieties. See for example \cite[Section 2]{HL20}.
\end{remark}

\subsection{Weighted-cscK equations}
\label{Weighted-cscK equations}

In this subsection, following the idea in \cite{CC21a}, we re-write weighted-cscK equation \eqref{weighted-csck} to a coupled PDE system, in which the first one is a $g$-Monge-Amp\`ere equation (as studied in \cite{BWN14}, \cite{HL20}) and the second one is a linear elliptic equation.
Compared with the cscK equation studied in \cite{CC21a}, one main difference is that, there are first order derivative terms appearing in the second equation.
In the following of this subsection, we will collect some computations that are frequently used in latter sections.

Let 
\begin{align*}
	F=\log\left(\frac{\vv(m_\phi)\omega_\phi^n}{\omega^n}\right) \quad
	i.e., \vv(m_\phi)\omega_\phi^n=e^F\omega^n.
\end{align*}
Then the linearization of L.H.S. is given by
\begin{align*}
	\left.\frac{\d}{\d t}\right|_{t=0}
	(\vv(m_{\phi_t})\omega_{\phi_t}^n)
	=&\vv_{,\alpha}(m_\phi)\left.\frac{\d}{\d t}\right|_{t=0}\langle m_{\phi+t\psi},\xi^\alpha\rangle \omega_\phi^n
	+\vv(m_\phi)\Delta_\phi\psi\omega_\phi^n \\
	=&\left(\frac{\vv_{,\alpha}(m_\phi)}{\vv(m_\phi)}\d^{\rm c}\psi(\xi^\alpha)+\Delta_\phi(\psi)\right)\vv(m_\phi)\omega_\phi^n \\
	=&\left(\Delta_\phi-\frac{\vv_{,\alpha}(m_\phi)}{\vv(m_\phi)} J\xi^\alpha\right)(\psi)\vv(m_\phi)\omega_\phi^n
\end{align*}
where $\phi_t=\phi+t\psi$ and $\Delta_\phi:=\tr_\phi\ddbar$, we have used \cite[Lemma 1]{Lah19}
\begin{align}
\label{moment_map_formula}
	m_{\phi}^\xi=m_\omega^\xi+t\d^{\rm c}\phi(\xi).
\end{align}
Note that
\begin{align}
\label{gradient of logv(m_phi)}
	|\nabla^\phi\log\vv(m_\phi)|_\phi^2
    =\frac{\vv_{,\alpha}(m_\phi)\vv_{,\beta}(m_{\phi})}{\vv(m_\phi)^2}\langle\xi^\beta,\xi^\alpha\rangle_\phi. 
\end{align}
We denote $m_\phi^\xi=\langle m_\phi,\xi\rangle$.

One has
\begin{align}
\label{Lap of logv(m_phi)}
\Delta_\phi\log\vv(m_\phi)
=&g_\phi^{i\bar j}\partial_i\partial_{\bar j}(\log\vv(m_\phi)) \nonumber  
=g_\phi^{i\bar j}\partial_i\frac{\partial_{\bar j}\vv(m_\phi)}{\vv(\phi)} \nonumber 
=g_\phi^{i\bar j}\partial_i\frac{\vv_{,\alpha}(m_\phi)\partial_{\bar j}m_\phi^{\xi^\alpha}}{\vv(\phi)} \nonumber \\
=&g_\phi^{i\bar j}\frac{\vv_{,\alpha}(m_\phi)}{\vv(m_\phi)}\partial_i\partial_{\bar j}m_\phi^{\xi^\alpha}
  +g_\phi^{i\bar j}\frac{\vv_{,\alpha\beta}(m_\phi)\partial_i m_\phi^{\xi^\beta}\partial_{\bar j}m_\phi^{\xi^\alpha}}{\vv(\phi)}
  -g_\phi^{i\bar j}\frac{\vv_{,\beta}(m_\phi)\partial_{i}m_\phi^{\xi^\alpha}\vv_{,\alpha}(m_\phi)\partial_{\bar j}m_\phi^{\xi^\alpha}}{\vv(\phi)^2} \nonumber   \\
=&
\frac{\vv_{,\alpha}(m_\phi)}{\vv(m_\phi)}\Delta_\phi(m^{\xi^\alpha}_\phi)
	+\frac{\vv_{,\alpha\beta}(m_\phi)}{\vv(m_\phi)}\langle\xi^\alpha,\xi^\beta\rangle_\phi
	-\frac{\vv_{,\alpha}(m_\phi)\vv_{,\beta}(m_\phi)}{\vv(m_\phi)^2}\langle\xi^\alpha,\xi^\beta\rangle_\phi
\end{align}
and
\begin{align}
\label{Jxi log volume}
-J\xi^\alpha\log\left(\frac{\omega_\phi^n}{\omega^n}\right)
=\d^{\rm c}\left(\log\left(\frac{\omega_\phi^n}{\omega^n}\right)\right)(\xi^\alpha)
=\Delta_\phi(m^{\xi^\alpha}_\phi)-
	\Delta_\omega(m^{\xi^\alpha}_\omega),
\end{align}
(see \cite[Lemma 5]{Lah19}).
Thus one obtains
\begin{align}
	-\frac{\vv_{,\alpha}(m_\phi)}{\vv(m_\phi)} J\xi^\alpha(F) 
	=&\frac{\vv_{,\alpha}(m_\phi)}{\vv(m_\phi)} \d^cF(\xi^\alpha) \nonumber \\
	=&\frac{\vv_{,\alpha}(m_\phi)\vv_{,\beta}(m_{\phi})}{\vv(m_\phi)^2}\langle\xi^\alpha,\xi^\beta\rangle_\phi
	+ \frac{\vv_{,\alpha}(m_\phi)}{\vv(m_\phi)} \Delta_\phi(m^{\xi^\alpha}_\phi)
	-\frac{\vv_{,\alpha}(m_\phi)}{\vv(m_\phi)} 
	\Delta_\omega(m^{\xi^\alpha}_\omega)
	\label{derivative of F}
\end{align}
Then we have
	\begin{align*}
	\left(\Delta_\phi-\frac{\vv_{,\alpha}(m_\phi)}{\vv(m_\phi)} J\xi^\alpha\right)F
	=&\Delta_\phi\log\vv(m_\phi)+\Delta_\phi\log\left(\frac{\omega_\phi^n}{\omega^n}\right)-\frac{\vv_{,\alpha}(m_\phi)}{\vv(m_\phi)} J\xi^\alpha(F)  \\
	=&2\frac{\vv_{,\alpha}(m_\phi)}{\vv(m_\phi)} \Delta_\phi(m^{\xi^\alpha}_\phi)
+\frac{\vv_{,\alpha\beta}(m_\phi)}{\vv(m_\phi)}\langle\xi^\alpha,\xi^\beta\rangle_\phi 
	-\frac{\vv_{,\alpha}(m_\phi)}{\vv(m_\phi)} 
	\Delta_\omega(m^{\xi^\alpha}_\omega) \\
	&-S(\phi)+\tr_\phi\Ric(\omega) \\
	=&-\frac{S_\vv(\phi)}{\vv(m_\phi)} -\frac{\vv_{,\alpha}(m_\phi)}{\vv(m_\phi)} 
	\Delta_\omega(m^{\xi^\alpha}_\omega) +\tr_\phi\Ric(\omega) \\
	=&-\frac{S_\vv(\phi)}{\vv(m_\phi)} 
	+\vv(m_\phi)^{-1}\Big(\langle\d\vv(m_\phi),m_{\Ric(\omega)}\rangle
	+\vv(m_\phi)\tr_\phi\Ric(\omega) \Big) \\
	=&-\frac{S_\vv(\phi)}{\vv(m_\phi)}+\frac{\tr_{\vv,\phi}\Ric(\omega)}{\vv(m_\phi)},
\end{align*}
where
\begin{align*}
	S_\vv(\phi)
	=\vv(m_\phi)S(\phi)-2\vv_{,\alpha}\Delta_\phi(m_\phi^{\xi^\alpha})-\vv_{,\alpha\beta}(m_\phi)\langle\xi^\alpha,\xi^\beta\rangle_\phi,
\end{align*}
$m_{\Ric(\omega)}=\Delta_\omega(m_\omega)$ 
and
\begin{align*}
	-\vv_{,\alpha}(m_\phi) 
	\Delta_\omega(m^{\xi^\alpha}_\omega)
	=\langle \d\vv(m_\phi),m_{\Ric(\omega)}\rangle
\end{align*}
by \cite[Lemma 5]{Lah19}. Indeed, locally we can write $\omega=\ddc\psi$, $\Ric(\omega)=-\ddc \log\det(\psi_{i\bar j})$, and Ricci potential is $-\log\det(\psi_{i\bar{j}})$. Then
$\Delta_\omega(m^{\xi^\alpha}_\omega) = g^{i\bar{j}}\partial_{i\bar{j}}(-J\xi^\alpha\psi) = J\xi^\alpha(-\log\det(\psi_{i\bar j})) = m_{\Ric(\omega)}$. 

The continuity path \eqref{Chen-continuous path} can be reformulated into the following coupled PDE system
\begin{subnumcases}{}
		\vv(m_\phi)\omega_\phi^n
		=e^F\omega^n, 
		\label{couple system1} \\
		\left(\Delta_\phi-\frac{\vv_{,\alpha}(m_\phi)}{\vv(m_\phi)} J\xi^\alpha\right)F
		=-\ell_{\ext}(m_\phi)\frac{\ww(m_\phi)}{\vv(m_\phi)}
		+\frac{\tr_{\vv,\phi}(\eta_t)}{\vv(m_\phi)}
		-(1-\frac{1}{t})\underline{\theta}\frac{\ww(m_\phi)}{\vv(m_\phi)}, 
		\label{couple system2}
\end{subnumcases}
where we denote $\eta_t:=(1-\frac{1}{t})\theta+\Ric(\omega)$.
By \eqref{derivative of F} and \eqref{couple system2},
\begin{align*}
	\Delta_\phi F
	=&-\ell_{\ext}(m_\phi)\frac{\ww(m_\phi)}{\vv(m_\phi)}
		+\frac{\tr_{\vv,\phi}(\eta_t)}{\vv(m_\phi)}
		-(1-\frac{1}{t})\underline{\theta}\frac{\ww(m_\phi)}{\vv(m_\phi)}  \\
	&-\frac{\vv_{,\alpha}(m_\phi)\vv_{,\beta}(m_{\phi})}{\vv(m_\phi)^2}\langle\xi^\beta,\xi^\alpha\rangle_\phi
	-\frac{\vv_{,\alpha}(m_\phi)}{\vv(m_\phi)}(\Delta_\phi m_\phi^{\xi^\alpha}-\Delta_\omega m_\omega^{\xi^\alpha}) 
\end{align*}
and 
\begin{align*}
	\Delta_\phi(F-\log\vv(m_\phi))
	=&-\ell_{\ext}(m_\phi)\frac{\ww(m_\phi)}{\vv(m_\phi)}
		+\frac{\tr_{\vv,\phi}(\eta_t)}{\vv(m_\phi)}
		-(1-\frac{1}{t})\underline{\theta}\frac{\ww(m_\phi)}{\vv(m_\phi)}  \\
	&-\frac{\vv_{,\alpha}(m_\phi)\vv_{,\beta}(m_{\phi})}{\vv(m_\phi)^2}\langle\xi^\beta,\xi^\alpha\rangle_\phi
	-\frac{\vv_{,\alpha}(m_\phi)}{\vv(m_\phi)}\Delta_\phi m_\phi^{\xi^\alpha}
	 +\frac{\vv_{,\alpha}(m_\phi)}{\vv(m_\phi)}\Delta_\omega m_\omega^{\xi^\alpha} \\
	&- \frac{\vv_{,\alpha}(m_\phi)}{\vv(m_\phi)}\Delta_\phi(m^{\xi^\alpha}_\phi)
	-\frac{\vv_{,\alpha\beta}(m_\phi)}{\vv(m_\phi)}\langle\xi^\alpha,\xi^\beta\rangle_\phi
	+\frac{\vv_{,\alpha}(m_\phi)\vv_{,\beta}(m_\phi)}{\vv(m_\phi)^2}\langle\xi^\alpha,\xi^\beta\rangle_\phi \\
	=&-\left((\ell_\ext+(1-\frac{1}{t})\underline{\theta})\frac{\ww}{\vv}
    -(1-\frac{1}{t})\frac{\langle\d\vv(m_\phi),m_\theta\rangle}{\vv} \right)+\tr_\phi\eta_t \\
    &
    -2\frac{\vv_{,\alpha}(m_\phi)}{\vv(m_\phi)}\Delta_\phi m_\phi^{\xi^\alpha}
	-\frac{\vv_{,\alpha\beta}(m_\phi)}{\vv(m_\phi)}\langle\xi^\alpha,\xi^\beta\rangle_\phi.
\end{align*}
Note that, by \eqref{derivative of F}, 
\begin{align}
	\frac{\vv_{,\alpha}(m_\phi)}{\vv(m_\phi)}J\xi^\alpha(F-\log\vv(m_\phi))
	=&\frac{\vv_{,\alpha}(m_\phi)}{\vv(m_\phi)}J\xi^\alpha(F)
	 -\frac{\vv_{,\alpha}(m_\phi)}{\vv(m_\phi)}J\xi^\alpha(\log\vv(m_\phi)) \nonumber \\
	=&\frac{\vv_{,\alpha}(m_\phi)}{\vv(m_\phi)}J\xi^\alpha(F)
	 -\frac{\vv_{,\alpha}(m_\phi)}{\vv(m_\phi)}\frac{J\xi^\alpha(\vv(m_\phi))}{\vv(m_\phi)} \nonumber \\
	=&\frac{\vv_{,\alpha}(m_\phi)}{\vv(m_\phi)}J\xi^\alpha(F)
	 -\frac{\vv_{,\alpha}(m_\phi)}{\vv(m_\phi)^2}\vv_{,\beta}(m_\phi)J\xi^\alpha(\langle m_\phi,\xi^\beta\rangle) \nonumber \\
	=&\frac{\vv_{,\alpha}(m_\phi)}{\vv(m_\phi)}J\xi^\alpha(F)
	 +\frac{\vv_{,\alpha}(m_\phi)\vv_{,\beta}(m_\phi)}{\vv(m_\phi)^2}\langle\xi^\alpha,\xi^\beta\rangle_\phi \nonumber \\
	=&-\frac{\vv_{,\alpha}(m_\phi)}{\vv(m_\phi)}\Delta_\phi(m_\phi^{\xi^\alpha})
	+\frac{\vv_{,\alpha}(m_\phi)}{\vv(m_\phi)}\Delta_\omega(m_\omega^{\xi^\alpha}).
	\label{Jxi(F')}
\end{align}
It follows that
\begin{align}
\label{Laplace of F'}
	\Delta_\phi(F-\log\vv(m_\phi))
	=&-f_t'+\tr_\phi\eta_t
    +2\frac{\vv_{,\alpha}(m_\phi)}{\vv(m_\phi)}J\xi^\alpha(F-\log\vv(m_\phi))
	-\frac{\vv_{,\alpha\beta}(m_\phi)}{\vv(m_\phi)}\langle\xi^\alpha,\xi^\beta\rangle_\phi,
\end{align}
where
\begin{align*}
    f_t':=(\ell_\ext+(1-\frac{1}{t})\underline{\theta})\frac{\ww}{\vv}
    -(1-\frac{1}{t})\frac{\langle\d\vv(m_\phi),m_\theta\rangle}{\vv}
    +2\frac{\vv_{,\alpha}(m_\phi)}{\vv(m_\phi)}\Delta_\omega(m_\omega^{\xi^\alpha}).
\end{align*}
In addition, by \eqref{Lap of logv(m_phi)} and \eqref{Jxi(F')}
\begin{align}
	\Delta_\phi\log\vv(m_\phi)
=&-\frac{\vv_{,\alpha}(m_\phi)}{\vv(m_\phi)}J\xi^\alpha(F-\log\vv(m_\phi)) +\frac{\vv_{,\alpha}(m_\phi)}{\vv(m_\phi)}\Delta_\omega(m_\omega^{\xi^\alpha}).   \nonumber \\
	&+\frac{\vv_{,\alpha\beta}(m_\phi)}{\vv(m_\phi)}\langle\xi^\alpha,\xi^\beta\rangle_\phi
	-\frac{\vv_{,\alpha}(m_\phi)\vv_{,\beta}(m_\phi)}{\vv(m_\phi)^2}\langle\xi^\alpha,\xi^\beta\rangle_\phi
	\label{Lap of logv(m_phi) and Jxi(F')}
\end{align}
One also has
\begin{align}
\label{Lap of F+logv(m_phi)}
	\Delta_\phi(F+\log\vv(m_\phi))
	=&-f_t+\tr_\phi\eta_t-2\frac{\vv_{,\alpha}(m_\phi)\vv_{,\beta}(m_{\phi})}{\vv(m_\phi)^2}\langle\xi^\beta,\xi^\alpha\rangle_\phi
	+\frac{\vv_{,\alpha\beta}(m_\phi)}{\vv(m_\phi)}\langle\xi^\alpha,\xi^\beta\rangle_\phi,
\end{align}
where
\begin{align*}
    f_t=(\ell_\ext+(1-\frac{1}{t})\underline{\theta})\frac{\ww}{\vv}
    -(1-\frac{1}{t})\frac{\langle\d\vv(m_\phi),m_\theta\rangle}{\vv}.
\end{align*}
We have
\begin{align}
\label{upper bound of xi pair}
	\left|2\frac{\vv_{,\alpha}(m_\phi)\vv_{,\beta}(m_{\phi})}{\vv(m_\phi)^2}-\frac{\vv_{,\alpha\beta}(m_\phi)}{\vv(m_\phi)}\right| \langle\xi^\alpha,\xi^\beta\rangle_\phi 
	\leq A_{\vv,\xi}'\langle\xi^\alpha,\xi^\beta\rangle_\phi 
	\leq A_{\vv,\xi} \tr\omega_\phi,
\end{align}
for some constant $A_{\vv,\xi}>0$ only depends on $\vv$ and $\{\xi^\alpha\}$.

In addition, if $\vv$ is log-concave, then
\begin{align}
\label{log-concavity}
    \left(\frac{\vv_{,\alpha\beta}}{\vv}-\frac{\vv_{,\alpha}\vv_{,\beta}}{\vv^2}\right)\langle\xi^\alpha,\xi^\beta\rangle_\phi
    \leq0.
\end{align}
Indeed, for any $p\in X$, let $\{X^{i}\}$ be the orthnormal basis of $T_pX$ with respect to $g_\phi$ and write $\xi^\alpha=\xi^\alpha_i X^i$.
Thus, one has
\begin{align*}
    \left(\frac{\vv_{,\alpha\beta}}{\vv}-\frac{\vv_{,\alpha}\vv_{,\beta}}{\vv^2}\right)\langle\xi^\alpha,\xi^\beta\rangle_\phi
    =\sum_{i=1}^n\left(\frac{\vv_{,\alpha\beta}}{\vv}-\frac{\vv_{,\alpha}\vv_{,\beta}}{\vv^2}\right)\xi^\alpha_i\xi^\beta_{\bar i}.
\end{align*}
Since $\vv$ is log-concave, then $\left(\frac{\vv_{,\alpha\beta}}{\vv}-\frac{\vv_{,\alpha}\vv_{,\beta}}{\vv^2}\right)_{\alpha,\beta}$ is a negative definite matrix. 
It follows that
\begin{align*}
    \left(\frac{\vv_{,\alpha\beta}}{\vv}-\frac{\vv_{,\alpha}\vv_{,\beta}}{\vv^2}\right)\xi^\alpha_i\xi^\beta_{\bar i}
    =(\xi^1_i,\ldots,\xi^r_i)  \cdot\left(\frac{\vv_{,\alpha\beta}}{\vv}-\frac{\vv_{,\alpha}\vv_{,\beta}}{\vv^2}\right)\cdot (\xi^1_{\bar i},\ldots,\xi^r_{\bar i})^T
    \leq0
\end{align*}
for each fixed $i$. Therefore, \eqref{log-concavity} follows.

In the study of the closedness part of the continuity method, it is also necessary to study the following modified (by automorphism) coupled equations (see Lemma \ref{sigma-action of couple system}):
\begin{subnumcases}{}
		\vv(m_\phi)\omega_\phi^n
		=e^F\omega^n, 
		\label{modified couple system1} \\
		\left(\Delta_{\phi}-\frac{\vv_{,\alpha}(m_{\phi})}{\vv(m_{\phi})} J\xi^\alpha\right)F
	   =-\ell_{\ext}(m_{\phi})\frac{\ww(m_{\phi})}{\vv(m_{\phi})} +(1-\frac{1}{t})\frac{\tr_{\vv,\phi}(\theta)}{\vv(m_{\phi})}\nonumber\\
	  \qquad\qquad\qquad\qquad\qquad\quad	  -\frac{\tr_{\vv,\phi}(\ddc f_*)}{\vv(m_{\phi})}
		-(1-\frac{1}{t})\underline{\theta}\frac{\ww(m_{\phi})}{\vv(m_{\phi})} +\frac{\tr_{\vv,\phi}\Ric(\omega)}{\vv(m_\phi)}, 
		\label{modified couple system2}
\end{subnumcases}
where $f_*\in \Psh(X, (\frac{1-t}{t})\theta)$, $\sup_X f_* = 0$.

\subsection{Examples}
The examples in this subsection all satisfy the log-concavity assumption of $\vv$.
By choosing the functions $\vv,\ww$ accordingly, the expression of the weighted-cscK metric can be formulated into K\"ahler--Einstein metric, cscK metric, K\"ahler--Ricci soliton, extremal metrics. We refer the interested readers to the examples listed in\cite[Section 3]{Lah19}.

Among others, one interesting example is studied in \cite{Ino22}, where the author defined the {constant $\mu_\xi^\lambda$-scalar curvature K\"ahler metric} ($\mu$-cscK for short) as a generalization of cscK metric. 
The $\mu_\xi^\lambda$-scalar curvature is defined by
\begin{align*}
	S_\xi^\lambda(\omega_\phi)
	=S(\omega_\phi) -2\Delta_\phi m_\phi^\xi -J\xi(m_\phi^\xi) -\lambda m_\phi^\xi,
\end{align*}
where $J\xi + \sqrt{-1}\xi$ is a holomorphic vector field, $m_\phi^\xi$ is the corresponding Hamiltonian function of $\xi$, $\lambda\in\RR$.
The $\mu$-cscK metric can be expressed as a weighted-cscK metric by letting $\vv=\ww=e^{m_\phi^\xi}$.
When $X$ is a Fano manifold, one interesting feature of the $\mu$-cscK metric $\omega_\phi$ is that, it interpolates the extremal metric and the K\"ahler--Ricci soliton metric under the assumption of the existence. In particular, as $\lambda\to -\infty$,  $\omega_\phi$ converges to the extremal metric; as $\lambda$ converges to the average scalar curvature $\underline{S}>0$, $\omega_\phi$ converges to the K\"ahler--Ricci soliton.

Another interesting feature is, when the $\mu$-cscK metric exists, the vector field $\xi$ is a critical point of the volume functional $\Vol^\lambda(\xi)$, where
\begin{align*}
    \log(\Vol^\lambda)(\xi) = \frac{1}{\int_X e^{m_\phi^\xi}\omega_\phi^n}\int_X \Big(S(\omega_\phi) -\Delta_\phi m_\phi^\xi - \lambda m_\phi^\xi \Big) e^{m_\phi^\xi} \omega_\phi^n + \lambda\cdot \log(\int_X e^{m_\phi^\xi}\omega_\phi^n).
\end{align*}
After a normalization, $\log(\Vol^\lambda)$ has the same expression as Perelman's $W$-functional \cite{Pe02}. When restricted on Fano manifolds, $\log(\Vol^\lambda)$ is equal to Tian--Zhu's volume functional \cite{TZ02} up to the difference of a constant.

Some other related works can be found in \cite{He19, Ino21, Liu22, Dyr22, Jub23, Zha23, Fut24, GJSS24, His24, NN24}.

\section{Apriori estimates}
\label{Apriori estimates}
In this section, we provide uniform $C^0$-estimate, $W^{2,p}$-estimate of $\phi$ and gradient estimate of $F+f_*-\log(\vv(m_\phi))$ for equation \eqref{modified couple system1}, \eqref{modified couple system2}.
The proof of the uniform $C^0$-estimate is based on a modification of Guo--Phong's method \cite{guo2022}. 
The proof of $W^{2,p}$-estimate and gradient estimate follows from Chen--Cheng's method \cite{CC21b} with a new choice of test functions and integral weights.

\subsection{Uniform \texorpdfstring{$C^0$}{}-estimate}
\label{Uniform C0-estimate}

Let $\omega_X$ be a fixed smooth K\"ahler form, and let $\omega = \chi + s \omega_X$ be a smooth K\"ahler form, where $\chi$ is a smooth $(1,1)$-form in a nef class, $s> 0$.
We simplify the equations \eqref{modified couple system1}, \eqref{modified couple system2}.
The equation \eqref{modified couple system1} can be rewritten as
\begin{align*}
	\vv(m_\phi)\; \omega_\phi^n
		= e^{F+\log\frac{\omega^n}{\omega_X^n}}  \omega_X^n
		=e^{\tilde F+\log c_{\omega,\vv}}\omega_X^n,
\end{align*}
where we denote $\tilde{F}=F+\log\frac{\omega^n}{\omega_X^n}-\log c_{\omega,\vv}$.
On the other hand, the equation \eqref{modified couple system2} can be re-written as
\begin{align*}
	(\Delta_\phi-\frac{\vv_{,\alpha}}{\vv}J\xi^\alpha)(F+f_*)=-f_t-\frac{\vv_{,\alpha}}{\vv}\Delta_\omega(m_\omega^{\xi^\alpha})+\tr_\phi((1-\frac{1}{t})\theta+\Ric(\omega)).
\end{align*}
Thus, by \eqref{Jxi log volume}, one has
\begin{align*}
	(\Delta_\phi-\frac{\vv_{,\alpha}}{\vv}J\xi^\alpha)(\tilde{F}+f_*)
	=&-f_t-\frac{\vv_{,\alpha}}{\vv}\Delta_\omega(m_\omega^{\xi^\alpha})+\tr_\phi((1-\frac{1}{t})\theta+\Ric(\omega)) \\
	 &+\Delta_\phi\log\frac{\omega^n}{\omega_X^n}-\frac{\vv_{,\alpha}}{\vv}J\xi^\alpha\log\frac{\omega^n}{\omega_X^n}\\
	=&-f_t-\frac{\vv_{,\alpha}}{\vv}\Delta_{\omega_X}(m_{\omega_X}^{\xi^\alpha})+\tr_\phi\Big((1-\frac{1}{t})\theta+\Ric(\omega_X)\Big).
\end{align*}

We will consider the uniform $C^0$-estimate of the following PDE system
\begin{equation}
\label{CscK_u}
\begin{cases}
    \vv(m_\phi)\; \omega_\phi^n
		=c_{\omega,\vv} \; e^{\tilde{F}}  \omega_X^n, \\
  (\Delta_\phi - \frac{\vv_{,\alpha}}{\vv}J\xi^\alpha ) \big(\tilde{F} +  f_* \big)
		=-f + \tr_\phi\eta, 
\end{cases}
\end{equation}
where  $c_{\omega,\vv} = \int_X \vv \; \omega^n/{\int_X\omega_X^n}$. 
We assume that $f$ is a smooth function on $X$. Furthermore, we also assume that there exists a constant $p>1$ such that $\int_X e^{-p f_*} \omega^n < C(p)$.
\begin{remark}
\label{remark for xi(f_*)}
    In the modified continuity path, for the $(\frac{1-t}{t})\theta$-psh function $f_*$, for arbitrarily large $p$, we have $\int_X e^{-pf_*}\omega^n$ uniformly bounded from above when $t$ is sufficiently close to $1$ (see Proposition \ref{integrability of f_*}).
    Furthermore, the function $\frac{\vv_{,\alpha}}{\vv} J\xi^\alpha (f_*)$ is bounded by a constant which only depends on the class $[\theta]$. 
    Indeed,  denote $\theta_t:=\frac{1-t}{t}\theta$ and by \eqref{moment_map_formula},
    \begin{align*}
    	-\frac{\vv_{,\alpha}(m_\phi)}{\vv(m_\phi)} J\xi^\alpha(f_*)
	=\frac{\vv_{,\alpha}(m_\phi)}{\vv(m_\phi)}\d^cf_*(\xi^\alpha)
	=\frac{\vv_{,\alpha}(m_\phi)}{\vv(m_\phi)}(m_{\theta_t+\ddc f_*}^{\xi^\alpha}-m_{\theta_t}^{\xi^\alpha}),
    \end{align*}
   where the right hand side only depends on the moment polytope determined by $[\theta]$.
       
    Therefore the modified continuity path \eqref{modified couple system1}, \eqref{modified couple system2} can be reformulated into \eqref{CscK_u} when $t$ is sufficiently close to $1$.
\end{remark}

\begin{thm}
\label{C^0-estimate}
	Assume $\omega \leq \kappa \omega_X$ for some $\kappa >0$ and
 $\int_X \tilde{F} e^{\tilde{F}}\omega_X^n \leq K_1$, $\eta \geq -K_2 \omega$, $\int_X e^{-p f_*} \omega^n < C(p)$ for some $p > 1$, $\sup_X f_* = 0$, $\sup_X \phi = 0$.
 Then
 $\sup_X |\phi|\leq C$, $\sup_X (\tilde F+  f_*) \leq C$,
 where the constant $C$ depends on $\omega_X, n,\kappa,  f, \vv, C(p), K_1, K_2$.

 If we further assume $\eta\leq K_3 \omega$, then
 $\tilde{F}+ f_*$ is bounded from below by another constant depending further on $K_3$.
\end{thm}

\begin{remark}
    The format of the statement of Theorem \ref{C^0-estimate} is to make it suitable for our future application in singular cases.
\end{remark}

\begin{proof}

    By applying \cite[Theorem 2]{guo2022} to equation $\omega_\phi^n = c_\omega e^{\tilde{F}}\omega_X^n$, for $l>n$, we have
    $\sup_X|\phi|\leq C$, where $C$ depends on $\omega_X, n, l, \kappa, {\rm Ent}_l(e^{\tilde{F}})$.

    Therefore, we just need to modify \cite[Theorem 3]{guo2022} to use $\int_X \tilde{F} e^{\tilde{F}} \omega_X^n$ to bound ${\rm Ent}_l(\tilde{F})$.

    We need to use the auxiliary function $\psi_k$ which satisfies the following equation
    \begin{equation}
        (\omega+{\rm dd}^c \psi_k)^n = \frac{\tau_k\Big(-\phi+\delta(\tilde{F} +  f_* )\Big)}{A_k} \cdot c_{\omega,\vv} \cdot e^{\tilde{F}} \omega_X^n,
        \quad\sup_X\psi_k=0,
    \end{equation}
    where $\delta = \frac{1}{10}K_2^{-1}$, $\tau_k:\RR\>\RR_{+}$ is a sequence of positive smooth functions that converges monotonically decreasingly to the function $x\cdot{\rm Id}_{\RR_+}(x)$.
    And 
    $$A_k = \frac{c_{\omega,\vv}}{V_\omega}\int_X \tau_k\Big(-\phi+\delta(\tilde{F} + f_*  )\Big)\cdot e^{\tilde{F}}\omega_X^n,$$ which converges to 
    \begin{align*}
        A_\infty &= \frac{c_{\omega,\vv}}{V_\omega}\int_\Omega \Big(-\phi+\delta(\tilde{F} + f_* ) \Big)e^{\tilde{F}}\omega_X^n \\
        &\leq \frac{c_{\omega,\vv}}{V_\omega}\int_\Omega \Big(-\phi+\delta\cdot \tilde{F}  \Big) e^{\tilde{F}}\omega_X^n,
    \end{align*}
     where $\Omega = \{x\in X:-\phi+\delta(\tilde{F} + f_*) >0 \}$ and $0\leq A_\infty\leq C(K_1)$, which is due to the entropy bound assumption, \cite[Theorem 1 (ii)]{guo2022} for $p=1$ and H\"older inequality.

 Consider 
\begin{equation}
    \Psi = -\epsilon (-\psi_k+\Lambda)^{\frac{n}{n+1}} -\phi + \delta (\tilde{F}+ f_* ),
\end{equation}
    where $\epsilon =2^{\frac{n-1}{n+1}} (\frac{(n+1)(n + C_{f,\delta})}{n^2})^{\frac{n}{n+1}} A_k^{\frac{1}{n+1}}$, $\Lambda^{\frac{1}{n+1}} = \frac{2n}{n+1}\epsilon$, where the constant $C_{f,\delta}$ is chosen such that $\left|-\delta f + \frac{\vv_\alpha}{\vv}J\xi^\alpha(\phi) \right|\leq C_{f,\delta}$.

    The next is to show that $\Psi\leq 0$. It suffices to assume $\sup_X \Psi$ is achieved at a point $x_0\in \Omega$.
    We have
    \begin{align*}
        \frac{\vv_\alpha}{\vv}J\xi^\alpha (\Psi) (x_0) = 0,
    \end{align*}
    which together with the second equation of \eqref{CscK_u} implies
    \begin{align}
        \delta\cdot \Delta_\phi (\tilde F+f_*) 
        = -\delta\cdot f + \frac{\vv_\alpha}{\vv}J\xi^\alpha (\phi) 
        + \delta \cdot \tr_\phi \eta
        - \frac{n\epsilon}{n+1} \frac{\vv_\alpha}{\vv}J\xi^\alpha(\psi_k) \cdot (-\psi_k + \Lambda)^{-\frac{1}{n+1}}.
    \end{align}

    Then at $x_0$, we have
    \begin{align*}
        0 &\geq \Delta_\phi \Psi \\
        & \geq \frac{\epsilon n}{n+1} (-\psi_k+\Lambda)^{\frac{-1}{n+1}}({\rm tr}_\phi\omega_{\psi_k}-{\rm tr}_\phi\omega) - (n-{\rm tr}_\phi \omega) + \delta \Delta_\phi(\tilde{F}+ f_* )\\
        &\geq \frac{\epsilon n^2}{n+1} (-\psi_k+\Lambda)^{\frac{-1}{n+1}}
        \Big( \frac{-\phi+\delta(\tilde{F}+ f_* )}{A_k} \Big)^{\frac{1}{n}}
         - n -  C_{f,\delta} \\
        &\quad+\left(1-\frac{\epsilon n}{n+1}(-\psi_k+\Lambda)^{\frac{-1}{n+1}} - \frac{1}{10}\right) {\rm tr}_\phi \omega 
         - \frac{n\epsilon}{n+1}\frac{\vv_\alpha}{\vv}J\xi^\alpha(\psi_k)(-\psi_k+\Lambda)^{\frac{-1}{n+1}}\\
        &\geq \frac{\epsilon n^2}{n+1} (-\psi_k+\Lambda)^{\frac{-1}{n+1}} \left( \Big(
        \frac{-\phi+\delta(\tilde{F}+ f_* )}{A_k} \Big)^{\frac{1}{n}} -\frac{\vv_\alpha}{n\vv}J\xi^\alpha(\psi_k) \right) -(n +  C_{f,\delta} ).
    \end{align*}
Note that $|\frac{\vv_\alpha}{n \vv} J\xi^\alpha (\psi_k)|$ is bounded by some constant $C'$ uniformly.
This implies
\begin{align*}
    -\phi+\delta(\tilde{F}+ f_* )  
    \leq&\left[C'+\frac{(n+1)(n+C_{f,\delta})}{\epsilon n^2}(-\psi_k+\Lambda)^{\frac{1}{n+1}} \right]^nA_k \\
    \leq&2^{n-1}\left( \frac{(n+1)(n+C_{f,\delta})}{\epsilon n^2} \right)^n A_k (-\psi_k+\Lambda)^{\frac{n}{n+1}} 
    + 2^{n-1}C'^{n}A_k \\
    =& \epsilon (-\psi_k+\Lambda)^{\frac{n}{n+1}} + C'.
\end{align*}
    This concludes that $\sup_X \Psi  = \Psi(x_0) \leq C'$.
Therefore
\begin{align*}
    \delta (\tilde{F}+ f_* ) \leq -\phi + \delta (\tilde{F}+ f_* ) \leq C \big( (-\psi_k +1)^{\frac{n}{n+1}} +1) \leq -\tilde{\epsilon} \psi_k + C_{\tilde{\epsilon}},
\end{align*}
 where $\tilde{\epsilon}$ can be chosen arbitrarily small. Denote $q^*>1$ such that $\frac{1}{q}+\frac{1}{q^*}=1$.
 Choose $r\,\tilde{\epsilon}$ such that $r\delta>1$ and $r q \tilde{\epsilon} \leq \alpha(X, \kappa \omega_X)$ for some $q^*>1$ that is close to $1$, and $p = q^* r \delta>1$.  Then we have
 \begin{align*}
     \int_X e^{r\delta \tilde{F}}\omega_X^n 
     &= \int_X e^{r\delta (\tilde{F}+f_* 
     -f_*)}\omega_X^n \\
     &\leq \Big(\int_X e^{q r\delta (\tilde{F}+  f_* )}\omega_X^n \Big)^{\frac{1}{q}}\cdot \Big(\int_X e^{-q^* r\delta f_*}\omega_X^n  \Big)^{\frac{1}{q^*}} \\
     &\leq C \Big(\int_X e^{-q r\tilde{\epsilon}\psi_k}  \omega_X^n \Big)^{\frac{1}{q}} \leq C'.
 \end{align*}
In particular, this implies $\Ent_l(\tilde{F})$ is uniformly bounded. Then by \cite[Theorem 1]{guo2023} we have $\sup_X |\phi| \leq C$, where the constant $C$ depends on parameters listed in the statement of the theorem. (Alternatively, we can also use the $C^0$-estimate in \cite{EGZ09}).

In order to prove the upper and lower bounds of $\tilde{F}+ f_*$, we need to use the mean-value type inequality proved in \cite{guo2022}.

Let $u:= \tilde{F}+ f_*-K_2\phi$, which satisfies $(\Delta_\phi - \frac{\vv_\alpha}{\vv}J\xi^\alpha ) u \geq -a$ for some constant $a\geq 0$.
Apply \cite[Lemma 3]{guo2022} to $u$, we can obtain the upper bound of $\tilde{F}+ f_*$.

Similarly, let $u:= -(\tilde{F}+ f_* )-K_3 \phi$, which satisfies $(\Delta_\phi -\frac{\vv_\alpha}{\vv}J\xi^\alpha ) u \geq -a$ for some $a\geq 0$, we can obtain a lower bound of $\tilde{F}+f_*$.
\end{proof}

\subsection{\texorpdfstring{$W^{2,p}$}{}-estimate}

For the convenience of the computation, 
since
\begin{align*}
	\frac{\tr_{\vv,\phi}(\ddc f_*)}{\vv(m_{\phi})}
	=\Delta_\phi f_*+\frac{\langle\d\vv(m_\phi),\d^cf_*\rangle}{\vv}
	=\Delta_\phi f_*-\frac{\vv_{,\alpha}(m_\phi)}{\vv(m_\phi)}J\xi^\alpha f_*,
\end{align*} we rewrite the coupled system \eqref{modified couple system1}, \eqref{modified couple system2} into
\begin{align}
\label{csck-mt}
         \vv(m_{\phi})\omega_{\phi}^n
		 &=e^{F}\omega^n, 
		 \\
		\left(\Delta_{\phi}-\frac{\vv_{,\alpha}(m_{\phi})}{\vv(m_{\phi})} J\xi^\alpha\right)F
		 & =\tilde{f}_t +\tr_{\phi}(\eta_t) 
            -\Delta_\phi f_*
            +\frac{\vv_{,\alpha}}{\vv}J\xi^\alpha f_*
            \nonumber,
\end{align}
where $f_*\in \Psh(X, (\frac{1}{t}-1)\theta)$, $\sup_X f_*=0$,
\begin{align*}
\tilde{f}_t = -(\ell_\ext+(1-\frac{1}{t})\underline{\theta})\frac{\ww}{\vv}
    +(1-\frac{1}{t})\frac{\langle\d\vv(m_\phi),m_\theta\rangle}{\vv}
    -\frac{\vv_{,\alpha}}{\vv}\Delta_\omega m_\omega^{\xi^\alpha}
\end{align*} 
which is a bounded function,
    $\eta_t = (1-\frac{1}{t})\theta + \Ric(\omega)$,
and $\int_X e^{-p f_*} \omega^n <C$ for a sufficiently large positive constant $p$. (which can be achieved when $t\to 1$ in the modified continuity path.)

Note that by \eqref{Lap of logv(m_phi)} and \eqref{derivative of F}, the second equation of \eqref{csck-mt} can be rewritten as
\begin{align}
	\Delta_\phi(F+\log\vv(m_\phi))
	=&\tilde{f_t}+\tr_{\phi}(\eta_t) 
            -\Delta_\phi f_*
            +\frac{\vv_{,\alpha}}{\vv}J\xi^\alpha f_*+\frac{\vv_{,\alpha}}{\vv}J\xi^\alpha F+\Delta_\phi\log\vv(m_\phi) \nonumber \\
   =&\tilde{f_t}+\frac{\vv_{,\alpha}}{\vv}\Delta_\omega(m_\omega^{\xi^\alpha})+\tr_{\phi}(\eta_t) 
            -\Delta_\phi f_*
            +\frac{\vv_{,\alpha}}{\vv}J\xi^\alpha f_* \nonumber \\
            &-2|\nabla^\phi\log\vv(m_\phi)|_{\phi}^2 
		 +\frac{\vv_{,\alpha\beta}(m_\phi)}{\vv(m_\phi)}\langle\xi^\alpha,\xi^\beta\rangle_\phi \nonumber \\
  =&-f_t+\frac{\vv_{,\alpha}}{\vv}J\xi^\alpha f_*  -\Delta_\phi f_* +\tr_\phi\eta_t
		 -2|\nabla^\phi\log\vv(m_\phi)|_{\phi}^2 
		 +\frac{\vv_{,\alpha\beta}(m_\phi)}{\vv(m_\phi)}\langle\xi^\alpha,\xi^\beta\rangle_\phi.
		 \label{Lap of F+logv(m_phi) twisted} 
\end{align}

\begin{lemma}
\label{Lap_phi log tr}
	Let $\phi$ be a solution of \eqref{csck-mt}, then
	\begin{align*}
		\Delta_\phi\log\tr\omega_\phi
		\geq\frac{\Delta(F-\log\vv(m_\phi))-S(\omega)}{\tr\omega_\phi}-B\tr_\phi\omega,
	\end{align*}
	where $B>0$ is a constant depending on the lower bound of the holomorphic bisectional curvature of $\omega$.
\end{lemma}
\begin{proof}
	By \cite[Lemma 3.7]{Sze14}, then
	\begin{align*}
		\Delta_\phi\log\tr\omega_\phi
		\geq-B\tr_\phi\omega- \frac{g^{j\bar k}\Ric(\phi)_{j\bar k}}{\tr\omega_\phi},
	\end{align*}
	where $B>0$ is a constant depending on the lower bound of the holomorphic bisectional curvature of $\omega$.
	Since
	\begin{align*}
		\Ric(\omega_\phi)
		=\Ric(\omega)-\ddbar\log\frac{\omega_\phi^n}{\omega^n}
		=\Ric(\omega)-\ddbar(F-\log\vv(m_\phi)),
	\end{align*} 
	then
	\begin{align*}
		g^{j\bar k}\Ric(\phi)_{j\bar k}
		=S(\omega)-\Delta(F-\log\vv(m_\phi)).
	\end{align*}
	Therefore the lemma follows.
\end{proof}

\begin{thm}
\label{integral C^2}
	Assume $\vv$ is log-concave. Let $\phi$ be a solution of \eqref{csck-mt},  for any $p\geq1$, then
	\begin{align}
		\int_X e^{(p-1)f_*}(n+\Delta\phi)^p \omega^n \leq C,
	\end{align}
	where $C$ is a constant depending on $p$, $\vv$, $\ww$, $\|\phi\|_{C^0}$, $\|F+f_*\|_{C^0}$, an upper bound of $\theta$ and $\Ric(\omega)$,  and a lower bound of the holomorphic bisectional curvature of $\omega$.
\end{thm}
\begin{remark}
    By H\"older inequality, we also have
    \begin{align}
        \int_X (n+\Delta\phi)^p \omega^n \leq C,
    \end{align}
    for a constant $C$ that depends on the same parameters as stated above.
\end{remark}
\begin{proof}
	We consider 
	\begin{align*}
		u:=e^{-a(F+\log\vv(m_\phi)+\delta f_* +b\phi)}\tr\omega_\phi,
	\end{align*}
	where $a,b>1$ are uniform constants determined later.
	Since 
	$$\ddbar e^{f}=e^f\ddbar f+e^f\sq\p f\wedge\db f,$$
	then one has
	\begin{align}
		\Delta_\phi u
		=\Delta_\phi e^{\log u}
		=&e^{\log u}\Delta_\phi\log u
		 +u\nabla^\phi(\log u)\cdot_\phi\nabla^\phi(\log u) \nonumber \\
		=&-au\Delta_\phi(F+\log\vv(m_\phi) +\delta f_* +b\phi) +u\Delta_\phi\log\tr\omega_\phi+u^{-1}|\nabla^\phi u|_\phi^2 .
		\label{Lap_phi e logu}
	\end{align}
	By \eqref{Lap of F+logv(m_phi) twisted} and Remark \ref{remark for xi(f_*)}, we have
	\begin{align}
		&\Delta_\phi(F+\log\vv(m_\phi) + \delta f_* +b\phi) \nonumber \\
		&= -f_t+\frac{\vv_{,\alpha}}{\vv}J\xi^\alpha f_* +\tr_\phi\eta_t
		 -2|\nabla^\phi\log\vv(m_\phi)|_{\phi}^2 
		 +\frac{\vv_{,\alpha\beta}(m_\phi)}{\vv(m_\phi)}\langle\xi^\alpha,\xi^\beta\rangle_\phi +bn-b\tr_\phi\omega \nonumber
            -(1-\delta)\Delta_\phi f_*\\
		&\leq (C+bn-f_t) + (A_1-b)\tr_\phi\omega
		-2|\nabla^\phi\log\vv(m_\phi)|_{\phi}^2 
		 +\frac{\vv_{,\alpha\beta}(m_\phi)}{\vv(m_\phi)}\langle\xi^\alpha,\xi^\beta\rangle_\phi -(1-\delta)\Delta_\phi f_*, 
   \label{Lap_phi(tilde F+bphi)}
	\end{align}
	where $A_1>0$ is a constant satisfying $\eta_t\leq A_1\omega$.
	By Lemma \ref{Lap_phi log tr}, \eqref{Lap_phi e logu} and \eqref{Lap_phi(tilde F+bphi)}, 
	\begin{align}
		\Delta_\phi u 
		&\geq e^{-a(F+\log\vv +\delta f_* +b\phi)}
		\Big(a(f_t-bn-C + (1-\delta)\Delta_\phi f_*)\tr\omega_\phi+(ab-aA_1-B)\tr_\phi\omega\tr\omega_\phi \nonumber \\
		&\qquad\quad
		+\Delta F'-S(\omega) \Big) +2au|\nabla^\phi\log\vv(m_\phi)|_{\phi}^2 
		 -au\frac{\vv_{,\alpha\beta}(m_\phi)}{\vv(m_\phi)}\langle\xi^\alpha,\xi^\beta\rangle_\phi+u^{-1}|\nabla^\phi u|_\phi^2\ ,
		 \label{Lap_phi u geq in W^2,p}
	\end{align}
	where we denote $F':=F-\log\vv(m_\phi)$.

Note that
\begin{align*}
		\frac{1}{2p+1}\Delta_\phi u^{2p+1}
		=&u^{2p}\Delta_\phi u +2pu^{2p-1}|\nabla^\phi u|_\phi^2. 
	\end{align*}
Integrate the equality above with respect to the weighted measure $e^{2a\log\vv}\omega_\phi^n$, together with \eqref{Lap_phi u geq in W^2,p},
we obtain
\begin{align*}
    & \int_X 2p u^{2p-1} \cdot |\nabla^\phi u|_\phi^2 \cdot e^{2a\log\vv} \omega_\phi^n \\
    =& \int_X u^{2p} \Big( -\Delta_\phi u - \langle \nabla^\phi u, 2a \nabla^\phi \log\vv \rangle_\phi \Big) \cdot e^{2a\log\vv} \omega_\phi^n\\
    \leq& -2a\int_X u^{2p} \langle \nabla^\phi u,\nabla^\phi \log\vv \rangle_\phi e^{2a\log\vv} \omega_\phi^n \\
    & -\int_X u^{2p}e^{-a(F+\log\vv+\delta f_* +b\phi)} \Bigl\{ a\Big(f_t-bn-C + (1-\delta)\Delta_\phi f_* \Big) \tr\omega_\phi
    + (ab-a A_1 - B)\tr_\phi\omega \tr \omega_\phi \\
    & + \Delta (F-\log\vv) - S(\omega) \Bigr\} e^{2a\log\vv} \omega_\phi^n \\
    & -2a\int_X u^{2p+1} |\nabla^\phi \log\vv|_\phi^2 e^{2a\log\vv} \omega_\phi^n - \int_X u^{2p-1} |\nabla^\phi u|_\phi^2 e^{2a\log\vv} \omega_\phi^n\\
    & + a\int_X u^{2p+1} \frac{\vv_{,\alpha\beta}}{\vv} \langle \xi_{\alpha}, \xi_{\beta} \rangle_\phi e^{2a\log\vv} \omega_\phi^n.
\end{align*}
Using the fact that $|S(\omega)| < C$, and by choosing $b>1$ sufficiently large so that $ab-a A_1 -B \geq \frac{ab}{2}$ and  $|a(f_t-bn-C)| \leq(n+1)ab$,
the inequality above can be simplified into
\begin{align}
    & (2p+1-a)\int_X  u^{2p-1} \cdot |\nabla^\phi u|_\phi^2 \cdot e^{2a\log\vv} \omega_\phi^n \nonumber \\
    \leq& (n+1)ab\int_X u^{2p+1} e^{2a\log\vv} \omega_\phi^n -a(1-\delta) \int_X u^{2p+1} \Delta_\phi f_* e^{2a\log\vv} \omega_\phi^n \nonumber
 \\
    & -\frac{ab}{2}\int_X u^{2p+1} \tr_\phi \omega e^{2a\log\vv} \omega_\phi^n
     + C\int_X u^{2p} e^{-a(F+\log\vv+\delta f_* +b\phi)} e^{2a\log\vv} \omega_\phi^n \nonumber \\
    &    -\int_X u^{2p} e^{-a(F+\log\vv+\delta f_* +b\phi)} \Delta(F-\log\vv) e^{2a\log\vv} \omega_\phi^n,
    \label{ineq reduced by log-concave}
\end{align}
where we use the fact \eqref{log-concavity}
which is ensured by the log-concavity of $\vv$, and 
\begin{align*}
    & -2a\int_X u^{2p} \langle \nabla^\phi u,\nabla^\phi \log\vv \rangle_\phi e^{2a\log\vv} \omega_\phi^n \\
    \leq& {a} \int_X u^{2p+1} |\nabla^\phi \log\vv|_\phi^2 e^{2a\log\vv}\omega_\phi^n 
     + a\int_X u^{2p-1} |\nabla^\phi u|_\phi^2 e^{2a\log\vv}\omega_\phi^n.
\end{align*}
  Next, we handle the term involving $\Delta(F-\log\vv)$ in the above inequality by repeating the integration by parts.
Recall that $\omega_\phi^n = e^{F-\log\vv}\omega^n$.
 \begin{align*}
     & -\int_X u^{2p} e^{-a(F+\log\vv +\delta f_* +b\phi)} \Delta (F-\log\vv) e^{2a\log\vv} \omega_\phi^n \\
     =& -\int_X u^{2p} e^{(1-a)(F-\log\vv)-a\delta f_* -ab\phi} \frac{1}{1-a} \Delta \Big( (1-a)(F-\log\vv) \Big) \omega^n \\
     =& -\int_X u^{2p} e^{(1-a)(F-\log\vv)-a\delta f_* -ab\phi} \cdot \frac{1}{1-a} \Delta \Big((1-a)(F-\log\vv) -a\delta f_* -ab\phi \Big) \omega^n \\
     & - \int_X u^{2p} e^{(1-a)(F-\log\vv)-a\delta f_* -ab\phi} \cdot \frac{a\delta\Delta f_* + ab\Delta\phi}{1-a} \omega^n \\
     =:& I + II.
 \end{align*}
One has
 \begin{align*}
     I =& -\int_X u^{2p} e^{(1-a)(F-\log\vv)-a\delta f_* -ab\phi} \cdot \frac{1}{1-a} \Delta \Big((1-a)(F-\log\vv) -a\delta f_* -ab\phi \Big) \omega^n \\
     =& -\int_X \frac{u^{2p} e^{(1-a)(F-\log\vv)-a\delta f_* -ab\phi}}{a-1} \cdot | \nabla \big( (1-a)(F-\log\vv) -a\delta f_* -ab\phi \big) |^2 \omega^n \\
     & - \int_X \frac{2p}{a-1} u^{2p-1} e^{(1-a)(F-\log\vv)-a\delta f_* -ab\phi} \cdot \langle \nabla u, \nabla \big( (1-a)(F-\log\vv) -a\delta f_* -ab\phi \big) \rangle \omega^n \\
     \leq&  -\int_X \frac{u^{2p} e^{(1-a)(F-\log\vv)-a\delta f_* -ab\phi}}{a-1} \cdot | \nabla \big( (1-a)(F-\log\vv) -a\delta f_* -ab\phi \big) |^2 \omega^n \\
     & + \int_X \frac{u^{2p} e^{(1-a)(F-\log\vv)-a\delta f_* -ab\phi}}{(a-1)} \cdot | \nabla \big( (1-a)(F-\log\vv) -a\delta f_* -ab\phi \big) |^2 \omega^n \\
     & + \int_X \frac{p^2}{a-1} u^{2p-2} e^{(1-a)(F-\log\vv)-a\delta f_* -ab\phi} \cdot |\nabla u|^2 \omega^n \\
     \leq& \frac{p^2}{a-1} \int_X u^{2p-2} e^{-a(F-\log\vv)-a\delta f_* -ab\phi} \cdot |\nabla u|^2 \omega_\phi^n \\
     =& \frac{p^2}{a-1} \int_X  u^{2p-2} e^{-a(F+\log\vv+\delta f_* +b\phi)} \cdot |\nabla u|^2 e^{2a\log\vv} \omega_\phi^n \\
     \leq& \frac{p^2}{a-1} \int_X  u^{2p-2} e^{-a(F+\log\vv+\delta f_* +b\phi)} \cdot |\nabla^\phi u|_\phi^2 (n+\Delta\phi) e^{2a\log\vv} \omega_\phi^n \\
     =&\frac{p^2}{a-1} \int_X  u^{2p-1} |\nabla^\phi u|_\phi^2 e^{2a\log\vv} \omega_\phi^n,
 \end{align*}
where we have used 
	\begin{align*}
		|\nabla^\phi u|_\phi^2\tr\omega_\phi\geq|\nabla u|^2
	\end{align*}
for the sixth inequality.

Then the inequality \eqref{ineq reduced by log-concave} can be simplified into
\begin{align}
    & \int_X (2p+1-a) u^{2p-1} \cdot |\nabla^\phi u|_\phi^2 \cdot e^{2a\log\vv} \omega_\phi^n 
    +\frac{ab}{2} \int_X u^{2p+1} \tr_\phi \omega e^{2a\log\vv} \omega_\phi^n \nonumber \\
    \leq& (n+1)ab \int_X u^{2p+1} e^{2a\log\vv}\omega_\phi^n 
    + C \int_X u^{2p} e^{-a(F+\log\vv+\delta f_* +b\phi)} e^{2a\log\vv} \omega_\phi^n \nonumber \\
    & -a(1-\delta) \int_X u^{2p+1} \Delta_\phi f_* e^{2a\log\vv} \omega_\phi^n
    + \frac{p^2}{a-1} \int_X u^{2p-1} |\nabla_\phi u|_\phi^2 e^{2a\log\vv} \omega_\phi^n +II.
    \label{interm. inequ involve II}
\end{align}
Note that
\begin{align}
	&II-a(1-\delta) \int_X u^{2p+1} \Delta_\phi f_* e^{2a\log\vv} \omega_\phi^n
	  + C \int_X u^{2p} e^{-a(F+\log\vv+\delta f_* +b\phi)} e^{2a\log\vv} \omega_\phi^n \nonumber \\
	=&-\int_X u^{2p} e^{-a(F+\log\vv+\delta f_* +b\phi)} \Big( \frac{a\delta\Delta f_*}{1-a} + \frac{ab\Delta \phi}{1-a} \Big) e^{2a\log\vv} \omega_\phi^n \nonumber \\
	&-a(1-\delta) \int_X u^{2p+1} \Delta_\phi f_* e^{2a\log\vv} \omega_\phi^n+ C \int_X u^{2p} e^{-a(F+\log\vv+\delta f_* +b\phi)} e^{2a\log\vv} \omega_\phi^n \nonumber \\
	\leq& \int_X u^{2p} e^{-a(F+\log\vv+\delta f_* +b\phi)} \Big( C - \frac{ab}{1-a} \cdot \Delta\phi \Big) e^{2a\log\vv} \omega_\phi^n \nonumber \\
    &+ \int_X u^{2p} e^{-a(F+\log\vv+\delta f_* +b\phi)} \Big( -a(1-\delta)\Delta_\phi f_* \cdot (n+\Delta\phi) - \frac{a\delta}{1-a} \cdot \Delta f_* \Big) e^{2a\log\vv} \omega_\phi^n \nonumber \\
    \leq&C b \int_X u^{2p+1} e^{2a\log\vv} \omega_\phi^n
    +C\int_X u^{2p+1} \tr_\phi \omega e^{2a\log\vv} \omega_\phi^n
    \label{handle term II}
\end{align}
where we have used the following facts 
\begin{align*}
    C-\frac{ab}{1-a}\Delta \phi &\leq C + \frac{ab}{a-1} (n+\Delta\phi) -\frac{abn}{a-1} \leq C b (n+\Delta\phi),
\end{align*}
since $b$ is sufficiently large, and  
\begin{align*}
     & -a(1-\delta)\Delta_\phi f_* (n+\Delta\phi) + \frac{a}{a-1} \cdot \delta \cdot \Delta f_* =
     -\Delta_\phi f_* \cdot (n+\Delta\phi) + \Delta f_* \\
     =& -\sum_{i\neq j} \frac{(f_*)_{i\bar{i}}(1+\phi_{j\bar{j}})}{1+\phi_{i\bar{i}}} \leq (\frac{1-t}{t})\sum_{i\neq j}\frac{\theta_{i\bar{i}}(1+\phi_{j\bar{j}})}{1+\phi_{i\bar{i}}} \leq C\cdot\tr_\phi\omega\cdot (n+\Delta\phi),
\end{align*}
where we let $\delta = \frac{a-1}{a}$.

Choose $a=p+1$, then $(2p+1-a)-\frac{p^2}{a-1} =0$. 
By \eqref{interm. inequ involve II}, \eqref{handle term II} and choosing $b$ sufficiently large,
we have
\begin{align*}
    \frac{ab}{8} \int_X u^{2p+1} \tr_\phi \omega e^{2a\log\vv} \omega_\phi^n 
    \leq C' b \int_X u^{2p+1} e^{2a\log\vv} \omega_\phi^n.
\end{align*}
By 
\begin{align*}
		\tr\omega_\phi
    \leq \frac{1}{(n-1)!}(\tr_\phi\omega)^{n-1}\frac{\omega_\phi^n}{\omega^n},
	\end{align*}
we obtain
\begin{align*}
    \int_X u^{2p+1} (n+\Delta\phi)^{\frac{1}{n-1}} e^{\frac{-1}{n-1}(F-\log\vv)} e^{2a\log\vv} \omega_\phi^n 
    \leq C \int_X u^{2p+1} \omega_\phi^n
\end{align*}
 
It follows that
\begin{align*}
    & \int_X e^{(-a(2p+1) +\frac{n-2}{n-1})F - (2p+1)(a-1)f_* - ab(2p+1)\phi + (2a+\frac{1}{n-1}-a(2p+1))\log\vv} (n+\Delta\phi)^{2p+1+\frac{1}{n-1}} \omega^n \\
    &\leq C \int_X e^{(1-a(2p+1))F -a(2p+1)\log\vv -(2p+1)(a-1)f_* -ab(2p+1) \phi} (n+\Delta \phi)^{2p+1} \omega^n.
\end{align*}
Use $\|\phi\|_{C^0}, \|F+f_*\|_{C^0}, |\log\vv| \leq C$, then we derive
\begin{align*}
    \int_X e^{(2p+1-\frac{n-2}{n-1})f_*} (n+\Delta\phi)^{2p+1+\frac{1}{n-1}} \omega^n
    \leq C \int_X e^{2pf_*} (n+\Delta\phi)^{2p+1} \omega^n.
\end{align*}
Take $p=k\cdot\frac{1}{n-1}$, then the inductive argument concludes the proof.
 
\end{proof}

\subsection{Gradient estimate of \texorpdfstring{$F-\log\vv+f_*$}{} }

\begin{thm}
\label{gradient of F}
	Assume $\vv$ is log-concave. Let $\phi$ be a solution of \eqref{csck-mt},  then
	\begin{align*}
		\sup_X(|\nabla^\phi (F-\log\vv(m_\phi)+f_*)|_\phi)\leq C,
	\end{align*}
	where $C$ is a constant depending on $\omega$, $[\theta]$, $\vv$, $\ww$, $\|\phi\|_{C^0}$, $\|F+f_*\|_{C^0}$, $\|\tr\omega_\phi\|_{L^p}$ for some $p$.

 Furthermore, for any $p>0$, with a constant $C_p$ that depends on the same parameters above,
 \begin{align}
     \int_X |\nabla^\phi (F+f_*)|_\phi^P \omega^n < C_p.
 \end{align}
\end{thm}

\begin{remark}
    We want to emphasize that, the proof of Theorem \ref{gradient of F} only depends on the $C^0$-estimate Theorem \ref{C^0-estimate} and $W^{2,p}$-estimate Theorem \ref{integral C^2}, but not on the log-concavity assumption. We put the log-concavity assumption in the statement of Theorem \ref{gradient of F} only for its compatibility with our main result.
\end{remark}
	
\begin{proof}
Let $W = F'+f_* = F-\log\vv + f_*$.	

Let's begin by reviewing some preliminaries that are needed along the proof.
We have the following formulas and inequalities:
\begin{align*}
		\tr\omega_\phi
		\geq ne^{\frac{F'}{n}}
		\geq C>0,
	\end{align*}
which is due to the arithmetic-geometric inequality and the lower bound of $F'$,
\begin{align*}
    \Ric(\phi)_{i\bar{j}} 
    &=\Ric_{i\bar{j}}-F'_{i\bar{j}} 
    =\Ric_{i\bar{j}}-W_{i\bar{j}}+(f_*)_{i\bar{j}},\\
    \tr_\phi \omega &\leq n^{-1} e^{-F'} (\tr\omega_\phi)^{n-1}
    \leq C(\tr\omega_\phi)^{n-1} ,\\
    |\Ric(\omega)(\nabla^\phi W,\nabla^\phi W)| &\leq C \cdot e^{-F'} \cdot (\tr \omega_\phi)^{n-1}\cdot |\nabla^\phi W|_\phi^2,
\end{align*}
\begin{align*}
    \Ric(\phi)_{i\bar{j}}W_{\bar{i}}W_{j} &= \Ric(\omega)(\nabla^\phi W,\nabla^\phi W) - \frac{W_{i\bar{j}}W_{\bar{i}}W_{j}}{(1+\phi_{i\bar{i}})(1+\phi_{j\bar{j}})} + \frac{(f_*)_{i\bar{j}}W_{\bar{i}}W_{j}}{(1+\phi_{i\bar{i}})(1+\phi_{j\bar{j}})} \\
    &\geq -C e^{-F'}(\tr \omega_\phi)^{n-1} |\nabla^\phi W|_\phi^2 - \frac{W_{i\bar{j}}W_{\bar{i}}W_{j}}{(1+\phi_{i\bar{i}})(1+\phi_{j\bar{j}})} .
\end{align*}
Since $f_*$ is a $(\frac{1-t}{t})\theta$-psh function , $\frac{\vv_\alpha}{\vv} J \xi^\alpha (f_*)$ is determined by class $[\theta]$, which is bounded. And by Theorem \ref{C^0-estimate}, $|W|_{C^0}$ is bounded.
With the preliminaries in mind and under an orthonormal
frame of $g_\phi$, using Bochner formula (see the proof of \cite[Proposition 4.1]{CC21a}) implies
\begin{align}
    e^{\frac{-W}{2}} \Delta_\phi (e^{\frac{W}{2}} |\nabla^\phi W|_\phi^2) 
    \geq& (\Delta_\phi W)_i W_{\bar{i}} 
    +(\Delta_\phi W)_{\bar{i}} W_i + \Ric(\phi)_{i\bar{j}} W_{\bar{i}} W_{j} + |W_{i\bar{j}}|_\phi^2
                       +|W_{ij}|_\phi^2 \nonumber \\
    & +\frac{1}{4} |\nabla^\phi W|_\phi^4 + \frac{1}{2}(W_iW_{\bar{j}}W_{\bar{i}j} + W_{\bar{i}}W_{j}W_{i\bar{j}}+W_iW_jW_{\bar{i}\bar{j}}+W_{\bar{i}}W_{\bar{j}}W_{ij}) \nonumber \\
    & + \frac{1}{2} \Delta_\phi W \cdot |\nabla^\phi W|_\phi^2 \nonumber \\
    \geq& (\Delta_\phi W)_i W_{\bar{i}} + (\Delta_\phi W)_{\bar{i}}W_i + |W_{i\bar{j}}|_\phi^2 \nonumber \\
    & + \frac{1}{2}\Delta_\phi W \cdot |\nabla^\phi W|_\phi^2 - C e^{-F'}(\tr \omega_\phi)^{n-1}\cdot |\nabla^\phi W|_\phi^2 \nonumber \\
    \geq& 2\nabla^\phi W \cdot_\phi \nabla^\phi \Delta_\phi W + \frac{|W_{i\bar{j}}|^2}{(1+\phi_{i\bar{i}})(1+\phi_{j\bar{j}})} - C e^{-F'}(\tr \omega_\phi)^{n-1}\cdot |\nabla^\phi W|_\phi^2 \nonumber \\
    & + \frac{1}{2}\big( -C - C e^{-F'}(\tr \omega_\phi)^{n-1} + 2\frac{\vv_\alpha}{\vv} J\xi^\alpha (W) \Big) |\nabla^\phi W|_\phi^2 \nonumber \\
    \geq& 2\nabla^\phi W \cdot_\phi \nabla^\phi \Delta_\phi W + \frac{|W_{i\bar{j}}|^2}{(1+\phi_{i\bar{i}})(1+\phi_{j\bar{j}})} \nonumber \\
    & -C(1 + e^{-F'}(\tr \omega_\phi)^{n-1})\cdot |\nabla^\phi W|_\phi^2 
    + \frac{\vv_\alpha}{\vv} J \xi^\alpha (W) |\nabla^\phi W|_\phi^2.
    \label{Lap of e^W and gradient of W}
\end{align}

Note that \eqref{csck-mt} and \eqref{Lap of logv(m_phi) and Jxi(F')},
\begin{align}
    \Delta_\phi W =& \Delta_\phi(F+f_*)-\Delta_\phi \log\vv \nonumber \\ 
    =& \frac{\vv_{,\alpha}}{\vv}J\xi^\alpha(F+f_*-\log\vv)
    +\frac{\vv_{,\alpha}}{\vv}J\xi^\alpha\log\vv-\Delta_\phi\log\vv
    +\tilde{f}_t+\tr_\phi\eta_t \nonumber  \\
    =&\frac{\vv_{,\alpha}}{\vv}J\xi^\alpha(W)+\frac{\vv_{,\alpha}}{\vv}J\xi^\alpha(F'+f_*)-\frac{\vv_{,\alpha}}{\vv}J\xi^\alpha(f_*) 
    +\left(\frac{\vv_{,\alpha}\vv_{,\beta}}{\vv^2}-\frac{\vv_{,\alpha\beta}}{\vv}\right)\langle\xi^\alpha,\xi^\beta\rangle_\phi
     \nonumber \\
      &-\frac{\vv_{,\alpha}(m_\phi)}{\vv(m_\phi)}\Delta_\omega(m_\omega^{\xi^\alpha})
      -\frac{\vv_{,\alpha}\vv_{,\beta}}{\vv^2}\langle\xi^\alpha,\xi^\beta\rangle_\phi   
      +\tilde f_t+\tr_\phi\eta_t \nonumber, 
\end{align}
which together with   
\eqref{upper bound of xi pair} implies
\begin{align}
\label{Laplace W}
      2 \frac{\vv_\alpha}{\vv}J\xi^\alpha(W) - C -C \tr_\phi\omega -C\tr\omega_\phi \leq \Delta_\phi W \leq 2 \frac{\vv_\alpha}{\vv}J\xi^\alpha(W) + C +C \tr_\phi\omega +C\tr\omega_\phi 
\end{align}

To control the derivative term $J\xi^\alpha(W)|\nabla^\phi W|_\phi^2$, we need to add another barrier function.
	One computes
	\begin{align}
		&\Delta_\phi(e^{-\log\vv(m_\phi)} \cdot e^{\frac{W}{2}}|\nabla^\phi W|_\phi^2)
		 \nonumber \\
		=&g_\phi^{i\bar j}\p_i\left(e^{-\log\vv(m_\phi)}(-(\log\vv(m_\phi))_{\bar j})(e^{\frac{W}{2}}|\nabla^\phi W|_\phi^2)
		 +e^{-\log\vv(m_\phi)}\p_{\bar j}(e^{\frac{W}{2}}|\nabla^\phi W|_\phi^2) \right) \nonumber \\
		=&e^{\frac{W}{2}-\log\vv(m_\phi)}|\nabla^\phi W|_\phi^2|\nabla^\phi\log\vv(m_\phi)|_\phi^2
		 +e^{-\log\vv(m_\phi)}(-\Delta_\phi\log\vv(m_\phi))e^{\frac{W}{2}}|\nabla^\phi W|_\phi^2 \nonumber \\
		&-g_\phi^{i\bar j}e^{-\log\vv(m_\phi)}(\log\vv(m_\phi))_{\bar j}(e^{\frac{W}{2}}|\nabla^\phi W|_\phi^2)_i 
		  -g_\phi^{i\bar j}e^{-\log\vv(m_\phi)}(\log\vv(m_\phi))_{i}(e^{\frac{W}{2}}|\nabla^\phi W|_\phi^2)_{\bar j} \nonumber \\
		&+e^{-\log\vv(m_\phi)}\Delta_\phi(e^{\frac{W}{2}}|\nabla^\phi W|_\phi^2) \nonumber \\
		=&e^{\frac{W}{2}-\log\vv(m_\phi)}|\nabla^\phi W|_\phi^2|\nabla^\phi\log\vv(m_\phi)|_\phi^2
		 -2e^{-\log\vv(m_\phi)}\nabla^\phi\log\vv(m_\phi)\cdot_\phi\nabla^\phi(e^{\frac{W}{2}}|\nabla^\phi W|_\phi^2) \nonumber \\
		 &+e^{-\log\vv(m_\phi)}\left(\left(\frac{\vv_{,\alpha}(m_\phi)}{\vv(m_\phi)}J\xi^\alpha(W)-\frac{\vv_{,\alpha}(m_\phi)}{\vv(\phi)}J\xi^\alpha(f_*) -\frac{\vv_{,\alpha}(m_\phi)}{\vv(m_\phi)}\Delta_\omega(m_\omega^{\xi^\alpha})
          \right. \right.  \nonumber \\
	     & \left.\left.	     
	      -\frac{\vv_{,\alpha\beta}(m_\phi)}{\vv(m_\phi)}\langle\xi^\alpha,\xi^\beta\rangle_\phi+\frac{\vv_{,\alpha}(m_\phi)\vv_{,\beta}(m_\phi)}{\vv(m_\phi)^2}\langle\xi^\alpha,\xi^\beta\rangle_\phi\right)e^{\frac{W}{2}}|\nabla^\phi W|_\phi^2+\Delta_\phi(e^{\frac{W}{2}}|\nabla^\phi W|_\phi^2)\right) \nonumber \\
		 \geq&e^{\frac{W}{2}-\log\vv(m_\phi)}|\nabla^\phi W|_\phi^2|\nabla^\phi\log\vv(m_\phi)|_\phi^2
		 -2e^{-\log\vv(m_\phi)}\nabla^\phi\log\vv(m_\phi)\cdot_\phi\nabla^\phi(e^{\frac{W}{2}}|\nabla^\phi W|_\phi^2) \nonumber \\
		 &+e^{\frac{W}{2}-\log\vv(m_\phi)}2\nabla^\phi W\cdot_\phi\nabla^\phi\Delta_\phi W -C_{\vv,\xi}(1+(\tr\omega_\phi)^{n-1})e^{\frac{W}{2}-\log\vv(m_\phi)}|\nabla^\phi W|_\phi^2 \nonumber \\
		 &+\frac{|W_{i\bar j}|^2}{(1+\phi_{i\bar i})(1+\phi_{j\bar j})}e^{\frac{F'}{2}-\log\vv(m_\phi)} 
		 +2\frac{\vv_{,\alpha}(m_\phi)}{\vv(m_\phi)}J\xi^\alpha(W)e^{\frac{W}{2}-\log\vv(m_\phi)}|\nabla^\phi W|_\phi^2,
		 \label{Lap for the new barrier modified}
	\end{align}
	where we use \eqref{Lap of logv(m_phi) and Jxi(F')} for the third equality, \eqref{upper bound of xi pair} and \eqref{Lap of e^W and gradient of W} for the last inequality.

Set
	\begin{align*}
		u=e^{\frac{W}{2}-\log\vv(m_\phi)}|\nabla^\phi W|_\phi^2+1,
	\end{align*}
	then we have
	\begin{align}
		\Delta_\phi u 
		\geq&-C_{\vv,\xi}(\tr\omega_\phi)^{n-1}u +2e^{\frac{W}{2}-\log\vv(m_\phi)}\nabla^\phi W\cdot_\phi\nabla^\phi\Delta_\phi W 
		   +2\frac{\vv_{,\alpha}(m_\phi)}{\vv(m_\phi)}J\xi^\alpha(W)u
		    \nonumber \\
		&-2\frac{\vv_{,\alpha}}{\vv}J\xi^\alpha(W)+e^{\frac{W}{2}-\log\vv(m_\phi)}|\nabla^\phi W|_\phi^2|\nabla^\phi\log\vv(m_\phi)|_\phi^2 \nonumber \\
		 &-2e^{-\log\vv(m_\phi)}\nabla^\phi\log\vv(m_\phi)\cdot_\phi\nabla^\phi(e^{\frac{W}{2}}|\nabla^\phi W|_\phi^2).
		 \label{Laplace u modified}
	\end{align}
	For any $p>0$, 
	\begin{align*}
		\frac{1}{2p+1}\Delta_\phi(u^{2p+1})
		=u^{2p}\Delta_\phi u+2pu^{2p-1}|\nabla^\phi u|_\phi^2.
	\end{align*}
	Taking the integration with respect to the weighted measure $e^{2\log\vv(m_\phi)}\omega_\phi^n$, one has
	\begin{align}
	\label{integ formula initial modified}
		&\int_X2pu^{2p-1}|\nabla^\phi u|_\phi^2e^{2\log\vv(m_\phi)}\omega_\phi^n \nonumber \\
		=&\int_Xu^{2p}(-\Delta_\phi u)e^{2\log\vv(m_\phi)}\omega_\phi^n
		+\frac{1}{2p+1}\int_X\Delta_\phi(u^{2p+1})e^{2\log\vv(m_\phi)}\omega_\phi^n. 
	\end{align}
	On the one hand,
	\begin{align}
	\label{integ of u^{2p+1} modified}
		&\frac{1}{2p+1}\int_X\Delta_\phi(u^{2p+1})e^{2\log\vv(m_\phi)}\omega_\phi^n \nonumber \\
		=&-\frac{1}{2p+1}\int_X\nabla^\phi(u^{2p+1})\cdot_\phi\nabla^\phi e^{2\log\vv(m_\phi)}\omega_\phi^n \nonumber \\
		=&-\int_X2u^{2p}\nabla^\phi(e^{-\log\vv(m_\phi)}e^{\frac{W}{2}}|\nabla^\phi W|_\phi^2)\cdot_\phi\nabla^\phi\log\vv(m_\phi) e^{2\log\vv(m_\phi)}\omega_\phi^n \nonumber \\
		=&-2\int_Xu^{2p}\nabla^\phi(e^{\frac{W}{2}}|\nabla^\phi W|_\phi^2)\cdot_\phi\nabla^\phi\log\vv(m_\phi)e^{\log\vv(m_\phi)}\omega_\phi^n \nonumber \\
		&+2\int_Xu^{2p}e^{\frac{W}{2}}|\nabla^\phi W|_\phi^2|\nabla^\phi\log\vv(m_\phi)|_\phi^2e^{\log\vv(m_\phi)}\omega_\phi^n.  
	\end{align}
	On the other hand, by \eqref{Laplace u modified},
	\begin{align}
		&\int_Xu^{2p}(-\Delta_\phi u)e^{2\log\vv(m_\phi)}\omega_\phi^n \nonumber \\
		\leq&C\int_Xu^{2p+1}(\tr\omega_\phi)^{n-1}e^{2\log\vv(m_\phi)}\omega_\phi^n
		-2\int_Xu^{2p+1}\frac{\vv_{,\alpha}(m_\phi)}{\vv(m_\phi)}J\xi^\alpha(W)e^{2\log\vv(m_\phi)}\omega_\phi^n \nonumber \\
		&-\int_Xu^{2p} e^{\frac{W}{2}}|\nabla^\phi W|_\phi^2|\nabla^\phi\log\vv(m_\phi)|_\phi^2e^{\log\vv(m_\phi)}\omega_\phi^n  \nonumber \\
		 &+2\int_Xu^{2p} \nabla^\phi\log\vv(m_\phi)\cdot_\phi\nabla^\phi(e^{\frac{W}{2}}|\nabla^\phi W|_\phi^2)e^{\log\vv(m_\phi)}\omega_\phi^n \nonumber \\
		 &-2\int_Xu^{2p} e^{\frac{W}{2}}\nabla^\phi W\cdot_\phi\nabla^\phi\Delta_\phi W e^{\log\vv(m_\phi)}\omega_\phi^n.
		 \label{integ of Laplace u modified}
	\end{align}
    where we have used 
    \begin{align}
    \label{upperbound of xi(W)}
            |J\xi^\alpha(W)|
            \leq C|\nabla W|
            \leq C(\tr\omega_\phi)^{\frac{1}{2}}|\nabla^\phi W|_\phi
            \leq C(\tr\omega_\phi)^{\frac{1}{2}}C\sqrt{u} 
         \leq C(\tr\omega_\phi)^{\frac{1}{2}}u
         \leq C(\tr\omega_\phi)^{n-1}u.
        \end{align}
 
    By \eqref{integ of u^{2p+1} modified} and \eqref{integ of Laplace u modified}, we have
    \begin{align}
    \label{middle step terms}
       &\int_X2pu^{2p-1}|\nabla^\phi u|_\phi^2e^{2\log\vv(m_\phi)}\omega_\phi^n \nonumber \\
       \leq&C\int_Xu^{2p+1}(\tr\omega_\phi)^{n-1}e^{2\log\vv(m_\phi)}\omega_\phi^n 
       -2\int_Xu^{2p+1}\frac{\vv_{,\alpha}(m_\phi)}{\vv(m_\phi)}J\xi^\alpha(W)e^{2\log\vv(m_\phi)}\omega_\phi^n \nonumber \\
       &-2\int_Xu^{2p} e^{\frac{W}{2}}\nabla^\phi W\cdot_\phi\nabla^\phi\Delta_\phi W e^{\log\vv(m_\phi)}\omega_\phi^n.
    \end{align}
 
	Using the integration by parts,
	\begin{align}
		&-2\int_Xu^{2p} e^{\frac{W}{2}}\nabla^\phi W\cdot_\phi\nabla^\phi\Delta_\phi W e^{\log\vv(m_\phi)}\omega_\phi^n \nonumber \\
		=&\int_X4pu^{2p-1}e^{\frac{W}{2}}\Delta_\phi W(\nabla^\phi W\cdot_\phi\nabla^\phi u) e^{\log\vv(m_\phi)}\omega_\phi^n 
		+\int_X2u^{2p}e^{\frac{W}{2}}(\Delta_\phi W)^2e^{\log\vv(m_\phi)}\omega_\phi^n \nonumber \\
		&+\int u^{2p}e^{\frac{W}{2}}|\nabla^\phi W|_\phi^2(\Delta_\phi W)e^{\log\vv(m_\phi)}\omega_\phi^n \nonumber \\
		&+2\int_Xu^{2p}e^{\frac{W}{2}}(\Delta_\phi W)(\nabla^\phi W\cdot_\phi\nabla^\phi\log\vv(m_\phi))e^{\log\vv(m_\phi)}\omega_\phi^n. 
    \label{integ by part in CC21 modified}
	\end{align}
	We estimate the four terms on the right side of \eqref{integ by part in CC21 modified}.
	First,
	\begin{align}
		&\int_X4pu^{2p-1}e^{\frac{W}{2}}\Delta_\phi W(\nabla^\phi W\cdot_\phi\nabla^\phi u) e^{\log\vv(m_\phi)}\omega_\phi^n \nonumber \\
		\leq& \int_Xpu^{2p-1}|\nabla^\phi u|_\phi^2e^{2\log\vv(m_\phi)}\omega_\phi^n 
		+\int_X4pu^{2p-1}e^{W}(\Delta_\phi W)^2|\nabla^\phi W|_\phi^2\omega_\phi^n \nonumber \\
		\leq& \int_Xpu^{2p-1}|\nabla^\phi u|_\phi^2e^{2\log\vv(m_\phi)}\omega_\phi^n 
		+\int_X4pu^{2p}e^{\frac{W}{2}}(\Delta_\phi W)^2e^{\log\vv(m_\phi)}\omega_\phi^n.
		\label{1st estimate of integ by part modified}
	\end{align}
	By formula \eqref{Laplace W},
	\begin{align}
		&\int u^{2p}e^{\frac{W}{2}}|\nabla^\phi W|_\phi^2(\Delta_\phi W)e^{\log\vv(m_\phi)}\omega_\phi^n \nonumber \\
		\leq&\int u^{2p}e^{\frac{W}{2}}|\nabla^\phi W|_\phi^2\left(C+C\tr_\phi\omega+C\tr\omega_\phi+2\frac{\vv_{,\alpha}(m_\phi)}{\vv(m_\phi)}J\xi^\alpha(W)\right)e^{\log\vv(m_\phi)}\omega_\phi^n \nonumber \\
		=&2\int_Xu^{2p+1}\frac{\vv_{,\alpha}(m_\phi)}{\vv(m_\phi)}J\xi^\alpha(W)e^{2\log\vv(m_\phi)}\omega_\phi^n
		-2\int_Xu^{2p}\frac{\vv_{,\alpha}(m_\phi)}{\vv(m_\phi)}J\xi^\alpha(W)e^{2\log\vv(m_\phi)}\omega_\phi^n
		  \nonumber \\
		&
		 +\int u^{2p}e^{\frac{W}{2}}|\nabla^\phi W|_\phi^2(C+C\tr_\phi(\omega)+C\tr\omega_\phi)e^{\log\vv(m_\phi)}\omega_\phi^n
         \nonumber \\
		\leq& 2\int_Xu^{2p+1}\frac{\vv_{,\alpha}(m_\phi)}{\vv(m_\phi)}J\xi^\alpha(W)e^{2\log\vv(m_\phi)}\omega_\phi^n
		+C_{\vv,\xi}\int_Xu^{2p+1}(\tr\omega_\phi)^{n-1} e^{2\log\vv(m_\phi)}\omega_\phi^n,
		\label{2nd estimate of integ by parts modified}
	\end{align}
        where we have used \eqref{upperbound of xi(W)}
        for the last inequality above.
	Again by formula \eqref{Laplace W},
	\begin{align*}
		|\Delta_\phi W|\leq C+B_{\vv,\xi}(\tr\omega_\phi)^{n-1} +2\left|\frac{\vv_{,\alpha}(m_\phi)}{\vv(m_\phi)} \right||J\xi^\alpha(W)|
		\leq& B_{\vv,\xi}(\tr\omega_\phi)^{n-1}+C_\vv|\nabla W| \\
		\leq& B_{\vv,\xi}(\tr\omega_\phi)^{n-1}+C_\vv(\tr\omega_\phi)^{\frac{1}{2}}|\nabla^\phi W|_\phi.
	\end{align*}
	where $B_{\vv,\xi}$ is a positive constant such that
	\begin{align*}
		|\tr_\phi\eta_t|+\left|\left(\frac{\vv_{,\alpha}(m_\phi)\vv_{\beta}(m_\phi)}{\vv(m_\phi)^2}-\frac{\vv_{,\alpha\beta}(m_\phi)}{\vv(m_\phi)}\right)\langle\xi^\alpha,\xi^\beta\rangle_\phi \right|
		\leq B_{\vv,\xi}(\tr\omega_\phi)^{n-1}.
	\end{align*}
	It follows that
	\begin{align}
	\label{3rd estimate of integ by parts modified}
		&\int_Xu^{2p}e^{\frac{W}{2}}(\Delta_\phi W)^2e^{\log\vv(m_\phi)}\omega_\phi^n \nonumber \\
		\leq&\int_Xu^{2p}e^{\frac{W}{2}}C\Big((\tr\omega_\phi)^{2n-2}+(\tr\omega_\phi)^{n-1}\sqrt{\tr\omega_\phi}|\nabla^\phi W|_\phi+\tr\omega_\phi|\nabla^\phi W|_\phi^2\Big)e^{\log\vv(m_\phi)}\omega_\phi^n \nonumber \\
		\leq&\int_XC(\tr\omega_\phi)^{2(n-1)}u^{2p+1}e^{2\log\vv(m_\phi)}\omega_\phi^n. 
	\end{align}
	By H\"older inequality,
	\begin{align}
		&2\int_Xu^{2p}e^{\frac{W}{2}}(\Delta_\phi W)(\nabla^\phi W\cdot_\phi\nabla^\phi\log\vv(m_\phi))e^{\log\vv(m_\phi)}\omega_\phi^n \nonumber \\
		\leq&\int_Xu^{2p}e^{\frac{W}{2}}(\Delta_\phi W)^2|\nabla^\phi\log\vv(m_\phi)|_\phi^2 e^{\log\vv(m_\phi)}\omega_\phi^n
		+\int_Xu^{2p}e^{\frac{W}{2}}|\nabla^\phi W|_\phi^2e^{\log\vv(m_\phi)}\omega_\phi^n \nonumber \\
		\leq&\int_XC(\tr\omega_\phi)^{3(n-1)}u^{2p+1}e^{2\log\vv(m_\phi)}\omega_\phi^n
		+\int_Xu^{2p+1}e^{2\log\vv(m_\phi)}\omega_\phi^n,
		\label{4th estimate of integ by parts modified}
	\end{align}
	where we have used \eqref{3rd estimate of integ by parts modified} and 
    \begin{align}
\label{upperbound of |nabla logv|}
		|\nabla^\phi\log\vv(m_\phi)|_\phi^2 
		=\frac{\vv_{,\alpha}(m_\phi)\vv_{,\beta}(m_{\phi})}{\vv(m_\phi)^2}\langle\xi^\beta,\xi^\alpha\rangle_\phi
		\leq C_\vv(1+\phi_{i\bar i})\xi_{\alpha,i}\xi_{\beta,\bar i}
		\leq C_{\vv,\xi}\tr\omega_\phi.
	\end{align}
	
	By \eqref{integ by part in CC21 modified}, \eqref{1st estimate of integ by part modified}, \eqref{2nd estimate of integ by parts modified}, \eqref{3rd estimate of integ by parts modified} and \eqref{4th estimate of integ by parts modified}, we obtain
	\begin{align}
    \label{W terms}
    \begin{split}
		&-2\int_Xu^{2p} e^{\frac{W}{2}}\nabla^\phi W\cdot_\phi\nabla^\phi\Delta_\phi W e^{\log\vv(m_\phi)}\omega_\phi^n \\
		\leq& \int_Xpu^{2p-1}|\nabla^\phi u|_\phi^2e^{2\log\vv(m_\phi)}\omega_\phi^n 
		+C(p+1)\int_X(\tr\omega_\phi)^{3(n-1)}u^{2p+1}e^{2\log\vv(m_\phi)}\omega_\phi^n \\
		&+2\int_Xu^{2p+1}\frac{\vv_{,\alpha}(m_\phi)}{\vv(m_\phi)}J\xi^\alpha(W)e^{2\log\vv(m_\phi)}\omega_\phi^n
		+\int_Xu^{2p+1}e^{2\log\vv(m_\phi)}\omega_\phi^n
  \end{split}
	\end{align}
	Therefore, by \eqref{middle step terms}, \eqref{W terms},  we deduce
	\begin{align*}
		\int_X p u^{2p-1}|\nabla^\phi u|_\phi^2e^{2\log\vv(m_\phi)}\omega_\phi^n 
		\leq
		C(p+1)\int_X(\tr\omega_\phi)^{3(n-1)}u^{2p+1}e^{\log\vv(m_\phi)}\omega_\phi^n.
	\end{align*}
	It follows that
	\begin{align*}
		\int_X|\nabla^\phi(u^{p+\frac{1}{2}})|_\phi^2\omega^n
		\leq\frac{C_{\vv,\xi}(p+\frac{1}{2})^2(p+1)}{p}\int_XC_\xi(\tr\omega_\phi)^{3(n-1)}u^{2p+1}\omega_\phi^n.
	\end{align*}
    In the inequality above, we use $\omega^n = e^{f_*}\cdot e^{-(F'+f_*)}\omega_\phi^n\leq C \omega_\phi^n$, since $F'+f_*$ is uniformly bounded, and $f_*\leq 0$.
 
	For any $0<\eps<2$, by H\"older inequality, we have
	\begin{align*}
		\int_X(\tr\omega_\phi)^{3(n-1)}u^{2p+1}\omega_\phi^n
        &=\int_X(\tr\omega_\phi)^{3(n-1)}u^{2p+1}e^{-f_*}e^{F'+f_*}\omega^n \\
		&\leq C\left(\int_Xu^{(p+\frac{1}{2})(2+\eps)}\omega^n\right)^{\frac{2}{2+\eps}} \left(\int_X(\tr\omega_\phi)^{3(n-1)\frac{2+\eps}{\eps}}e^{-\frac{2+\eps}{\eps}f_*}\omega^n \right)^{\frac{\eps}{2+\eps}}.
	\end{align*}
	Denote $v:=u^{p+\frac{1}{2}}$, then the above inequalities imply
	\begin{align*}
		\int_X|\nabla^\phi v|_\phi^2\omega^n
		\leq\frac{C_{\vv,\xi}(p+\frac{1}{2})^3}{p}
		\left(\int_Xv^{(2+\eps)}\omega^n\right)^{\frac{2}{2+\eps}} \left(\int_X(\tr\omega_\phi)^{3(n-1)\frac{2+\eps}{\eps}}e^{-\frac{2+\eps}{\eps}f_*}\omega^n \right)^{\frac{\eps}{2+\eps}}
	\end{align*}
	Similar with (4.26) in the proof of \cite[Proposition 4.1]{CC21a}, we also have
	\begin{align*}
		\left(\int_X|\nabla v|^{2-\eps}\omega^n\right)^{\frac{2}{2-\eps}}
		\leq& n^{\frac{\eps}{2-\eps}}\left(\int_X(\tr\omega_\phi)^{\frac{2-\eps}{\eps}}\omega^n \right)^{\frac{\eps}{2-\eps}} 
		   \int_X|\nabla^\phi v|_\phi^2\omega^n \\
		\leq& C_{\vv,\xi}p^2K_\eps\left(\int_Xv^{(2+\eps)}\omega^n\right)^{\frac{2}{2+\eps}}, 
	\end{align*}
	where
	\begin{align*}
		K_\eps:=n^\frac{\eps}{2-\eps}\left(\int_X(\tr\omega_\phi)^{\frac{2-\eps}{\eps}}\omega^n \right)^{\frac{\eps}{2-\eps}} \left(\int_X(\tr\omega_\phi)^{3(n-1)\frac{2+\eps}{\eps}}e^{-\frac{2+\eps}{\eps}f_*}\omega^n \right)^{\frac{\eps}{2+\eps}}
	\end{align*}
	is bounded, which can be seen by using Theorem \ref{integral C^2}, the assumption of $f_*$ and H\"older inequality.
	
	Applying the Sobolev inequality (\cite[Theorem 2.21]{Aub13})  with exponent $2-\eps$, and denote $\theta:=\frac{2n(2-\eps)}{2n-2+\eps}$ to be the index, we obtain
	\begin{align*}
		\|v\|_{L^\theta(\omega^n)}
		\leq C_s(\|\nabla v\|_{L^{2-\eps}(\omega^n)}+\|v\|_{L^{2-\eps}(\omega^n)}).
	\end{align*}
	It follows that
	\begin{align*}
		\left(\int_Xu^{(p+\frac{1}{2})\theta}\omega^n \right)^{\frac{1}{\theta}}
		\leq& C_s\left( \Big(\int_X|\nabla(u^{p+\frac{1}{2}})|^{2-\eps}\omega^n\Big)^{\frac{1}{2-\eps}} +\Big(\int_X u^{(p+\frac{1}{2})(2-\eps)}\omega^n\Big)^{\frac{1}{2-\eps}} \right) \\
		\leq&  C_s\left(C_{\vv,\xi}p\sqrt{K_\eps}\Big(\int_Xu^{(p+\frac{1}{2})(2+\eps)}\omega^n\Big)^{\frac{1}{2+\eps}} 
		 +\Big(\int_X u^{(p+\frac{1}{2})(2-\eps)}\omega^n\Big)^{\frac{1}{2-\eps}} \right) \\
		\leq& C_\eps p\left(\int_Xu^{(p+\frac{1}{2})(2+\eps)}\omega^n\right)^{\frac{1}{2+\eps}}. 
	\end{align*}
	One can choose $\eps>0$ sufficiently small such that $\theta>(2+\eps)$ (e.g. $\eps=\frac{1}{2n}$). 
	We deduce
	\begin{align*}
		\|u\|_{L^{(p+\frac{1}{2})\theta}}
		\leq (C_\eps p)^{\frac{1}{p+\frac{1}{2}}}\|u\|_{L^{(p+\frac{1}{2})(2+\eps)}}.
	\end{align*}
	Denote $\chi:=\frac{\theta}{2+\eps}>1$ and choose $p+\frac{1}{2}=\chi^i$, then we have
	\begin{align*}
		\|u\|_{L^{(2+\eps)\chi^{i+1}}}
		\leq C^{\frac{1}{\chi^i}}(\chi^i)^{\frac{1}{\chi^i}}\|u\|_{L^{(2+\eps)\chi^i}}.
	\end{align*}
	It follows that
	\begin{align*}
		\|u\|_{L^\infty}
		\leq C^{\sum_{i\geq0}\frac{1}{\chi^i}}\chi^{\sum_{i\geq0}\frac{i}{\chi^i}}
		  \|u\|_{L^{(2+\eps)\chi^i}}
	\end{align*}
	Then we
	\begin{align}
	\label{L^infty bound of u modified}
		\|u\|_{L^\infty}
		\leq C\|u\|_{L^{2+\eps}}
		\leq C\|u\|_{L^1}^{\frac{1}{2+\eps}}\|u\|_{L^\infty}^{\frac{1+\eps}{2+\eps}}.
	\end{align}
	Thus, by \eqref{L^infty bound of u modified}, Theorem \ref{integral C^2} and Lemma \ref{L^2 bound of gradient W squre modified}, we derive the $\|u\|_{L^\infty}$,
which implies $\sup_X |\nabla^\phi(F+f_*-\log(\vv(m_\phi)))|_\phi \leq C$.

Furthermore, since $|\nabla^\phi \log\vv |_\phi^2 \leq C\cdot \tr\omega_\phi$, by Theorem \ref{integral C^2}, we have
\begin{align}
     \int_X |\nabla^\phi (F+f_*)|_\phi^P \omega^n < C_p.
 \end{align}
\end{proof}

\begin{lemma}
	\label{L^2 bound of gradient W squre modified}
	As above notation, the following estimate holds
	\begin{align}
		\|u\|_{L^1}\leq C.
	\end{align}
\end{lemma}
\begin{proof}
	Recall $u=e^{\frac{W}{2}-\log(\vv(m_\phi))}|\nabla^\phi W|_\phi^2+1$. 
	It suffices to show that
	\begin{align*}
		\int_X|\nabla^\phi W|_\phi^2\omega_\phi^n\leq C.
	\end{align*}
	By \eqref{Lap of logv(m_phi) and Jxi(F')} and \eqref{csck-mt},  one observes that
	\begin{align*}
		\Delta_\phi(F+f_*+\log\vv(m_\phi))
		=&\frac{\vv_{,\alpha}}{\vv}J\xi^\alpha(F+f_*)+\Delta_\phi\log(\vv(m_\phi))+\tilde{f}_t+\tr_\phi\tilde\eta_t \\
		=&\frac{\vv_{,\alpha}}{\vv}J\xi^\alpha(f_*)
		+(\frac{\vv_{,\alpha\beta}}{\vv}-2\frac{\vv_{,\alpha}\vv_{\beta}}{\vv^2})\langle\xi^\alpha,\xi^\beta\rangle_\phi +\tilde{f}_t+\tr_\phi\tilde\eta_t.
	\end{align*}
	Denote $\tilde W=F+f_*+\log\vv(m_\phi)=W+2\log\vv(m_\phi)$.
	One has
	\begin{align*}
		\Delta_\phi(\tilde W^2)
		=&2\tilde W\Delta_\phi\tilde W+2|\nabla^\phi\tilde W|_\phi^2 \\
		=&2\tilde W\left(\frac{\vv_{,\alpha}}{\vv}J\xi^\alpha(f_*)
		+(\frac{\vv_{,\alpha\beta}}{\vv}-2\frac{\vv_{,\alpha}\vv_{\beta}}{\vv^2})\langle\xi^\alpha,\xi^\beta\rangle_\phi +\tilde{f}_t+\tr_\phi\tilde\eta_t\right)
		+2|\nabla^\phi\tilde W|_\phi^2.
	\end{align*}
	It follows that
	\begin{align*}
		\int_X|\nabla^\phi\tilde W|_\phi^2\omega_\phi^n
		\leq&\int_X\tilde W(\tilde{f}_t+\frac{\vv_{,\alpha}}{\vv}J\xi^\alpha(f_*))\omega_\phi^n
			+\int_X\tilde W\left(
		(\frac{\vv_{,\alpha\beta}}{\vv}-2\frac{\vv_{,\alpha}\vv_{\beta}}{\vv^2})\langle\xi^\alpha,\xi^\beta\rangle_\phi+\tr_\phi\tilde\eta_t
\right)\omega_\phi^n \\
		\leq&C+\|\tilde W\|_{C^0}B_{\vv,\xi}\int_X(\tr\omega_\phi+\tr_\phi\omega)\omega_\phi^n \\
		\leq&C+C\int_X(\tr\omega_\phi)^{n-1}\omega_\phi^n
		\leq C.
	\end{align*}
	Therefore, by Theorem \ref{integral C^2} and \eqref{upperbound of |nabla logv|}, we obtain
	\begin{align*}
		\int_X|\nabla^\phi W|_\phi^2\omega_\phi^n
		=&\int_X|\nabla^\phi(\tilde W-2\log\vv(m_\phi))|_\phi^2\omega_\phi^n \\
		\leq&\int_X|\nabla^\phi\tilde W|_\phi^2\omega_\phi^n
		+2\int_X|\nabla^\phi\log\vv(m_\phi)|_\phi^2\omega_\phi^n \\
		\leq& C+4\int_X\frac{\vv_{,\alpha}(m_\phi)\vv_{,\beta}(m_{\phi})}{\vv(m_\phi)^2}\langle\xi^\beta,\xi^\alpha\rangle_\phi\omega_\phi^n \\
		\leq&C+C_{\vv,\xi}\int_X\tr\omega_\phi\omega_\phi^n 
		\leq C.
	\end{align*}
\end{proof}
\section{Proof of Continuity Method}
\label{Proof of Continuity Method}
In this section, we prove the Theorem \ref{main-thm} by finishing the continuity method, concluding the openness and closedness of continuity path.

The weighted Mabuchi functional associated to the continuity path \eqref{couple system1}, \eqref{couple system2} is
\begin{align}
    \MM_{\vv,\ww,t} = \HH_\vv(\phi)  - \EE_\vv^{\Ric}(\phi) 
    + \EE_{\ww\cdot \ell_{\ext}}(\phi) 
    + \frac{1-t}{t}(\EE_\vv^{\theta}(\phi)-\underline{\theta} \cdot \EE_\ww(\phi)) - C_0,
\end{align}
where 
\begin{align*}
    & C_0 = \int_X \log(\vv(m_\omega))\vv(m_\omega)\omega^n, 
\end{align*}
and $\underline{\theta}$ is a contant determined by the equality:
\begin{align}
     \int_X \tr_{\vv,\omega}(\theta) \frac{\omega^{n}}{n!} =\int_X \vv(m_\omega) \theta\wedge\frac{\omega^{n-1}}{(n-1)!} + \int_X \langle \d\vv(m_\omega),m_\theta \rangle \frac{\omega^n}{n!} 
     = \underline{\theta}\cdot \int_X \ww(m_\omega) \frac{\omega^n}{n!}.
    \label{const theta_lower_bar}
\end{align}
In particular, $\underline{\theta}$ only depends on $[\omega]$ (\cite[Lemma 3]{Lah19}).

Denote $\JJ^{\Ric}_{\vv,\ww\cdot \ell_{\ext}}(\phi) =\EE_\vv^{\Ric}(\phi) - \EE_{\ww\cdot \ell_{\ext}}(\phi)$,
$\JJ^{\theta}_{\vv,\ww}(\phi) 
=\EE_\vv^{\theta}(\phi) - \underline{\theta} \cdot \EE_{\ww}(\phi)$.

Then
\begin{align}
    \MM_{\vv,\ww,t} = \HH_\vv(\phi)  - \JJ^{\Ric}_{\vv,\ww\cdot \ell_{\ext}}(\phi) + \frac{1-t}{t} \cdot \JJ^{\theta}_{\vv,\ww}(\phi) - C_0.
\end{align}

\subsection{Openness}
\label{Openness}

In this subsection, we show the openness of the continuity path \eqref{Chen-continuous path}, following the standard implicit function theory (\cite{Ch18}, \cite{Has19}). 
Firstly, note that as explained in the proof of Proposition 4.2 (\ref{lower bound of (v,w)-K-energy}), we can choose $K$-invariant $\theta\in[\omega]$ such that $\phi=0$ is the solution of continuity path at $t=0$.

Consider
\begin{align*}
	\cH^{4,\alpha}_{K}
	:=\{\phi\in C_K^{4,\alpha}(X)~|~\omega_\phi>0 \}
\end{align*}
and define
\begin{align*}
	\cF:\cH^{4,\alpha}_K\times[0,1]&\> C^\alpha(X) \\
	(\phi,t)&\mapsto t\left(\frac{S_\vv(\phi)}{\ww(m_\phi)}-\ell_\ext(m_\phi)\right)
	  -(1-t)\left(\frac{\tr_{\vv,\phi}(\theta)}{\ww(m_\phi)}-\underline{\theta}\right).
\end{align*}
Suppose there exists $(\phi_0,t_0)\in\cH^{4,\alpha}_K\times[0,1]$ such that $\cF(\phi_0,t_0)=0$. 
We will consider the following modified map with a normalization:
\begin{align*}
	\tcF:\cH^{4,\alpha}_{K,0}\times[0,1]&\> C^\alpha(X)_{K,0} \\
	(\phi,t)&\mapsto \cF(\phi,t)-\int_X\cF(\phi,t)\omega_{\phi_0}^n,
\end{align*}
where
\begin{align*}
	\cH^{4,\alpha}_{K,0}
	:=\{\phi\in\cH_K^{4,\alpha}~|~\int_X \phi\omega_{\phi_0}^n=0 \}.
\end{align*}
Note that $\int_X\cF(\phi,t)\ww(m_\phi)\omega_\phi^n=0$ by the definition of the constant $\underline{\theta}$ and extremal function $\ell_\ext$. Therefore if $\tcF(\phi,t)=0$, $\cF(\phi,t)$=0.
This is because $\tcF(\phi,t)=0$ implies $\cF(\phi,t)\equiv C$, then by integrating with respect to the weighted measure $\ww(m_\phi)\omega_\phi^n$, we obtain that $C$ must be zero.
Thus it suffices to apply implicit function theorem on the map $\tcF$ near $(\phi_0,t_0)$.
Denote
\begin{align*}
	\mathcal{R}(\phi):=&\frac{S_{\vv}(\phi)}{\ww(m_\phi)}-\ell_{\rm ext}(m_\phi),
	\text{ and }
\mathcal{T}(\phi):=\frac{\tr_{\vv,\phi}(\theta)}{\ww(m_\phi)}.    \\
\end{align*}
Thus 
\begin{align*}
	\cF(\phi,t)
	=t\cR(\phi)-(1-t)(\cT(\phi)-\underline{\theta}).
\end{align*}
Compute the linearization of $\tcF$:
\begin{align*}
	D_{(\phi,t)}\tcF:C^{4,\alpha}(X)_{K,0}\times\RR&\> C^\alpha(X)_{K,0} \\
	(u,s)&\mapsto D_{(\phi,t)}\cF(u,s)-\int_XD_{(\phi,t)}\cF(u,s)\omega_{\phi_0}^n,
\end{align*}
\begin{align*}
	(D_{(\phi,t)}\cF)(u,s)
	=&\lim_{\delta\>0}\frac{\cF(\phi+\delta u,t+\delta s)-\cF(\phi,t)}{\delta} \\
	=&t D_\phi\cR(u)-(1-t)D_\phi\cT(u) 
	+s(\cR(\phi)+\cT(\phi)-\underline{\theta}).
\end{align*}
Using \cite[Lemma B.1]{Lah19}, one computes
\begin{align}
\begin{split}\label{der-S-ext-v-w}
t(D_\phi\cR)(u)
=&-t\frac{(\mathcal{D}^{*}_\phi\vv(m_\phi)\mathcal{D}_\phi)u}{\ww(m_\phi)}
+t\langle\d(\cR(\phi)),\d u\rangle_\phi,\\
=&-t\frac{(\mathcal{D}^{*}_\phi\vv(m_\phi)\mathcal{D}_\phi)u}{\ww(m_\phi)}
 +\langle\d(\cF(\phi,t)),\d u\rangle_\phi
 +(1-t)\langle\d(\mathcal{T}(\phi)),\d u\rangle_\phi,
\end{split}
\end{align}
where $\mathcal{D}_\phi u:=\sqrt{2}(\nabla^\phi \d u)^-$ is the $J$-anti-invariant part of the tensor $(\nabla^\phi \d u)$, with $\nabla^\phi$
the $g_\phi$-Levi-Civita connection, and $\mathcal{D}^{*}_\phi$ is the formal adjoint of $\mathcal{D}_\phi$ with respect to $(\cdot,\cdot)_\phi$,
where
\begin{align*}
    (f,h)_\phi = \int_X f\; h \; \omega_\phi^n,
\end{align*}
for any $f,h\in C^{4,\alpha}(X)_{K,0}$.

We denote the operator $F_{\phi}:C^{4,\alpha}(X)_{K,0}\to C^{\alpha}(X)_{K,0}$ by
\begin{align*}
	F_\phi(u)
	=-D_\phi\cT(u)+ \langle\d(\mathcal{T}(\phi)),\d u\rangle_\phi.
\end{align*}
Then
\begin{align*}
	&(D_{(\phi,t)}\cF)(u,s) \\
=&-t\frac{(\mathcal{D}^{*}_\phi\vv(m_\phi)\mathcal{D}_\phi)u}{\ww(m_\phi)}
 +\langle\d(\cF(\phi,t)),\d u\rangle_\phi 
 +(1-t)F_\phi(u)+s(\cR(\phi)+\cT(\phi)-\underline{\theta}).
\end{align*}
In particular at $(\phi_0,t_0)$, we have
\begin{align*}
	(D_{(\phi_0,t_0)}\cF)(u,s)
	=-t_0\frac{(\mathcal{D}^{*}_{\phi_0}\vv(m_{\phi_0})\mathcal{D}_{\phi_0})u}{\ww(m_{\phi_0})}
  +(1-t_0)F_\phi(u)+s(\cR(\phi_0)+\cT(\phi_0)-\underline{\theta}).
\end{align*}
We denote
\begin{align*}
	\cL_{(\phi_0,t_0)}u
	:=-t_0\frac{(\mathcal{D}^{*}_{\phi_0}\vv(m_{\phi_0})\mathcal{D}_{\phi_0})u}{\ww(m_{\phi_0})}
  +(1-t_0)F_{\phi_0}(u).
\end{align*}
Thus
\begin{align*}
	D_{(\phi_0,t_0)}\tcF(u)
	=\cL_{(\phi_0,t_0)}u+s(\cR(\phi_0)+\cT(\phi_0)-\underline{\theta}).
\end{align*}
By Lemma \ref{self-adjoint}, we have that the operator $\cL_{(\phi_0,t_0)}$ is a $(\cdot,\cdot)_{\ww,\phi_0}$-self-adjoint $K$-invariant linear operator satisfying
\begin{align*}
	(\cL_{(\phi_0,t_0)}u_1,u_2)_{\ww,\phi_0}
	=&\int_X\cL_{(\phi_0,t_0)}u_1\cdot u_2\ww(m_{\phi_0})\omega_{\phi_0}^n \\
	=&-t_0\int_X\vv(m_{\phi_0})\cD_{\phi_0}u_1\cD_{\phi_0}u_2\omega_{\phi_0}^n
	-(1-t_0)\int_X\langle\theta,\d u_1\wedge \d^cu_2\rangle_{\phi_0}\vv(m_{\phi_0})\omega_{\phi_0}^{n} \\
	=&( u_1,\cL_{(\phi_0,t_0)}u_2)_{\ww,\phi_0}.
\end{align*}
When $t_0>0$, following the argument in \cite{Ch18}, we divide the remaining argument into three steps:

Step 1: show the inverse operator $\cL_{(\phi_0,t_0)}^{-1}$ is bounded, that is,
if $u\in C_K^{4,\alpha}(X)$ satisfies the equation $\cL_{(\phi_0,t_0)}u=f$ for some $f\in C_K^\alpha(X)$ and $\int_Xu\omega_{\phi_0}^n=0$,
	then there exists a constant $C$ depending on $\omega_{\phi_0}$ such that
	\begin{align*}
		\|u\|_{C^{4,\alpha}(X)}
		\leq C\|f\|_{C^\alpha(X)}.
	\end{align*}

	The proof follows from \cite[Lemma 4.3]{Ch18} by using Schauder estimates, Poincar\'e inequality. 

Step 2: Show that $D_{(\phi_0,t_0)}\tcF(\cdot,0)=\cL_{(\phi_0,t_0)}(\cdot):C^{4,\alpha}(X)_{K,0}\>C^{\alpha}(X)_{K,0}$
	is an isomorphism.

Similar to \cite[Lemma 4.4]{Ch18}, the proof is given by a standard Fredholm theorem for elliptic operators.

Step 3: 
By the standard implicit function theory, there exists $\phi(t)\in \cH^{4,\alpha}_K $ such that $\cF(\phi(t),t)=0$.

~\\
When $t_0=0$, following \cite{Has19}, we modify the operator $\cF$ as follows
\begin{align*}
	\cF_0:\cH^{4,\alpha}_{K,0}&\> C^\alpha(X)_{K,0} \\
	\phi&\mapsto \frac{S_\vv(\phi)}{\ww(m_\phi)}-\ell_\ext(m_\phi)
	  -R\left(\frac{\tr_{\vv,\phi}(\theta)}{\ww(m_\phi)}-\underline{\theta}\right)\\
   &\qquad -\int_X \frac{S_\vv(\phi)}{\ww(m_\phi)}-\ell_\ext(m_\phi)
	  -R\left(\frac{\tr_{\vv,\phi}(\theta)}{\ww(m_\phi)}-\underline{\theta}\right) \omega_{\phi_0}^n
\end{align*}
for some $R>>0$.
Its linearization is 
\begin{align}
\label{linearization_of_cF_0}
    \cL_{0,\phi_0}(u)
    :=(D_{\phi_0}\cF_0)(u)
    =-\frac{(\mathcal{D}^{*}_{\phi_0}\vv(m_{\phi_0})\mathcal{D}_{\phi_0})u}{\ww(m_{\phi_0})}
  +RF_{\phi_0}(u),
\end{align}
which is a $(\cdot,\cdot)_{\ww,\phi_0}$-self-adjoint $K$-invariant linear operator with a trivial kernal (see Lemma \ref{self-adjoint}).

By the detailed computation of the linearized operators in Appendix \ref{Linearization of opeartors}, we find that $F_{\phi_0}$ has a negative first eigenvalue.
Thus, the first eigenvalue $\lambda_{1,R}$ of $\cL_{0,\phi_0}$ is always negative for $R>0$. Indeed, let $\phi_{1,R}$ be the first eigenfunction of $\cL_{0,\phi_0}$, one has
\begin{align}
	\lambda_{1,R}
	=&\frac{(\cL_{0,\phi_0}(\phi_{1,R}),\phi_{1,R})_{\ww,\phi_0}}{(\phi_{1,R},\phi_{1,R})_{\ww,\phi_0}} \nonumber \\
	=&\frac{1}{(\phi_{1,R},\phi_{1,R})_{\ww,\phi_0}}\left(-\int_X\phi_{1,R}(\cD_{\phi_0}^*\vv(m_{\phi_0})\cD_{\phi_0})\phi_{1,R}\omega_{\phi_0}^n+R(F_{\phi_0}(\phi_{1,R}),\phi_{1,R})_{\ww,\phi_0}\right) \nonumber \\
	=&\frac{-1}{(\phi_{1,R},\phi_{1,R})_{\ww,\phi_0}}\int_X\vv(m_{\phi_0})|\cD_{\phi_0}\phi_{1,R}|^2_{\phi_0}\omega_{\phi_0}^n
	  +R\frac{(F_{\phi_0}(\phi_{1,R}),\phi_{1,R})_{\ww,\phi_0}}{(\phi_{1,R},\phi_{1,R})_{\ww,\phi_0}} \nonumber \\
	\leq& R\lambda_1(F_{\phi_0})\leq -C\cdot R <0.
	\label{lamdba_1,r}
\end{align}

At the end, the implicit function theorem argument in \cite{Has19} implies that there exists a solution $\phi\in\cH_K$ of $\cF_0(\phi)=0$ for $R$ large enough. 
We sketch the main steps of Hashimoto's argument as follows.
\begin{enumerate}[(i)]
	\item As stated in \eqref{equ_t=0}, we can find a solution $\omega$ for the continuity path at $t=0$. Perturb $\omega$ to get    
	      \begin{align*}
	      	\omega_m:=\omega+\ddc(R^{-1}\phi_1+\cdots+R^{-m}\phi_m)
	      \end{align*}	 
	      for some $\phi_1,\ldots,\phi_m\in C^\infty(X)_K$, satisfying $\cR(\omega_m)-R\cT(\omega_m)=const+R^{-m}f_{m,R}$ for a function $f_{m,R}\in C^\infty(X)_K$ with $(f_{m,R},1)_{\ww,\phi_0}=0$ which is bounded in $C^\infty(X)_K$ for all sufficiently large $R$. 
	      The proof follows from \cite[Lemma 3]{Has19}. 
          
	      Then by solving a weighted Green's function $\frac{\d^*(\vv(m_{\omega_m})\d G_{m,R})}{\ww(m_{\omega_m})}=f_{m,R}$, we can find that $\omega_m$ satisfies 
	      \begin{align}
	      \label{equ_of_(omega_j,theta_j)}
	      	\cR(\omega_m)-R\frac{\tr_{\vv,\omega_m}\theta_m}{\ww(m_{\omega_m})}=const
	      \end{align} where $\theta_m:=\theta+R^{-m-1}\ddc G_{m,R}$.
	\item Consider the Banach spaces $B_1=\Omega^{1,1}(X)_K\times L^2_{p+4}(X)_{K,0}$ and $B_2=\Omega^{1,1}(X)_K\times L^2_{p}(X)_{K,0}$, and an open set $U:=\{(\chi,\phi)\in B_1~|~\omega_{m,\phi}:=\omega_m+\ddc\phi>0\}$, where $\Omega^{1,1}(X)_K$ is the $L^2_p(X)$ Sobolev completion of  the space of $K$-invariant real $(1,1)$-forms and $L^2_p(X)_{K,0}$ is the Sobolev completion of $C^\infty(X)_K$ with $(\phi,1)_{\ww,\phi_0}=0$. We can define a map $T:U\rightarrow B_2$ by
	    \begin{align*}
	    	T(\chi,\phi)
	    	:=\left(\theta_m+\chi, \cR(\omega_{m,\phi})-R\frac{\tr_{\vv,\omega_{m,\phi}}(\theta_m+\chi)}{\ww(m_{\omega_{m,\phi}})} \right)
	    \end{align*}
	    Using \eqref{equ_of_(omega_j,theta_j)}, the linearization of $T$ at $0$ is given by
	    \begin{align*}
	    (D_0T)(\chi,u)
	    =
	    	\left(\begin{array}{cc}
	    		1 & 0 \\
	    		-R\frac{\tr_{\vv,\omega_m}}{\ww(m_{\omega_m})} & \cL_{0,\omega_m}	    	\end{array} \right)
	    	\left(\begin{array}{c}
	    		\chi \\
	    		u
	    	\end{array} \right),
	    \end{align*}
	    where $\cL_{0,\omega_m}$ is defined as in \eqref{linearization_of_cF_0}, where we substitute $\theta$ by $\theta_m$ and substitute $\omega_{\phi_0}$ by $\omega_m$.
	    
	    By Appendix \ref{Linearization of opeartors}, we have $\cL_{0,\omega_m}$ is a $(\cdot,\cdot)_{\ww,\omega_m}$-self-adjoint linear operator with trivial kernel. Then by Fredholm theory, we have that $D_0T:B_1\rightarrow B_2$ is an isomorphism with the inverse
	    \begin{align*}
	    	P
	    =
	    	\left(\begin{array}{cc}
	    		1 & 0 \\
	    		\left(\cL_{0,\omega_m}\right)^{-1}\left(\frac{-R\tr_{\vv,\omega_m}}{\ww(m_{\omega_m})}\right) & 
	    		\left(\cL_{0,\omega_m}\right)^{-1}
	    	\end{array} \right).
	 	\end{align*}
	 	By the computation of \eqref{lamdba_1,r} and Appendix \ref{Linearization of opeartors}  (see also \cite[lemma 4]{Has19}), there is a constant $C=C(\omega,\theta)>0$ such that $\lambda_{1,m}<-CR$ for all sufficiently large $R$, where $\lambda_{1,m}<0$ is the largest non-zero eigenvalue of $\cL_{0,\omega_m}$.
	 	By the above definition of $P$, one has $\|P\|_{\rm op}\leq CR^2$ (see \cite[Below of Lemma 4]{Has19}).
        \item By the definition of $T$, for $l\in\NN$  and $l\geq5$, for any $R$ sufficiently large, the higher order part $T-D_0T$ is Lipschitz with Lipschitz constant $1/(2\|P\|_{\rm op})$ in a ball centered at  $0$ of radius $\delta'=R^{-l}$.
        \item We choose $\delta=\frac{\delta'}{2\|P\|_{\rm op}}>\frac{1}{2C}R^{-l-2}$ for latter use.
        \item Note that $T(0,0)=(\theta_m,-R\underline{\theta})$. Since $\theta_m-\theta=O(R^{-m-1})$, we obtain
\begin{align*}
    \|T(0,0)-(\theta,-R\underline{\theta})\|_{L^2_p}
    \leq C R^{-m-1}
    \leq \frac{1}{2C}R^{-l-2}<\delta,
\end{align*}
for all large enough $R > 0$, by taking $m$ to be sufficiently large.
Choose $p>0$ to be sufficiently large such that $C^4(X) \hookrightarrow L^2_p(X)$. By the argument above, the assumptions of the quantitative inverse function theorem (see \cite[Theorem 2]{Has19}) are satisfied. By applying this theorem, there exists $(\chi,\phi)\in U$ such that $\theta_m+\chi=\theta$ and 
\begin{align*}
    \cR(\omega_{m,\phi})-R\frac{\tr_{\vv,\omega_{m,\phi}}(\theta)}{\ww(m_{\omega_{m,\phi}})}=R\underline{\theta}
\end{align*}
for all large enough $R>0$. The standard elliptic PDE theory shows that $\phi$ is smooth.
\end{enumerate}  
~\\

Combining the two cases of $t_0=0$ and $0<t_0<1$, we obtain the following openness property.
\begin{thm}
	Suppose $\cF(\phi_0,t_0)=0$ for some $t_0\in[0,1)$.
	Then for any $t\in [0,1)$ that is sufficiently close to  $t_0$, there exists $\phi(t)\in \cH^{4,\alpha}_K $ such that $\cF(\phi(t),t)=0$.
\end{thm}

\subsection{Closedness}
\label{Closedness}

In this subsection, we will always assume the weighted Mabuchi functional $\MM_{\vv,\ww\cdot \ell_{\ext}}$ is $\mathbb{G}$-coercive. Recall that $\GG = K_\CC$ is a reductive Lie subgroup of ${\Aut}_T(X)$ and $K$ contains a maximal torus, and $\TT$ is the center of $\GG$.

\begin{prop}
\label{inv and lower bound of (v,w)-K-energy}
\begin{enumerate}[$(1)$]
    \item 
    \label{lower bound of (v,w)-K-energy}
    The weighted Mabuchi functional $\MM_{\vv,\ww\cdot \ell_{\ext}}$ is bounded from below. 
    Furthermore, $\MM_{\vv,\ww\cdot \ell_{\ext},t}$ is coercive for $0<t<1$.
    
    \item 
    \label{inv of (v,w)-K-energy under gp}
    The weighted Mabuchi functional $\MM_{\vv,\ww\cdot \ell_{\ext}}$ is invariant under the group action $\sigma\in \GG$. 
    \end{enumerate}
\end{prop}
\begin{proof}
    \begin{enumerate}
    \item 
    The lower bound is a direct consequence of $\GG$-coercivity.
    
    For $t\in [0,1)$, by the definition of $\MM_{\vv,\ww\cdot \ell_{\ext},t}$, it suffices to show that $\JJ_{\vv,\ww}^\theta$ is coercive. 
    We take a suitable choice of $\theta:=\omega+\ddbar\psi_{0}$ such that 
	\begin{align}
	\label{equ_t=0}
		\frac{\tr_{\vv,\omega}\theta}{\ww(m_\omega)}=\underline{\theta}.
	\end{align}
	Indeed, the equation \eqref{equ_t=0} can be rewritten as
	\begin{align*}
		\frac{\d^*(\vv(m_\omega)\d\psi_0)}{\ww(m_\omega)}
		=\frac{\tr_{\vv,\omega}(\omega)}{\ww(m_\omega)}-\underline{\theta}.
	\end{align*}
	Note that $\frac{\d^*(\vv(m_\omega)\d\psi_0)}{\ww(m_\omega)}$ a second order  $(\cdot,\cdot)_{\ww,\omega}$-self-adjoint elliptic operator.
	By standard elliptic theory it follows that there exist a solution $\psi_0\in C^\infty(X)_K$ of \eqref{equ_t=0}.
    
    Then $\theta\geq \eps\omega$ for some $\eps>0$. Recall that for any $\phi\in \cH_{K,0}$,
    \begin{align*}
        \JJ^\omega_{\vv,\ww}(\phi) = (\II_{\vv} - \JJ_{\vv})(\phi) - \underline{\theta} \EE_\ww(\phi) \geq C\cdot \JJ_{\vv}(\phi) - C
    \end{align*}
    Thus we have
    \begin{align*}
        \JJ_{\vv,\ww}^\theta(\phi)
        \geq\eps\JJ_{\vv,\ww}^\omega (\phi) \geq C\epsilon\cdot \JJ_\vv(\phi) - C,
    \end{align*}
    which shows that $\JJ^\theta_{\vv,\ww}$ is coercive.
    \item 
    Let $\sigma\in \TT$, then there exists a holomorphic vector field $\xi$ which generates a one-parameter subgroup $\{\sigma(t)\}_{t\in\RR}$ such that $\sigma(0)=\id$ and $\sigma(1)=\sigma$.
    By the definition of weighted Mabuchi functional and weighted Futaki invariant (see \cite[Section 6]{Lah19}), we have
    	\begin{align*}
    	   	\frac{\d}{\d t}\MM_{\vv,\ww\cdot \ell_{\ext}}(\sigma(t)^*\phi)
    	   	=\FF_{\vv,\ww}(\xi)=a,
    	\end{align*}
    where $a$ is a constant and $\FF_{\vv,\ww}$ is the weighted Futaki invariant.
    Since $\MM_{\vv,\ww\cdot \ell_{\ext}}$ is bounded from below, then $a=0$. It implies $\MM_{\vv,\ww\cdot \ell_{\ext}}(\sigma^*\phi)=\MM_{\vv,\ww\cdot \ell_{\ext}}(\phi)$. 
    \end{enumerate}
\end{proof}

\begin{prop}
\label{solution and minimizer of twisted (v,w)-K-energy}
    Suppose $\MM_{\vv,\ww\cdot \ell_{\ext},t}$ is coercive. Then there exists a solution $\varphi_t$ to \eqref{Chen-continuous path}, which is also the minimizer of $\MM_{\vv,\ww\cdot \ell_{\ext},t}$ for $t\leq 1$. 
\end{prop}
\begin{proof}
    This corresponds to the case ${\Aut}_0(X)=0$ in \cite[Section 4.1]{CC21b}. 
    We have the higher order estimate of equations \eqref{couple system1}, \eqref{couple system2} (see Theorem \ref{C^2+gradient of F} and Corollary \ref{bootstrapping}). 
    Exactly, using the standard conitinuity method in \cite[Section 4.1]{CC21b} implies the existence of the solution to \eqref{couple system1}, \eqref{couple system2}.
\end{proof}

\begin{prop}
    There exists a sequence $\{t_i\}_{i\in\mathbb{N}}$, $t_i\to1$, satisfying
    \begin{align}
    \label{converg of twisted (v,w)-K-energy}
    \MM_{\vv,\ww\cdot\ell_\ext,t_i}(\varphi_i)
    \to \inf_{\phi\in\cE^1_{{K}}}\MM_{\vv,\ww\cdot \ell_{\ext}}(\phi)
    \end{align} 
    and 
    \begin{align}
    \label{converg of (v,w)-K-energy}
    	\MM_{\vv,\ww\cdot \ell_{\ext}}(\varphi_i)
    	\>\inf_{\phi\in\cE^1_{{K}}}\MM_{\vv,\ww\cdot \ell_{\ext}}(\phi)
    \end{align}
    where $\varphi_i = \varphi_{t_i}$ is the solution of \eqref{couple system1}, \eqref{couple system2} for $t=t_i<1$.
\end{prop}
\begin{proof}
	By Proposition \ref{solution and minimizer of twisted (v,w)-K-energy}, one has 
	$\MM_{\vv,\ww\cdot\ell_\ext,t_i}(\varphi_i)
	=\inf_{\cE^1_K}\MM_{\vv,\ww\cdot\ell_\ext,t_i}$.
	By Proposition \ref{inv and lower bound of (v,w)-K-energy} \eqref{lower bound of (v,w)-K-energy},  $\inf_{\phi\in\cE^1_{{K}}}\MM_{\vv,\ww\cdot \ell_{\ext}}(\phi)>-\infty$.
	Let $\varphi^\eps\in\cE^1_K$ such that $\MM_{\vv,\ww\cdot\ell_\ext}(\varphi^\eps)\leq\inf_{\phi\in\cE^1_{{K}}}\MM_{\vv,\ww\cdot \ell_{\ext}}(\phi)+\eps$.
	We obtain
	\begin{align}
	\label{limsup of twist (v,w)-K-energy}
		\limsup_{i\>\infty}\MM_{\vv,\ww\cdot\ell_\ext,t_i}(\varphi_i)
		\leq \limsup_{i\>\infty}\MM_{\vv,\ww\cdot\ell_\ext,t_i}(\varphi^\eps)
		=\MM_{\vv,\ww\cdot\ell_\ext}(\varphi^\eps)
		\leq\inf_{\phi\in\cE^1_{{K}}}\MM_{\vv,\ww\cdot \ell_{\ext}}(\phi)+\eps.
	\end{align}
	On the other hand,
	\begin{align*}
		\MM_{\vv,\ww\cdot\ell_\ext,t_i}(\varphi_i)
		=\MM_{\vv,\ww\cdot\ell_\ext}(\varphi_i)
		+\frac{1-t_i}{t_i} \cdot \JJ^{\theta}_{\vv,\ww}(\varphi_i)
		\geq\inf_{\phi\in\cE^1_{{K}}}\MM_{\vv,\ww\cdot \ell_{\ext}}(\phi)
		+\frac{1-t_i}{t_i}\JJ_{\vv,\ww}^\theta(0),
	\end{align*}
	where we take a suitable choice of $\theta:=\omega+\ddbar\psi_{0}$ such that 
	\begin{align*}
		\frac{\tr_{\vv,\omega}\theta}{\ww(m_\omega)}=\underline{\theta},
	\end{align*}
	of which the solution $\phi=0$ is the minimizer of $\JJ_{\vv,\ww}^\theta(\cdot)$.
	Then we obtain
	\begin{align*}
		\liminf_{i\>\infty}\MM_{\vv,\ww\cdot\ell_\ext,t_i}(\varphi_i)
		\geq\inf_{\phi\in\cE^1_{{K}}}\MM_{\vv,\ww\cdot \ell_{\ext}}(\phi).
	\end{align*}
	Together with \eqref{limsup of twist (v,w)-K-energy}, we obtain \eqref{converg of twisted (v,w)-K-energy}.
	
	The remaining part is the \eqref{converg of (v,w)-K-energy}.
	By \eqref{limsup of twist (v,w)-K-energy}, we have
	\begin{align*}
		\inf_{\phi\in\cE^1_{{K}}}\MM_{\vv,\ww\cdot \ell_{\ext}}(\phi)+\eps
		\geq& \MM_{\vv,\ww\cdot\ell_\ext}(\varphi_i)
		+\frac{1-t_i}{t_i} \cdot \JJ^{\theta}_{\vv,\ww}(\varphi_i) \\
		\geq&\MM_{\vv,\ww\cdot\ell_\ext}(\varphi_i)+\frac{1-t_i}{t_i}\JJ_{\vv,\ww}^\theta(0).
	\end{align*}
	Taking the limsup on the right hand side, we obtain \eqref{converg of (v,w)-K-energy}. 
\end{proof}

\begin{cor}
    There exists a constant $C>0$ such that $\JJ_{\vv,\mathbb{T}}(\varphi_i) < C$ for any $i\in\mathbb{N}$.
\end{cor}
\begin{proof}
    It follows from the assumption of $\GG$-coercivity of weighted $K$-energy and \eqref{converg of (v,w)-K-energy}.
\end{proof}

Let $\phi_i = \sigma_i^*\varphi_i$, such that $\JJ_\vv(\phi_i) = \JJ_{\vv, \mathbb{T}}(\varphi_i)$.

\begin{lemma}
    There exists a constant $C>0$, such that $\HH_\vv(\phi_i) < C$ for any $i\in\mathbb{N}$.
\end{lemma}
\begin{proof}
    Due to the $\GG$-invariance of weighted $K$-energy as in Proposition \ref{inv and lower bound of (v,w)-K-energy} \eqref{inv of (v,w)-K-energy under gp} and \eqref{converg of (v,w)-K-energy}, we have
    \begin{align*}
    	\sup_i \MM_{\vv,\ww\cdot \ell_{\ext}}(\phi_i)
       =\sup_i \MM_{\vv,\ww\cdot \ell_{\ext}}(\varphi_i)<\infty.
    \end{align*}  
    One has
    \begin{align*}
    	\sup_i|\JJ_{\vv,\ww\cdot\ell_\ext}^\Ric(\phi_i)|
    	\leq C|\Ric|_{\omega}\sup_i\JJ_\vv(\phi_i)<\infty.
    \end{align*}
    By the formula of weighted $K$-energy, we finish the proof.
\end{proof}

\begin{lemma}
\label{sigma-action of couple system}
    Let $u_i$ be the potential function which satisfies $\sigma_i^*\theta = \theta + \ddc u_i$, $\sup_X u_i = 0$. Then
    \begin{align}
         \vv(m_{\phi_i})\omega_{\phi_i}^n
		 =&e^{\tilde F_i}\omega^n, 
         \label{1st equ of phi_i}
		 \\
		\left(\Delta_{\phi_i}-\frac{\vv_{,\alpha}(m_{\phi_i})}{\vv(m_{\phi_i})} J\xi^\alpha\right)\tilde F_i
		 =&-\ell_{\ext}(m_{\phi_i})\frac{\ww(m_{\phi_i})}{\vv(m_{\phi_i})} 
	   +(1-\frac{1}{t})\frac{\tr_{\vv,\phi_i}(\theta_{u_i})}{\vv(m_{\phi_i})}
    \label{2nd equ of phi_i}\\
		&-(1-\frac{1}{t})\underline{\theta}\frac{\ww(m_{\phi_i})}{\vv(m_{\phi_i})} +\frac{\tr_{\vv,\phi_i}\Ric(\omega)}{\vv(m_{\phi_i})}. \nonumber 
    \end{align}
\end{lemma}	
	\begin{proof}
     Recall that $\phi_i=\sigma_i^*\varphi_i$. Then
    \begin{align*}
    	\sigma_i^*(\omega_{\varphi_i}^n)
    	=(\sigma_i^*\omega_{\varphi_i})^n
    	=\omega_{\phi_i}^n,\quad
    	\sigma_i^*(\vv(m_{\varphi_i})\omega_{\varphi_i}^n)
    	=\vv(m_{\phi_i})\omega_{\phi_i}^n
    \end{align*}
    and
    \begin{align*}
    	\sigma_i^*(e^{F_i}\omega^n)
    	=e^{F_i\circ\sigma_i}(\sigma_i^*\omega)^n.
    \end{align*}
    Taking pullback in the both sides of the first equation \eqref{couple system1},
    we obtain
    \begin{align*}
    	\vv(m_{\phi_i})\omega_{\phi_i}^n
    	=\sigma_i^*(e^{F_i}\omega^n)
    	=e^{\tilde F_i}\omega^n
    \end{align*}
    where $\tilde F_i=F_i\circ\sigma_i+\log(\frac{\sigma_i^*\omega^n}{\omega^n})$.
    
    Take $\sigma_i^*$-action on the second equation \eqref{couple system2}, one has
    \begin{align}
    \label{sigma_i-action of L.H.S.}
    	\text{L.H.S.}
    	=\sigma_i^*\left(\left(\Delta_{\varphi_i}-\frac{\vv_{,\alpha}(m_{\varphi_i})}{\vv(m_{\varphi_i})} J\xi^\alpha\right)F_i \right)
    	=\left(\Delta_{\phi_i}-\frac{\vv_{,\alpha}(m_{\phi_i})}{\vv(m_{\phi_i})} J\xi^\alpha\right)F_i\circ\sigma_i,
    \end{align}
    and
    \begin{align}
    \label{sigma_i-action of R.H.S.}
    	\text{R.H.S.}
    	=-\ell_{\ext}(m_{\phi_i})\frac{\ww(m_{\phi_i})}{\vv(m_{\phi_i})} 
    	+(1-\frac{1}{t})\frac{\tr_{\vv,\phi_i}(\sigma_i^*\theta)}{\vv(m_{\phi_i})}
    	-(1-\frac{1}{t})\underline{\theta}\frac{\ww(m_{\phi_i})}{\vv(m_{\phi_i})}
     +\frac{\tr_{\vv,\phi_i}\Ric(\sigma_i^*\omega)}{\vv(m_{\phi_i})}.
    \end{align}
    On the other hand,
    \begin{align}
    	&\tr_{\vv,\phi_i}(\Ric(\sigma_i^*\omega)-\Ric(\omega)) \nonumber \\
    	=&\vv(m_{\phi_i})\tr_{\phi_i}\left(\ddbar\log\frac{\omega^n}{\sigma_i^*\omega^n} \right)
    	+\langle\d\vv(m_{\phi_i}),m_{\Ric(\sigma_i^*\omega)}-m_{\Ric(\omega)}\rangle  \nonumber  \\
    	=&\vv(m_{\phi_i})\Delta_{\phi_i}\log\frac{\omega^n}{\sigma_i^*\omega^n}
    	+\vv_{,\alpha}(m_{\phi_i})J\xi^\alpha\left(\log\frac{\omega^n}{\sigma_i^*\omega^n}\right),
    	\label{tr_{v,phi_i}(diff of Ric)}
    \end{align}
    where we have used \eqref{Jxi log volume} for the last equality.
    Compare \eqref{sigma_i-action of L.H.S.}, \eqref{sigma_i-action of R.H.S.} and \eqref{tr_{v,phi_i}(diff of Ric)}, we derive \eqref{2nd equ of phi_i}.
	\end{proof}

\begin{prop}
\label{integrability of f_*}
     Let $f_i = \frac{1-t_i}{t_i}u_i$. Let $\alpha = \alpha (X, \omega)$ be the alpha-invariant of the K\"ahler class$[\omega]$. Then there exists $C>0$ such that $\int_X e^{-p\cdot f_i} \omega^n < C$ for any $p < \frac{t_i\alpha}{1-t_i}$.
     Moreover, we have $e^{-f_i}\>1$ in $L^p(\omega^n)$ as $t_i\>1$ for any $p<\infty$. 
\end{prop}
\begin{proof}
    Given $p>1$, suppose $t_i$ is sufficiently close to $1$ such that $p\frac{1-t_i}{t_i}<\alpha$, due to Tian's $\alpha$-invariant, then
    \begin{align*}
    	\int_Xe^{-pf_i}\omega^n
    	=\int_Xe^{-p\frac{1-t_i}{t_i}u_i}
    	<C.
    \end{align*}

    By \eqref{converg of twisted (v,w)-K-energy}, we have 
    \begin{align*}
    	\inf_{\phi\in\cE^1_{{K}}}\MM_{\vv,\ww\cdot \ell_{\ext}}(\phi)+\eps_i
    	\geq& \MM_{\vv,\ww\cdot \ell_{\ext},t_i}(\varphi_i)
    	=\MM_{\vv,\ww\cdot\ell_\ext}(\varphi_i)
		+\frac{1-t_i}{t_i} \cdot \JJ^{\theta}_{\vv,\ww}(\varphi_i)  \\
		\geq&\inf_{\phi\in\cE^1_{{K}}}\MM_{\vv,\ww\cdot \ell_{\ext}}(\phi)
		+\frac{1-t_i}{t_i}\JJ_\vv(\varphi_i).
    \end{align*}
    It follows that $(1-t_i)\JJ_\vv(\varphi_i)\>0$. In particular, $\int_X |f_i| \omega^n \rightarrow 0$.
    Then for any $\epsilon>0$, for $i$ sufficiently large, there exists $A_\epsilon\subset X$, such that $|f_i|\leq \epsilon$ on $A_\epsilon$, and $Vol_\omega(X\setminus A_\epsilon) \leq \epsilon$.
    Then 
    \begin{align*}
        \int_X |e^{-f_i} - 1|^p \omega^n &\leq C\cdot \big(\epsilon^p\cdot Vol_\omega(X) + \int_{X\setminus A_\epsilon} 1+e^{-p f_i} \omega^n \big) \\
        &\leq C \big( \epsilon + (\int_{X\setminus A_\epsilon} e^{-pqf_i}\omega^n)^{\frac{1}{q}}\cdot \epsilon^{\frac{1}{q^*}} \big) \\
        &\leq C \epsilon^{\frac{1}{q^*}}.
    \end{align*}
    Then $e^{-f_i}\>1$ in $L^p(\omega^n)$ as $t_i\>1$ for any $p<\infty$.
\end{proof}

Applying Theorem \ref{gradient of F} to equations \eqref{1st equ of phi_i}, \eqref{2nd equ of phi_i}, we obtain
\begin{prop} 
    By apriori estimates, $\phi_i\in W^{2,p}, \tilde F_i+f_i-\log(\vv(m_{\phi_i})) \in W^{1,p}$.
\end{prop}

Hence we may take a subsequence of $\phi_i$ (without relabeled), and a function $\phi_*\in W^{2,p}$ for any $p<\infty$,
and another function $F_*\in W^{1,p}$ for any $p<\infty$, such that
\begin{align}
	&\phi_i\>\phi_* \text{ in } C^{1,\alpha} \text{ for any }0<\alpha<1 \text{ and }\ddbar\phi_i\>\ddbar\phi_* \text{ weakly in }L^p, \label{W^2,p convergence of phi} \\
	&\tilde F_i-\log\vv(m_{\phi_i})+f_i\>F_* \text{ in } C^{\alpha} \text{ for any }0<\alpha<1 \text{ and } 
	\label{W^1,p convergence} \\
	&\qquad\qquad\qquad\qquad\nabla(\tilde F_i-\log\vv(m_{\phi_i})+f_i)\>\nabla F_* \text{ weakly in }L^p. \nonumber
\end{align}
It follows from \cite[Lemma 4.22]{CC21b} that
\begin{align*}
	\omega_{\phi_i}^k\>\omega_{\phi_i}^k \text{ weakly in }L^p \text{ for any }1\leq k\leq n \text{ and } p<\infty.
\end{align*} 
\begin{prop}
	Let $\phi_*$, $F_*$ be the limits as above. Then $\phi_*$ is a weak solution of the weighted cscK equation in the following sense:
	\begin{enumerate}[$(1)$]
		\item $\omega_{\phi_*}^n=e^{F_*}\omega^n$,
		\item \label{2nd equ of limit func}
		For any $u\in C^\infty(X)$, we have
		\begin{align}
		   	\int_XF_*\ddbar u\wedge\omega_{\phi_*}^{n-1}
		   	=&\int_Xu\frac{\vv_{,\alpha}(m_{\phi_*})}{\vv(m_{\phi_*})}J\xi^\alpha(F_*)\omega_{\phi_*}^n	
		   	-\int_X\log\vv(m_{\phi_*})\ddbar u\wedge\omega_{\phi_*}^{n-1} \nonumber \\ 
		   	&-\int_Xu\left(\frac{\vv_{,\alpha}(m_{\phi_*})\vv_{,\beta}(m_{\phi_*})}{\vv(m_{\phi_*})^2}\langle\xi^\alpha,\xi^\beta\rangle_{\phi_*}  
    	+\ell_{\ext}(m_{\phi_*})\frac{\ww(m_{\phi_*})}{\vv(m_{\phi_*})}
    	 \right)\omega_{\phi_*}^n \nonumber \\
    	 &+\int_Xu\Ric(\omega)\wedge\omega_{\phi_*}^{n-1}
         -\int_Xu\frac{\vv_{,\alpha}(m_{\phi_*})}{\vv(m_{\phi_*})}\Delta_\omega m_\omega^{\xi^\alpha}\omega_{\phi_*}^n.
    \label{2nd equ of limit function F_*}
	  \end{align}
	\end{enumerate}
\end{prop}
\begin{proof}
	For each $i$, we have $\omega_{\phi_i}^n=e^{\tilde F_i-\log\vv(m_{\phi_i})}\omega^n$.
	By \eqref{W^2,p convergence of phi}, we have $\omega_{\phi_i}^n\>\omega_{\phi_*}^n$ converge weakly in $L^p$ for any $p<\infty$. 
	Note that 
	$e^{\tilde F_i-\log\vv(m_{\phi_i})}\omega^n
	=e^{\tilde F_i-\log\vv(m_{\phi_i})+f_i}e^{-f_i}\omega^n
	$.
	Since $\tilde F_i-\log\vv(m_{\phi_i})+f_i$ is uniformly bounded and converges to $F_*$, then $e^{F_i-\log\vv(m_{\phi_i})+f_i}\>e^{F_*}$ in $L^p$ for any $p<\infty$. 
	On the other hand, by Proposition \ref{integrability of f_*}, one has $e^{-f_i}\>1$ in $L^p$ for any $p<\infty$. 
    Thus, $e^{\tilde F_i-\log\vv(m_{\phi_i})}\>e^{F_*}$ in $L^p$ for any $p<\infty$. 
    We obtain the first equation $\omega_{\phi_*}^n=e^{F_*}\omega^n$.
    
    To obtain the second equation, by \eqref{Lap of logv(m_phi)},  we first re-write the equation \eqref{2nd equ of phi_i} to
    \begin{align}
    	&\Delta_{\phi_i}(\tilde F_i+f_i-\log\vv(m_{\phi_i})) \nonumber \\
    	=&\Delta_{\phi_i}(\tilde F_i+f_i)
    	-\Delta_{\phi_i}\log\vv(m_{\phi_i}) \nonumber \\
    	=&\frac{\vv_{,\alpha}(m_{\phi_i})}{\vv(m_{\phi_i})}J\xi^\alpha(\tilde F_i+f_i-\log\vv(m_{\phi_i}))
    	+\frac{\vv_{,\alpha}(m_{\phi_i})}{\vv(m_{\phi_i})}J\xi^\alpha(\log\vv(m_{\phi_i}))-\Delta_{\phi_i}\log\vv(m_{\phi_i}) \\
    	&-\ell_{\ext}(m_{\phi_i})\frac{\ww(m_{\phi_i})}{\vv(m_{\phi_i})} 
	   +(1-\frac{1}{t_i})\frac{\tr_{\vv,\phi_i}(\theta)}{\vv(m_{\phi_i})}
		-(1-\frac{1}{t_i})\underline{\theta}\frac{\ww(m_{\phi_i})}{\vv(m_{\phi_i})} +\frac{\tr_{\vv,\phi_i}\Ric(\omega)}{\vv(m_{\phi_i})}. \nonumber 
    \end{align}
    Denote $W_i=\tilde F_i+f_i-\log\vv(m_{\phi_i})$.
    For any $u\in C^\infty(X)$, one has
    \begin{align*}
    	&\int_XW_i\ddbar u\wedge\omega_{\phi_i}^{n-1} \\
    	=&\int_Xu\frac{\vv_{,\alpha}(m_{\phi_i})}{\vv(m_{\phi_i})}J\xi^\alpha(W_i)\omega_{\phi_i}^n
    	-\int_X\log\vv(m_{\phi_i})\ddbar  u\wedge\omega_{\phi_i}^{n-1}
    	 \\
    	&-\int_Xu
    	\left(\frac{\vv_{,\alpha}(m_{\phi_i})\vv_{,\beta}(m_{\phi_i})}{\vv(m_{\phi_i})^2}\langle\xi^\alpha,\xi^\beta\rangle_{\phi_i} 
    	+\ell_{\ext}(m_{\phi_i})\frac{\ww(m_{\phi_i})}{\vv(m_{\phi_i})}
    	+(1-\frac{1}{t_i})\underline{\theta}\frac{\ww(m_{\phi_i})}{\vv(m_{\phi_i})} \right)\omega_{\phi_i}^n \\
    	&+\int_Xu\left((1-\frac{1}{t_i})\frac{\tr_{\vv,\phi_i}(\theta)}{\vv(m_{\phi_i})}+\frac{\tr_{\vv,\phi_i}\Ric(\omega)}{\vv(m_{\phi_i})}\right)\omega_{\phi_i}^n.
    \end{align*}
    For the left hand side, since $W_i\>F_*$ strongly in $L^p$, $\omega_{\phi_i}^{n-1}\>\omega_{\phi_*}^{n-1}$ weakly in $L^p$ for any $p<\infty$, we obtain
    \begin{align*}
    	\int_XW_i\ddbar u\wedge\omega_{\phi_i}^{n-1}
    	\>\int_XF_*\ddbar u\wedge\omega_{\phi_*}^{n-1}.
    \end{align*} 
    For the right hand side, by \eqref{W^1,p convergence}, \eqref{W^2,p convergence of phi}, $\phi_i\>\phi_*$ in $C^{1,\alpha}$ for any $0<\alpha<1$. Then following the construction in the proof of \cite[Proposition 10]{HL20}, we have
    \begin{align*}
        \frac{\vv_{,\alpha}(m_{\phi_i})\vv_{,\beta}(m_{\phi_i})}{\vv(m_{\phi_i})^2}\omega_{\phi_i}^n
        \> \frac{\vv_{,\alpha}(m_{\phi_*})\vv_{,\beta}(m_{\phi_*})}{\vv(m_{\phi_*})^2}\omega_{\phi_*}^n
    \end{align*}
    converges weakly as signed measures. 
    In addition, $J\xi^\alpha(W_i)\in L^p$ and $\langle\xi^\alpha,\xi^\beta\rangle_{\phi_i}\in L^p$ for any $p<\infty$.
    Thus,
    we have 
    \begin{align*}
    \int_X u\frac{\vv_{,\alpha}(m_{\phi_i})}{\vv(m_{\phi_i})}J\xi^\alpha(W_i)\omega_{\phi_i}^n
    \>\int_Xu\frac{\vv_{,\alpha}(m_{\phi})}{\vv(m_{\phi})}J\xi^\alpha(F_*)\omega_{\phi_*}^n	
    \end{align*}
    and
    \begin{align*}
    &-\int_Xu\left(\frac{\vv_{,\alpha}(m_{\phi_i})\vv_{,\beta}(m_{\phi_i})}{\vv(m_{\phi_i})^2}\langle\xi^\alpha,\xi^\beta\rangle_{\phi_i}  
    	+\ell_{\ext}(m_{\phi_i})\frac{\ww(m_{\phi_i})}{\vv(m_{\phi_i})}
    	+(1-\frac{1}{t_i})\underline{\theta}\frac{\ww(m_{\phi_i})}{\vv(m_{\phi_i})} \right)\omega_{\phi_i}^n \\
    	&\longrightarrow 
    	-\int_Xu\left(\frac{\vv_{,\alpha}(m_{\phi_*})\vv_{,\beta}(m_{\phi_*})}{\vv(m_{\phi_*})^2}\langle\xi^\alpha,\xi^\beta\rangle_{\phi_*}  
    	+\ell_{\ext}(m_{\phi_*})\frac{\ww(m_{\phi_*})}{\vv(m_{\phi_*})}
    	 \right)\omega_{\phi_*}^n.	
    \end{align*}
    We also have
    \begin{align*}
    &\int_Xu\left((1-\frac{1}{t_i})\frac{\tr_{\vv,\phi_i}(\theta)}{\vv(m_{\phi_i})}+\frac{\tr_{\vv,\phi_i}\Ric(\omega)}{\vv(m_{\phi_i})}\right)\omega_{\phi_i}^n \\
    =&\int_Xu(1-\frac{1}{t_i})\theta\wedge\omega_{\phi_i}^{n-1}
    +\int_Xu\Ric(\omega)\wedge\omega_{\phi_i}^{n-1}	\\
    &+\int_Xu\left((1-\frac{1}{t_i})\frac{\langle\d\vv(m_{\phi_i}),m_\theta\rangle}{\vv(m_{\phi})}-\frac{\vv_{,\alpha}(m_{\phi_i})}{\vv(m_{\phi_i})}\Delta_\omega m_\omega^{\xi^\alpha}\right)\omega_{\phi_i}^n \\
    &\longrightarrow \int_Xu\Ric(\omega)\wedge\omega_{\phi_*}^{n-1}
    -\int_Xu\frac{\vv_{,\alpha}(m_{\phi_*})}{\vv(m_{\phi_*})}\Delta_\omega m_\omega^{\xi^\alpha}\omega_{\phi_*}^n
    \end{align*}
    and
    \begin{align*}
    	\int_X\log\vv(m_{\phi_i})\ddbar u\wedge\omega_{\phi_i}^{n-1}
    	\>
    	\int_X\log\vv(m_{\phi_*})\ddbar u\wedge\omega_{\phi_*}^{n-1}.
    \end{align*}
    Thus, we obtain \eqref{2nd equ of limit func}.
\end{proof}

\begin{lemma}
\label{C^2 of omega phi_*}
	There exists a constant $C>0$ such that $C^{-1}\omega\leq\omega_{\phi_*}\leq C\omega$.
\end{lemma}
\begin{proof}
	We have obtained that $F_*\in W^{1,p}$ for any $p<\infty$, then we may find a sequence $G_k\in C^\infty(X)$, uniformly bounded, and $G_k\>F_*$ in $W^{1,p}$.
	Let $\psi_k$ be the solution of $\omega_{\psi_k}^n=e^{G_k}\omega^n$ with $\sup_X\psi_k=0$.
	By \cite[Theorem 1.1]{CH12}, for any $p<\infty$, one has
	\begin{align*}
		\sup_k\|\psi_k\|_{W^{3,p}}<\infty.
	\end{align*}
	Up to a subsequence, we have that there exists $\psi_*\in W^{3,p}$ such that $\psi_k\>\psi_*$ in $W^{2,p}$ for any finite $p$.
	It follows that $\omega_{\psi_*}^n=e^{F_*}\omega^n$.
	By uniqueness of Monge-Amp\`ere equations, we obtain that $\phi_*=\psi_*+c$ for some constant $c$.
	Hence $\omega_{\phi_*}=\omega_{\psi_*}\leq C\omega$ and $\omega_{\phi_*}=\omega_{\psi_*}\geq C^{-1}\omega$ follows from $F_*$ is bounded from below. 
\end{proof}

\begin{cor}
	$\phi_*$ is the smooth solution to the weighted cscK equation \eqref{weighted-csck}.
\end{cor}
\begin{proof}
	In the proof of Lemma \ref{C^2 of omega phi_*}, we have obtain that $\phi_*\in W^{3,p}$ for any $p<\infty$.
	Hence $\omega_{\phi_*}\in C^\alpha$ for any $0<\alpha<1$.
	By \eqref{2nd equ of limit function F_*} and Schauder estimate, we derive $F_*\in C^{2,\alpha}$ for any $0<\alpha<1$. 
	Then the higher regularity follows from bootstrap (see Corollary \ref{bootstrapping}). 
\end{proof}

\section{Appendix}
\label{Appendix}

\subsection{Apriori estimates for weighted cscK equation}
In this subsection, we provide the apriori estimates required to show that the weighted-cscK metric is smooth.

We consider the weighted cscK equations \eqref{couple system1 weighted cscK intro}, \eqref{couple system2 weighted cscK intro}.
Applying Theorem \ref{integral C^2} with  $f_*=0$ induces 
\begin{prop}
\label{integral C^2 for weighted-cscK}
	Assume $\vv$ is log-concave. Let $\phi$ be a solution of \eqref{couple system1 weighted cscK intro}, \eqref{couple system2 weighted cscK intro}, for any $p\geq1$, then
	\begin{align*}
		\|\tr\omega_\phi\|_{L^p(\omega^n)}\leq C=C(n, p, \vv, \ww, \ell_\ext, A_1, B, C_0),
	\end{align*}
	where $A_1$, $B$, $C_0$ are constant
    \begin{align*}
		\mathrm{Bisec}(\omega)\geq -B,\quad 
		\Ric(\omega)\leq A_1\omega, \quad
		\| F\|_{L^\infty}+\|\phi\|_{L^\infty}\leq C_0.
	\end{align*}
    
\end{prop}

Adapting the similar argument in the proof of Theorem \ref{gradient of F}, we obtain the gradient estimate of $F$ and $C^2$-estimate of $\phi$ for equations \eqref{couple system1 weighted cscK intro}, \eqref{couple system2 weighted cscK intro}.
\begin{thm}
\label{C^2+gradient of F}
	Assume $\vv$ is log-concave.
	Let $\phi$ be a solution of \eqref{couple system1 weighted cscK intro}, \eqref{couple system2 weighted cscK intro},  then
	\begin{align*}
		\max_X(|\nabla^\phi F|_\phi+\tr\omega_\phi)\leq C,
	\end{align*}
	where $C$ is a constant depending on $\omega$, $\vv$, $\ww$, $\|\phi\|_{C^0}$, $\|F\|_{C^0}$, $\|\tr\omega_\phi\|_{L^p}$ for some $p$.
\end{thm}

\begin{proof}
Comparing with the setting in the proof of Theorem \ref{gradient of F}, $f_*=0$.
Applying $f_*=0$ in \eqref{Lap for the new barrier modified}, we have
	\begin{align}
		&\Delta_\phi(e^{-\log\vv(m_\phi)} \cdot e^{\frac{F'}{2}}|\nabla^\phi F'|_\phi^2)
		 \nonumber \\
		 \geq&e^{\frac{F'}{2}-\log\vv(m_\phi)}|\nabla^\phi F'|_\phi^2|\nabla^\phi\log\vv(m_\phi)|_\phi^2
		 -2e^{-\log\vv(m_\phi)}\nabla^\phi\log\vv(m_\phi)\cdot_\phi\nabla^\phi(e^{\frac{F'}{2}}|\nabla^\phi F'|_\phi^2) \nonumber \\
		 &+e^{\frac{F'}{2}-\log\vv(m_\phi)}2\nabla^\phi F'\cdot_\phi\nabla^\phi\Delta_\phi F' 
         -C_{\vv,\xi}(1+(\tr\omega_\phi)^{n-1})e^{\frac{F'}{2}-\log\vv(m_\phi)}|\nabla^\phi F'|_\phi^2 \nonumber \\
		 &+\frac{|F'_{i\bar j}|^2}{(1+\phi_{i\bar i})(1+\phi_{j\bar j})}e^{\frac{F'}{2}-\log\vv(m_\phi)} 
		 +2\frac{\vv_{,\alpha}(m_\phi)}{\vv(m_\phi)}J\xi^\alpha(F')e^{\frac{F'}{2}-\log\vv(m_\phi)}|\nabla^\phi F'|_\phi^2.
		 \label{Lap for the new barrier}
	\end{align}
	
	By \cite[formula (3.6)]{CC21a},
	\begin{align*}
		\Delta_\phi\tr\omega_\phi
		\geq-C(\tr\omega_\phi)^n +\frac{\sum_i|\phi_{p\bar q i}|^2}{(1+\phi_{p\bar p})(1+\phi_{q\bar q})} +\Delta F'-R(g).
	\end{align*}
	By Young inequality,
	\begin{align*}
		|\Delta F'|
		\leq&\frac{\delta}{2}\frac{|F'_{i\bar i}|^2}{(1+\phi_{i\bar i})^2}+(2\delta)^{-1}(1+\phi_{i\bar i})^2  \\
		\leq&\frac{\delta}{2}\frac{|F'_{i\bar i}|^2}{(1+\phi_{i\bar i})^2}+n(2\delta)^{-1}(\tr\omega_\phi)^2.
	\end{align*}
	
	Note that, by the arithmetic-geometric inequality, we have the lower bound
	\begin{align*}
		\tr\omega_\phi
		\geq ne^{\frac{F'}{n}}
		\geq ne^{-\frac{\|F'\|_{C_0}}{n}}>0.
	\end{align*} 
	Then there exists a constant $C$ such that $1\leq C(\tr\omega_\phi)^{n-1}$.
	Now, by taking a suitable constant $\delta$ and \eqref{Lap for the new barrier}, one has
	\begin{align*}
		&\Delta_\phi(e^{-\log\vv(m_\phi)} \cdot e^{\frac{F'}{2}}|\nabla^\phi F'|_\phi^2 +\tr\omega_\phi) \\
		\geq&e^{\frac{F'}{2}-\log\vv(m_\phi)}|\nabla^\phi F'|_\phi^2|\nabla^\phi\log\vv(m_\phi)|_\phi^2
		 -2e^{-\log\vv(m_\phi)}\nabla^\phi\log\vv(m_\phi)\cdot_\phi\nabla^\phi(e^{\frac{F'}{2}}|\nabla^\phi F'|_\phi^2) \\
		 &+e^{\frac{F'}{2}-\log\vv(m_\phi)}2\nabla^\phi F'\cdot_\phi\nabla^\phi\Delta_\phi F' 
		 -C_{\vv,\xi}(\tr\omega_\phi)^{n-1}(e^{\frac{F'}{2}-\log\vv(m_\phi)}|\nabla^\phi F'|_\phi^2+\tr\omega_\phi) \\
		 &		 +2\frac{\vv_{,\alpha}(m_\phi)}{\vv(m_\phi)}J\xi^\alpha(F')(e^{\frac{F'}{2}-\log\vv(m_\phi)}|\nabla^\phi F'|_\phi^2+\tr\omega_\phi)
		 -2\frac{\vv_{,\alpha}(m_\phi)}{\vv(m_\phi)}J\xi^\alpha(F')\tr\omega_\phi.
	\end{align*}
	Set
	\begin{align*}
		u=e^{\frac{F'}{2}-\log\vv(m_\phi)}|\nabla^\phi F'|_\phi^2+\tr\omega_\phi+1,
	\end{align*}
	then we have
	\begin{align}
		\Delta_\phi u
		\geq&-C_{\vv,\xi}(\tr\omega_\phi)^{n-1}u +2e^{\frac{F'}{2}-\log\vv(m_\phi)}\nabla^\phi F'\cdot_\phi\nabla^\phi\Delta_\phi F' +2\frac{\vv_{,\alpha}(m_\phi)}{\vv(m_\phi)}J\xi^\alpha(F')u \nonumber \\
		&+e^{\frac{F'}{2}-\log\vv(m_\phi)}|\nabla^\phi F'|_\phi^2|\nabla^\phi\log\vv(m_\phi)|_\phi^2
		 -2e^{-\log\vv(m_\phi)}\nabla^\phi\log\vv(m_\phi)\cdot_\phi\nabla^\phi(e^{\frac{F'}{2}}|\nabla^\phi F'|_\phi^2).
		 \label{Laplace u}
	\end{align}
	For any $p>0$, 
	\begin{align*}
		\frac{1}{2p+1}\Delta_\phi(u^{2p+1})
		=u^{2p}\Delta_\phi u+2pu^{2p-1}|\nabla^\phi u|_\phi^2.
	\end{align*}
	Taking the integration, one has
	\begin{align}
	\label{integ formula initial}
		&\int_X2pu^{2p-1}|\nabla^\phi u|_\phi^2e^{2\log\vv(m_\phi)}\omega_\phi^n \nonumber \\
		=&\int_Xu^{2p}(-\Delta_\phi u)e^{2\log\vv(m_\phi)}\omega_\phi^n
		+\frac{1}{2p+1}\int_X\Delta_\phi(u^{2p+1})e^{2\log\vv(m_\phi)}\omega_\phi^n. 
	\end{align}
	On the one hand, by \eqref{integ of u^{2p+1} modified},
	\begin{align}
	\label{integ of u^{2p+1}}
		\frac{1}{2p+1}\int_X\Delta_\phi(u^{2p+1})e^{2\log\vv(m_\phi)}\omega_\phi^n 
        =&-2\int_Xu^{2p}\nabla^\phi(e^{\frac{F'}{2}}|\nabla^\phi F'|_\phi^2)\cdot_\phi\nabla^\phi\log\vv(m_\phi)e^{\log\vv(m_\phi)}\omega_\phi^n \nonumber \\
		&+2\int_Xu^{2p}e^{\frac{F'}{2}}|\nabla^\phi F'|_\phi^2|\nabla^\phi\log\vv(m_\phi)|_\phi^2e^{\log\vv(m_\phi)}\omega_\phi^n \nonumber \\
    &-2\int_Xu^{2p}\nabla^\phi\tr\omega_\phi\cdot_\phi\nabla^\phi\log\vv(m_\phi)e^{2\log\vv(m_\phi)}\omega_\phi^n. 
	\end{align}
	We estimate the third term on the right hand side of \eqref{integ of u^{2p+1}}
	\begin{align}
	\label{3rd term on integ of Lap u^{2p+1}}
		&-2\int_Xu^{2p}\nabla^\phi\tr\omega_\phi\cdot_\phi\nabla^\phi\log\vv(m_\phi)e^{2\log\vv(m_\phi)}\omega_\phi^n \nonumber \\
		=&2\int_X2pu^{2p-1}\tr\omega_\phi\nabla^\phi u\cdot_\phi\nabla^\phi\log\vv(m_\phi)e^{2\log\vv(m_\phi)}\omega_\phi^n \nonumber \\
		&+2\int_Xu^{2p}\tr\omega_\phi\Delta_\phi\log\vv(m_\phi)e^{2\log\vv(m_\phi)}\omega_\phi^n 
		+4\int_Xu^{2p}\tr\omega_\phi|\nabla^\phi\log\vv(m_\phi)|_\phi^2e^{2\log\vv(m_\phi)}\omega_\phi^n \nonumber \\
		\leq&\frac{p}{2}\int_Xu^{2p-1}|\nabla^\phi u|_\phi^2e^{2\log\vv(m_\phi)}\omega_\phi^n
		+\int_X8pu^{2p-1}|\nabla^\phi\log\vv(m_\phi)|_\phi^2(\tr\omega_\phi)^2e^{2\log\vv(m_\phi)}\omega_\phi^n \nonumber \\
		&+2\int_Xu^{2p}\tr\omega_\phi\left(-\frac{\vv_{,\alpha}(m_\phi)}{\vv(m_\phi)}J\xi^\alpha(F') +\frac{\vv_{,\alpha}(m_\phi)}{\vv(m_\phi)}\Delta_\omega(m_\omega^{\xi^\alpha})\right.   \nonumber \\
	    &\left.+\frac{\vv_{,\alpha\beta}(m_\phi)}{\vv(m_\phi)}\langle\xi^\alpha,\xi^\beta\rangle_\phi
	     -\frac{\vv_{,\alpha}(m_\phi)\vv_{,\beta}(m_\phi)}{\vv(m_\phi)^2}\langle\xi^\alpha,\xi^\beta\rangle_\phi\right)e^{2\log\vv(m_\phi)}\omega_\phi^n \nonumber \\
		&+4\int_Xu^{2p}\tr\omega_\phi|\nabla^\phi\log\vv(m_\phi)|_\phi^2e^{2\log\vv(m_\phi)}\omega_\phi^n. \nonumber \\
		\leq&\frac{p}{2}\int_Xu^{2p-1}|\nabla^\phi u|_\phi^2e^{2\log\vv(m_\phi)}\omega_\phi^n +C_{\vv,\xi}(p+1)\int_Xu^{2p+1}(\tr\omega_\phi)^{2n} e^{2\log\vv(m_\phi)}\omega_\phi^n, 
	\end{align}
	where we have used \eqref{Lap of logv(m_phi) and Jxi(F')} for the second inequality, \eqref{upper bound of xi pair},
	\begin{align*}
            |J\xi^\alpha(F')|
            \leq C|\nabla F'|
            \leq C(\tr\omega_\phi)^{\frac{1}{2}}|\nabla^\phi F'|
            \leq C(\tr\omega_\phi)^{\frac{1}{2}}C\sqrt{u} 
         \leq C(\tr\omega_\phi)^{\frac{1}{2}}u
         \leq C(\tr\omega_\phi)^{n-1}u,
        \end{align*}
    and \eqref{gradient of logv(m_phi)} for the third inequality.

	On the other hand, by \eqref{Laplace u},
	\begin{align}
		&\int_Xu^{2p}(-\Delta_\phi u)e^{2\log\vv(m_\phi)}\omega_\phi^n \nonumber \\
		=&C\int_Xu^{2p+1}(\tr\omega_\phi)^{n-1}e^{2\log\vv(m_\phi)}\omega_\phi^n
		-2\int_Xu^{2p+1}\frac{\vv_{,\alpha}(m_\phi)}{\vv(m_\phi)}J\xi^\alpha(F')e^{2\log\vv(m_\phi)}\omega_\phi^n \nonumber \\
		&-\int_Xu^{2p} e^{\frac{F'}{2}}|\nabla^\phi F'|_\phi^2|\nabla^\phi\log\vv(m_\phi)|_\phi^2e^{\log\vv(m_\phi)}\omega_\phi^n  \nonumber \\
		 &+2\int_Xu^{2p} \nabla^\phi\log\vv(m_\phi)\cdot_\phi\nabla^\phi(e^{\frac{F'}{2}}|\nabla^\phi F'|_\phi^2)e^{\log\vv(m_\phi)}\omega_\phi^n \nonumber \\
		 &-2\int_Xu^{2p} e^{\frac{F'}{2}}\nabla^\phi F'\cdot_\phi\nabla^\phi\Delta_\phi F' e^{\log\vv(m_\phi)}\omega_\phi^n.
		 \label{integ of Laplace u}
	\end{align}
    Applying the same computation as in \eqref{integ by part in CC21 modified}, \eqref{1st estimate of integ by part modified}, \eqref{2nd estimate of integ by parts modified}, \eqref{3rd estimate of integ by parts modified} and \eqref{4th estimate of integ by parts modified} implies
	\begin{align}
    \label{double-Laplace}
		&-2\int_Xu^{2p} e^{\frac{F'}{2}}\nabla^\phi F'\cdot_\phi\nabla^\phi\Delta_\phi F' e^{\log\vv(m_\phi)}\omega_\phi^n \nonumber \\
		\leq& \int_Xpu^{2p-1}|\nabla^\phi u|_\phi^2e^{2\log\vv(m_\phi)}\omega_\phi^n 
		+C(p+1)\int_X(\tr\omega_\phi)^{3(n-1)}u^{2p+1}e^{\log\vv(m_\phi)}\omega_\phi^n \nonumber \\
		&+2\int_Xu^{2p+1}\frac{\vv_{,\alpha}(m_\phi)}{\vv(m_\phi)}J\xi^\alpha(F')e^{2\log\vv(m_\phi)}\omega_\phi^n
		+\int_Xu^{2p+1}e^{2\log\vv(m_\phi)}\omega_\phi^n
	\end{align}
	Therefore, by \eqref{integ formula initial}, \eqref{integ of u^{2p+1}}, \eqref{3rd term on integ of Lap u^{2p+1}}, \eqref{integ of Laplace u} together with the above inequality, we deduce
	\begin{align*}
		\int_X p u^{2p-1}|\nabla^\phi u|_\phi^2e^{2\log\vv(m_\phi)}\omega_\phi^n 
		\leq
		C(p+1)\int_X(\tr\omega_\phi)^{3(n-1)}u^{2p+1}e^{\log\vv(m_\phi)}\omega_\phi^n.
	\end{align*}
	It follows that
	\begin{align*}
		\int_X|\nabla^\phi(u^{p+\frac{1}{2}})|_\phi^2\omega^n
		\leq\frac{C_{\vv,\xi}(p+\frac{1}{2})^2(p+1)}{p}\int_XC_\xi(\tr\omega_\phi)^{3(n-1)}u^{2p+1}\omega_\phi^n.
	\end{align*}
    Then applying the standard Moser iteration steps as in the proof of Theorem \ref{gradient of F} and $f_*=0$ in Lemma \ref{L^2 bound of gradient W squre modified}, we have $\|u\|_{C^0} \leq C$, which implies that $|\nabla^\phi F'|_\phi^2+\tr\omega_\phi\leq C$.

    Note that 
    \begin{align*}
        |\nabla^\phi F|_\phi^2
        \leq|\nabla^\phi F'|_\phi^2+|\nabla^\phi\log\vv(m_\phi)|_\phi^2
        \leq C.
    \end{align*}
    We finish the proof.
\end{proof}

Once we have the $C^2$-estimate, then, by standard bootstraping argument, we obtain the higher order estimates.
\begin{cor}
\label{bootstrapping}
	Suppose $C^{-1}\omega\leq\omega_\phi\leq C\omega$ for some constant $C>0$, then all higher derivatives can be estimated in terms of $C$.
\end{cor}
\begin{proof}
	By the assumption, then  \eqref{couple system2} is uniformly elliptic with bounded right hand side.
	It follows that $\|F\|_{C^\alpha}\leq C_1$ where $\alpha$ and $C_1$ depend on $C$.
	Then going back to \eqref{couple system1}, we conclude that $\|\phi\|_{C^{2,\alpha'}}\leq C_2$ for $\alpha'<\alpha$, thanks to Evans-Krylov estimate.
	Hence the coefficients of the elliptic operator of \eqref{couple system2} are in $C^{\alpha'}$.
	By \cite[Theorem 8.32]{GT77}, we may obtain $\|F\|_{C^{1,\alpha'}}\leq C_3$.
	
	By differentiating both sides of \eqref{couple system1}, we obtain the linear elliptic equation with $C^{\alpha'}$ coefficient
	\begin{align}
	\label{diff of couple system1}
		\Delta_\phi(\p_k\phi)-\frac{\vv_{,\alpha}(m_\phi)}{\vv(m_\phi)}\omega_\phi(\xi^\alpha,\p_k)
		=\p_kF-g_\phi^{i\bar j}(\p_kg_{i\bar j})+g^{i\bar j}(\p_kg_{i\bar j}).
	\end{align}
	By Schauder estimate, we deduce $\|\phi\|_{C^{3,\alpha'}}\leq C_4$. 
	Going back to \eqref{couple system2} again, the coefficient are in $C^\gamma$ for any $0<\gamma<1$, hence we have $\|F\|_{C^{1,\gamma}}\leq C_5$.
	Now, the coefficients of \eqref{diff of couple system1} are in $C^\gamma$ for any $0<\gamma<1$. 
	Thus we derive $\|\phi\|_{C^{3,\gamma}}\leq C_6$ for any $0<\gamma<1$.
	
	Now the second equation \eqref{couple system2} has $C^{1,\gamma}$ coefficients and bounded right hand side, the classical $L^p$ estimate (\cite[Theorem 9.11]{GT77}) implies $F\in W^{2,p}$ for any finite $p$.
	Differentiating the equation \eqref{diff of couple system1}, we obtain a linear elliptic equation in terms of the second derivatives of $\phi$, which has $C^{2,\gamma}$ coefficients and $L^p$ right hand side.
	Thus $\phi\in W^{4,p}$.
	The bootstrapping argument concludes the higher order estimates of $\phi$ and $F$.
\end{proof}

\subsection{Linearization of operators}
\label{Linearization of opeartors}
In this subsection, we present the details of the computation for the linearization of operators required in Subsection \ref{Openness}. 
The computation basically follows \cite{Has19, Lah19}.

Let $\cR = \frac{\rm Scal_\vv}{\ww} - \ell_\ext$, $\cT = \frac{\tr_{\vv,\phi}\theta}{\ww} = \frac{\vv\cdot \tr_\phi\theta + \langle d\vv, m_\theta \rangle}{\ww}$. Locally we can write $\theta = \ddc \psi$. 
Then $m_\theta^\alpha = -J\xi^\alpha(\psi)$, and $\langle \d\vv, m_\theta \rangle = -\sum_\alpha \vv_{,\alpha} J\xi^\alpha(\psi)$.
Let $\cF_0 = \cR - R \cT$ for some $R>0$ large enough.

Then
\begin{align*}
     \frac{d}{dt}|_{t=0}\cF_0 (\dot{\phi}) 
     &= \frac{-\cD^*(\vv \cD \dot{\phi})}{\ww} + \langle \partial\cR, \bar\partial\dot{\phi} \rangle_\phi - R\frac{d}{dt}\cT \\
    &= \frac{-\cD^*(\vv \cD \dot{\phi})}{\ww} + \langle \partial\cF, \bar\partial\dot{\phi} \rangle_\phi 
    -R\Big( \frac{d}{dt}\cT - \langle \partial\cT, \bar\partial\dot{\phi} \rangle_\phi \Big) .
\end{align*}

By direct computation,
\begin{align*}
    &-\frac{d}{dt}\cT = \frac{\cT}{\ww}\langle \partial\ww, \bar\partial\dot{\phi} \rangle_\phi - \frac{1}{\ww}\frac{d}{dt}\Big( \vv\cdot \tr_\phi\theta - \sum_\alpha \vv_{,\alpha} J\xi^\alpha(\psi) \Big), \\
    &-\frac{1}{\ww}\frac{d}{dt}\Big( \vv\cdot \tr_\phi\theta - \sum_\alpha \vv_{,\alpha} J\xi^\alpha(\psi) \Big) \\
    =& \frac{1}{\ww}\Big( -\langle \partial\vv, \bar\partial\dot{\phi} \rangle_\phi \cdot \tr_\phi\theta + \vv \langle \theta, \ddbar\dot{\phi} \rangle_\phi + \sum_\alpha \langle \partial \vv_{,\alpha}, \bar\partial\dot{\phi} \rangle_\phi J\xi^\alpha(\psi) \Big).
\end{align*}
Then 
\begin{align*}
    -\frac{d}{dt}\cT = \frac{\cT}{\ww}\langle\partial\ww,\bar\partial\dot{\phi}\rangle_\phi + \frac{1}{\ww}\Big( -\langle\partial\vv,\bar\partial\dot{\phi}\rangle_\phi\cdot \tr_\phi\theta + \vv\langle \theta,\ddbar\dot{\phi} \rangle_\phi + \sum_\alpha \langle \partial\vv_{,\alpha},\bar\partial\dot{\phi}\rangle_\phi J\xi^\alpha(\psi) \Big).
\end{align*}
And 
\begin{align*}
    \langle \partial\cT, \bar\partial\dot{\phi} \rangle_\phi = \frac{-\cT}{\ww}\langle \partial\ww,\bar\partial\dot{\phi}\rangle_\phi + \frac{\tr_\phi\theta}{\ww}\langle\partial\vv,\bar\partial\dot{\phi} \rangle_\phi + \frac{\vv}{\ww}\langle \partial(\tr_\phi\theta),\bar\partial\dot{\phi} \rangle_\phi - \frac{1}{\ww}\langle \partial \big(\sum_\alpha \vv_{,\alpha} J\xi^\alpha(\psi)\big), \bar\partial\dot{\phi} \rangle_\phi.
\end{align*}

Thus
\begin{align*}
    \langle \partial\cT,\bar\partial\dot{\phi} \rangle_\phi - \frac{d}{dt}\cT
    = \frac{\vv}{\ww}\langle \theta, \ddbar\dot{\phi} \rangle_\phi + \frac{\vv}{\ww}\langle \partial(\tr_\phi\theta),\bar\partial\dot{\phi} \rangle_\phi -  \frac{1}{\ww} \sum_\alpha \vv_{,\alpha}\langle \partial\big(J\xi^\alpha(\psi)\big), \bar\partial\dot{\phi} \rangle_\phi.
\end{align*}

Let $f$ be a test function. Then
\begin{align*}
    & \int_X \big( \langle \partial\cT,\bar\partial\dot{\phi} \rangle_\phi - \frac{d}{dt}\cT \big) \cdot f \cdot \ww\frac{\omega_\phi^n}{n!} \\
    =& \int_X \vv \langle \theta, \ddbar\dot{\phi} \rangle_\phi \cdot f \cdot \frac{\omega_\phi^n}{n!}
    + \int_X \vv \cdot \langle \partial (\tr_\phi\theta), \bar\partial\dot{\phi} \rangle_\phi \cdot f \cdot \frac{\omega_\phi^n}{n!} \\
    &- \int_X \sum_\alpha \vv_{,\alpha} \langle \partial\big( J\xi^\alpha(\psi) \big), \bar\partial\dot{\phi} \rangle_\phi \cdot f \cdot \frac{\omega_\phi^n}{n!},
\end{align*}
where $-\sum_\alpha \vv_{,\alpha} \langle \partial\big(J\xi^\alpha(\psi) \big), \bar\partial\dot{\phi} \rangle_\phi = \langle \theta, \partial\vv \wedge \bar\partial\dot{\phi} \rangle$.

For two $(1,1)$-forms $\alpha, \beta$, we have the formula 
$$n(n-1)\alpha\wedge \beta \wedge \omega_\phi^{n-2} = \big( \tr_\phi\alpha \cdot \tr_\phi\beta - \langle \alpha, \beta \rangle_\phi \big) \omega_\phi^n. $$
Apply this formula, we have
\begin{align*}
    & \int_X \vv \langle \theta, \ddbar\dot{\phi} \rangle_\phi \cdot f \cdot \frac{\omega_\phi^n}{n!} \\
    &= \int \vv \cdot \tr_\phi\theta \cdot \Delta_\phi(\dot{\phi}) \cdot f \frac{\omega_\phi^n}{n!} - \int_X \vv\cdot f \cdot \theta\wedge\ddbar\dot{\phi} \wedge \frac{\omega_\phi^{n-2}}{(n-2)!},
\end{align*}
Among which,
\begin{align*}
&  \int \vv \cdot \tr_\phi\theta \cdot \Delta_\phi(\dot{\phi}) \cdot f \frac{\omega_\phi^n}{n!} 
= \int_X \vv\cdot\tr_\phi\theta \cdot f \cdot \ddbar\dot{\phi} \wedge \frac{\omega_\phi^{n-1}}{(n-1)!} \\
=& -\int_X \vv\cdot \tr_\phi\theta \cdot \partial f \wedge \bar\partial \dot{\phi} \wedge \frac{\omega_\phi^{n-1}}{(n-1)!} - \int_X \tr_\phi\theta \cdot f \cdot \partial\vv \wedge \bar\partial\dot{\phi} \wedge \frac{\omega_\phi^{n-1}}{(n-1)!} \\
&- \int_X \vv \cdot f \cdot \partial(\tr_\phi\theta)\wedge \bar\partial\dot{\phi} \wedge \frac{\omega_\phi^{n-1}}{(n-1)!},
\end{align*}

and we also have
\begin{align*}
    & -\int_X \vv \cdot f \cdot (\theta\wedge \ddbar\dot{\phi}) \wedge \frac{\omega_\phi^{n-2}}{(n-2)!}\\
    =& \int_X \vv \cdot \partial f \wedge \bar\partial\dot{\phi} \wedge \theta \wedge \frac{\omega_\phi^{n-2}}{(n-2)!}
    +\int_X f\cdot (\partial\vv \wedge \bar\partial\dot{\phi}) \wedge\theta\wedge\frac{\omega_\phi^{n-2}}{(n-2)!}\\
    =& \int_X \vv \cdot \partial f \wedge \bar\partial\dot{\phi} \wedge \theta \wedge \frac{\omega_\phi^{n-2}}{(n-2)!} 
    + \int_X f \cdot \tr_\phi(\partial \vv \wedge \bar\partial\dot{\phi}) \cdot \tr_\phi \theta \frac{\omega_\phi^n}{n!} - \int_X f \langle \partial \vv \wedge \bar\partial\dot{\phi}, \theta \rangle_\phi \frac{\omega_\phi^n}{n!}.
\end{align*}

Therefore
\begin{align*}
    & \int_X \Big( \langle \partial\cT, \bar\partial\dot{\phi} \rangle_\phi - \frac{d}{dt} \cT \Big) \cdot f \cdot \ww \cdot \frac{\omega_\phi^n}{n!} \\
    =& \int_X -\vv \cdot \tr_\phi\theta \cdot \tr_\phi (\partial f \wedge \bar\partial\dot{\phi}) \frac{\omega_\phi^n}{n!} + \int_X \vv \cdot \partial f \wedge \bar\partial\dot{\phi} \wedge \theta \wedge \frac{\omega_\phi^{n-2}}{(n-2)!} \\
    =& -\int_X \vv \cdot \langle \theta, \partial f\wedge \bar\partial\dot{\phi} \rangle \frac{\omega_\phi^n}{n!}.
\end{align*}

By the computation above, we can see that 
\begin{lemma}
\label{self-adjoint}
The linearized operator
    $\frac{d}{dt} \cT - \langle \partial\cT, \bar\partial\cdot \rangle_\phi$
is a $(\cdot,\cdot)_{\ww,\phi}$- self-adjoint second order elliptic linear operator. In particular, if $\theta$ is positive, then the operator has negative first eigenvalue for $R>0$.
\end{lemma}

	\bibliographystyle{alpha}
	\bibliography{Weighted-cscK_Adv_v3}
	
\end{document}